\documentclass[a4paper,11pt]{amsart}

\usepackage{amsmath,amsthm,amssymb,amsfonts, tikz}
\usetikzlibrary{positioning, patterns}
\usetikzlibrary{shapes.geometric,arrows}
\tikzstyle{box} = [rectangle,text centered, draw=black]
\tikzstyle{arrow} = [thick, ->, >=stealth]
\usepackage{float} 
\usepackage{enumerate}
\usepackage{hyperref}
\usepackage{esvect}
\usepackage{mathrsfs}
\allowdisplaybreaks
\usepackage{color}

\definecolor{dgreen}{RGB}{0,150,0}
\definecolor{grey}{RGB}{50,50,50}

\newtheorem{thm}{\bf Theorem}[section]
\newtheorem{theorem}{Theorem}

\newtheorem{cor}[thm]{\bf Corollary}
\newtheorem{lem}[thm]{\bf Lemma}
\newtheorem{prop}[thm]{\bf Proposition}

\numberwithin{equation}{section}

\providecommand{\keywords}[1]
{
  \small	
  \textbf{\textit{Keywords---}} #1
}

\begin{document}

\title{Distances in sparse sets of large Hausdorff dimension}

\author{Malabika Pramanik and K S Senthil Raani}  
\date{\today}
\maketitle
\newcommand{\Addresses}{{
  \bigskip
  \footnotesize

  M.~Pramanik, {\textsc{Department of Mathematics, 1984 Mathematics Road, University of British Columbia, Vancouver, Canada V6T 1Z2}} \par\nopagebreak
  \textit{E-mail address}: \texttt{malabika@math.ubc.ca}

\medskip

K S Senthil Raani, {\textsc{IISER Berhampur Transit campus (Govt.~ITI Building),
Engg.~School Junction, Berhampur, Odisha, India 760010. }} \par\nopagebreak
  \textit{E-mail address}: \texttt{raani@iiserbpr.ac.in}

}}

\begin{abstract}
The distance set $\Delta(E)$ of a set $E \subseteq \mathbb R^d$ consists of all pairwise distances between points in $E$. This paper investigates distance sets of Borel subsets of $\mathbb R^d$ that are Lebesgue-null, but have  Hausdorff dimension close to $d$. Our results describe both the existence and distribution of intervals in $\Delta(E)$ for bounded $E$, and the appearance of all sufficiently large distances in unbounded sparse sets. Our contributions are fourfold. 
\vskip0.1in
\noindent First, we prove quantitative Steinhaus-type theorems for sets of large Hausdorff content. If $E \subseteq [0,1]^d$ has $s$-dimensional dyadic Hausdorff content at least $(1 - \rho)$, then $\Delta(E)$ contains a uniform interval $[a,b]\subseteq (0,1]$ whose endpoints depend only on $\rho$ and $d$. This gives the first uniform analogue of the classical Steinhaus–Piccard theory in the Lebesgue-null setting.
\vskip0.1in
\noindent Second, we obtain a quantitative refinement of the Mattila–Sjölin theorem. For any Borel set $E$ of Hausdorff dimension close to $d$, the set $\Delta(E)$ contains a union of intervals whose scales are determined by dyadic cubes on which $E$ has high $s$-density. This yields a flexible structure theorem for distances near the origin.
\vskip0.1in
\noindent Third, we derive a sufficient size condition ensuring that an unbounded sparse set contains all sufficiently large distances, extending a theorem of Bourgain (1986). We also provide examples of totally disconnected sets of near-full dimension satisfying this condition.
\vskip0.1in 
\noindent Finally, when $E$ enjoys additional geometric regularity, such as being locally uniformly $s$-dimensional or quasi-regular, we show that its distance set exhibits new analytic features. Using spectral gap methods and $L^2$ Fourier asymptotics, we obtain refined information on the distribution of distances in such sets.
\vskip0.1in
\noindent Several new examples, counterexamples, and open problems are presented.
\end{abstract}

 \tableofcontents

\section{Introduction} \label{intro-section}
\noindent Let $E, F \subseteq \mathbb R^d$ be two Borel sets, $d \geq 1$. Their distance set $\Delta(E, F)$ is by definition the collection of all distances generated by connecting points in $E$ to points in $F$: 
\[ \Delta(E, F) := \Bigl\{|x-y| : x \in E, \; y \in F\Bigr\} = \Bigl\{|z| : z \in E - F \Bigr\} \subseteq [0, \infty). \]  
Here $E - F$ is the algebraic difference set, and $| \cdot |$ denotes the Euclidean norm: 
\[|x| := (x_1^2 + \cdots + x_d^2)^{\frac{1}{2}} \quad \text{ for } \quad x = (x_1, \ldots, x_d) \in \mathbb  R^d. \] When $E = F$, the set $\Delta(E, E)$ is called the distance set of $E$, denoted $\Delta(E)$ for short. 
\vskip0.1in 
\noindent A rich branch of geometric measure theory is devoted to the study of $\Delta(E, F)$ given certain largeness assumptions on $E$ and $F$. A fundamental and representative example of this is the classical theorem of Steinhaus \cite{Steinhaus} on $\mathbb R$, and its generalization to $\mathbb R^d$ due to Piccard \cite{Piccard-1942}. The generalized version states: if $E$ and $F$ have positive Lebesgue measure, then the difference set $E - F$ contains a ball centred at the origin. As a result, the distance set $\Delta(E, F)$ also has non-empty interior, and in particular contains an interval of the form $[0, a)$ for some $a > 0$. Steinhaus's theorem has led to a vast array of subsequent work concerning $\Delta(E, F)$ and other configuration sets of a similar nature involving generalized distance functions. We refer the reader to \cite{{Falconer-1986}, {Mattila-Sjolin-99}, {Erdogan-2006}, {IMT-2012}, {GGIP-2015}, {BIT-2016}, {GIP-2017}, {Orponen-AD}, {JFraser-2018}, {Iosevich-Liu-2019}, {Keleti-Shmerkin-2019}, {Liu-2019}, {BDN-2021}, {DIOWZ-2021}, {GIOW-2020}, {Liu-2020},  {GIT-2021}, {Shmerkin-2021}, {FKY-2021}, {DOZ-2022}} and their extensive bibliography for significant milestones in the evolution of the subject. 
\vskip0.1in
\noindent Four distinct but intertwined lines of study concerning distances are central to this article:
\begin{itemize}  \label{list-of-themes} 
\item Quantitative Steinhaus-like theorems for sparse sets,
\item Intervals of distances generated by sets of large Hausdorff dimension, 
\item Sparse sets that admit all, or all sufficiently large, distances,
\item Finer structure of distance sets $\Delta(E)$ if $E$ has additional regularity.   
\end{itemize} 
The word ``sparse'' is used here in a qualitative sense. Depending on the context, it could refer to a set that is Lebesgue-null, or of strictly positive co-dimension, or of zero asymptotic density. Let us describe these four types of problems, the accompanying notions of sparsity and our results in more detail. 
\subsection{Quantitative Steinhaus-like theorems for sparse sets} \label{section: quantitative Steinhaus} 
The results of Steinhaus \cite{Steinhaus} and Piccard \cite{Piccard-1942} lead to a natural query:  if $E, F$ have positive Lebesgue measure, how large can an open set contained in $E-F$ or $\Delta(E, F)$ be? In the special case $d=1$ building on earlier ideas in \cite{{Steinhaus}, {Piccard-1942}}, Boardman \cite{Boardman-1970} obtains quantitative bounds on the length of an interval contained in $\Delta(E, F)$, depending on the Lebesgue measures of $E, F$. More precisely,  given $\alpha, \beta \in (0,1]$ the article \cite{Boardman-1970} provides estimates for    
\begin{align} 
& \mathtt a_{\alpha, \beta} = \inf \{\mathtt a(E, F) : E, F \subseteq [0,1], \; \lambda_1(E) = \alpha, \; \lambda_1(F) = \beta \},  \; \text{ where } \label{a-alpha-beta-def} \\
& \mathtt a(E, F) := \sup \bigl\{a > 0 : \Delta(E, F) \supseteq [0, a) \bigr\}, \nonumber
\end{align} 
and $\lambda_1$ denotes the Lebesgue measure on $\mathbb R$. The same argument extends to higher dimensions, and gives rise to the following statement. A proof of it is given in the appendix Section \ref{Appendix-1-large-distance-sets}.  
\begin{theorem}{\cite[Lemma II]{Boardman-1970}} \label{Boardman-Lemma} 
 Let $d \geq 1$. Given any $\rho < \frac{1}{2}$, there exists $c_{\rho} > 0$, depending only on $\rho$ and $d$, such that for all Borel sets $E, F \subseteq [0,1]^d$ with $\min\{\lambda_d(E), \lambda_d(F) \} \geq 1 -\rho$, the following inclusion holds: 
\begin{equation}  E - F \supseteq [0,c_\rho]^d; \quad \text{ as a result } \quad \Delta(E, F) \supseteq [0, c_{\rho}\sqrt{d}]. \label{Steinhaus-Boardman}  \end{equation} 
The constant $c_{\rho} \searrow 0$ as $\rho \nearrow \frac{1}{2}$. Here $\lambda_d$ denotes the $d$-dimensional Lebesgue measure.  
\end{theorem} 
\noindent The condition $\min\{\lambda_d(E), \lambda_d(F) \} \geq 1 -\rho$ implies that both $E$ and $F$ are large, since their complements have Lebesgue measure at most $\rho.$. 
\vskip0.1in
\noindent The assumption on the smallness of $\rho$ in the statement of Theorem \ref{Boardman-Lemma} cannot be removed. Indeed, for $d = 1$, it has been shown in \cite{Boardman-1970} that 
\[ \mathtt a_{\alpha, \beta} = 0 \text{ if } \alpha + \beta \leq 1, \quad \text{where $\mathtt a_{\alpha, \beta}$ is as in \eqref{a-alpha-beta-def}}. \] 
In particular for $\rho \geq \frac{1}{2}$, examples in \cite{Boardman-1970} show that one {\em{cannot}} find a positive $c_{\rho}$ with $\Delta(E) \supseteq [0, c_{\rho})$ for {\em{all}} sets $E \subseteq [0,1]$ with $\lambda_1(E) = 1 - \rho$. In other words, the interval in $E-F$ may become infinitesimally small, so that the infimum value of $c_{\rho}$ is zero. A similar statement holds in all dimensions if $\lambda_d(E) + \lambda_d(F)$ is small; an example in support of this is provided in Section \ref{Appendix-3}. \label{small-rho-needed} Unlike $d=1$, the threshold value of $\rho$ appears to be unknown for $d \geq 2$. 
\vskip0.1in
\noindent An important feature of Theorem \ref{Boardman-Lemma} is the universality of the constant $c_{\rho}$, which works for all sets $E, F \subseteq [0,1]^d$ whose Lebesgue measures exceed $1 - \rho > \frac12$. There does not appear to be a statement in the literature analogous to Theorem \ref{Boardman-Lemma} for Lebesgue-null sets. Indeed, statements of the form \eqref{Steinhaus-Boardman} are known to be false in this setting; see the discussion following Theorem \ref{mainthm-1} in Section \ref{section: MattilaSjolin}, also Corollary \ref{Steinhaus-sparse-corollary}. A contribution of this article is to provide a version of such a statement, first to the best of our knowledge, for sets in $\mathbb R^d$, $d \geq 2$, that have Hausdorff dimension less than but sufficiently close to $d$. Our main result in this direction is the following, which shows that an analogue of Theorem \ref{Boardman-Lemma} does hold under strong enough Hausdorff-content assumptions.
\begin{thm} \label{mainthm-2} 
\begin{enumerate}[(a)]
\item For $d \geq 2$, there exist dimensional constants $\bar{a}_d, \bar{b}_{d}, \rho_d, \bar{\varepsilon}_d \in (0,1)$ as follows. For Borel sets $E, F \subseteq [0,1]^d$ with $\min \bigl\{\mathcal H^s_{\infty}(E), \mathcal H^s_{\infty}(F) \bigr\} \geq 1 - \rho_d$, one has 
\begin{equation} \label{mainthm-2-conclusion} 
\Delta(E, F) := \Bigl\{|x-y| : x \in E, y \in F \Bigr\} \supseteq [\bar{a}_{d}, \bar{b}_{d}], \quad \text{ provided } s \geq d - \bar{\varepsilon}_d. 
\end{equation} \label{mainthm2-parta}
\vskip0.1in 
\item For $d \geq 2$, there exists a constant $\rho_d \in (0, 1)$ with the following property. For every $\rho \in (0, \rho_d]$, one can find $a_{\rho}, b_{\rho}, \varepsilon_{\rho}$ depending on $\rho$ such that for all Borel sets $E, F \subseteq [0,1]^d$,  
\begin{equation} \label{EF-large}
\Delta(E, F) \supseteq [a_{\rho}, b_{\rho}], \; \text{ provided } \; \min \bigl\{\mathcal H^s_{\infty}(E), \mathcal H^s_{\infty}(F) \bigr\} \geq 1 - \rho \text{ and } s \geq d - \varepsilon. 
\end{equation}  \label{mainthm2-partb} 
\noindent The parameters $a_{\rho}, b_{\rho}$ and $\varepsilon_{\rho}$ tend to 0 as $\rho \rightarrow 0$. 
\vskip0.1in
\noindent More precisely, there exists a dimensional constant $c_0 > 0$ such that for any $\kappa \in (0,1)$, the parameters $a_{\rho}, b_{\rho}, \varepsilon_{\rho}$ can be chosen to obey 
\[ a_{\rho} = \rho^{c_0}, \quad b_{\rho} = a_{\rho}^{\kappa}= \rho^{c_0 \kappa}, \quad \varepsilon_{\rho} = \frac{c_0 \rho}{\log(1/\rho)}, \]
provided $\rho$ is sufficiently small depending on $\kappa$.     
\end{enumerate} 
\end{thm}
\noindent Here $\mathcal H^s_{\infty}$ denotes the $s$-dimensional dyadic Hausdorff content; its definition has been recalled in Section \ref{preliminaries-section} for the reader's convenience. While technically not a measure (though an outer measure), it mimics the role of $\lambda_d$ in Theorem \ref{Boardman-Lemma}.  
\vskip0.1in 
\noindent {\em{Remarks:}} 
\begin{enumerate}[1.]
\item The proof of Theorem \ref{mainthm-2} appears in Section \ref{mainthm-2-proof-section}, conditional on a few important propositions stated earlier in Section \ref{Energy-and-Spectral-Gap-section}. These propositions are proved in Sections \ref{mainprop-1-section} and \ref{mainprop-2-section}. 
\vskip0.1in 
\item The size condition \eqref{EF-large} on $\mathcal H_{\infty}^s(E), \mathcal H_{\infty}^{s}(F)$ implies in particular that the Hausdorff dimensions of both $E$ and $F$ are at least $s$. The converse is, of course, not true. 
\vskip0.1in
\item Since $\mathcal H_{\infty}^d = \lambda_d$, part \eqref{mainthm2-parta} of Theorem \ref{mainthm-2}
recovers a weaker form of Theorem \ref{Boardman-Lemma} for $s=d$, i.e., it produces a uniform interval in $\Delta(E, F)$ away from zero, not containing it.  
\vskip0.1in 
\item Explicit examples of $\Delta(E)$ where $\mathcal H_{\infty}^{s}(E)$ is large have been worked out in Section \ref{examples-applications-section}. 
\vskip0.1in
\item An extensive literature exists on classes of fractal sets whose Hausdorff content and Hausdorff measure match at the critical dimension \cite{{Bandt-Graf-1992},{Farkas-Fraser-2015}, {Ma-Wu-2021}}; Theorem \ref{mainthm-2} applies to all such sets for which this quantity is large. In particular, these include self-similar sets in $[0,1]^d$ of Hausdorff dimension $s$ arbitrarily close to $d$ and Hausdorff measure (hence Hausdorff content) arbitrarily close to 1; see \cite[Corollary 2.3]{Farkas-Fraser-2015} \cite[Theorem 1.2]{Ma-Wu-2021}. Random Cantor sets in $\mathbb R^d$ \cite{{2009LabaMalabika},{2011LabaMalabika},{DSS-2011}, {2020SureshMalabika}, {Shmerkin-Suomala-2020}} provide another rich class of examples where the largeness of Hausdorff content can be tested. 
\vskip0.1in
\item Apart from its application in concrete examples, Theorem \ref{mainthm-2} is also instrumental in deducing finer properties of $\Delta(E)$ for more general sets $E$ that have large Hausdorff dimension but not necessarily large Hausdorff content; see Theorems \ref{mainthm-cor} and \ref{mainthm-3} below. 
\vskip0.1in
\item The proof yields an explicit dimension-dependent threshold value of $\rho_d$, below which Theorem \ref{mainthm-2} holds. Unlike the analysis in \cite{Boardman-1970} which provides sharp constants for $d = 1$, the value of $\rho_d$ obtained in our proof is not optimal.  
\vskip0.1in
\item \label{Remark-intervals} Superficially, Theorem \ref{mainthm-2} appears to provide a single interval contained in $\Delta(E, F)$, of the form $[a_{\rho}, b_{\rho}]$ for a given $\rho$; a moment's reflection however reveals that it establishes, in general, a union of such intervals in $\Delta(E, F)$. Indeed, \eqref{mainthm-2-conclusion} implies that $\Delta(E, F)$ contains an interval of the form $[a_{\rho}, b_{\rho}]$ for every $\rho$ obeying $\mathtt h_{s}(E, F) \geq 1 -\rho$, where
\[ \mathtt h_s(E, F) := \min \bigl\{\mathcal H^s_{\infty}(E), \mathcal H^s_{\infty}(F) \bigr\}  \]  is the minimal $s$-dimensional Hausdorff content of $E$ and $F$. Thus
\[ \Delta(E, F) \supseteq \bigcup_{\rho} \Bigl\{[a_{\rho}, b_{\rho}] : \rho \in (0, \rho_d) \text{ such that } d -s \leq \varepsilon_{\rho}, \; \rho \geq 1 - \mathtt h_s(E, F) \Bigr\}. \]
\end{enumerate}
\subsection{A quantitative Mattila-Sj$\ddot{\text{o}}$lin theorem} \label{section: MattilaSjolin}
Moving from the dyadic content perspective of Section \ref{section: quantitative Steinhaus} to classical Hausdorff dimension methods, we recall the foundational context for our quantitative results. Much of the interest and activity surrounding distance sets stems from Falconer's distance set conjecture. Originating in \cite{Falconer-1986}, the conjecture postulates that if $\dim_{\mathbb H}(E) > d/2$ for some Borel set $E \subseteq \mathbb R^d$, $d \geq 2$, then $\Delta(E)$ must have positive Lebesgue measure. The definitions of $s$-dimensional Hausdorff measure $\mathbb H^s$ and Hausdorff dimension $\dim_{\mathbb H}$ have been briefly recalled in Section \ref{preliminaries-section}.  An early result of Mattila and Sj$\ddot{\text{o}}$lin \cite{Mattila-Sjolin-99} following the work of Falconer \cite{Falconer-1986} shows that for sets of large enough Hausdorff dimension, the conjecture  is true in a stronger sense.  
\begin{theorem}{\cite{Mattila-Sjolin-99}}\label{mainthm-1} 
For $d \geq 2$, any Borel set in $\mathbb R^d$ of Hausdorff dimension strictly larger than $(d+1)/2$ has a distance set that not only has positive Lebesgue measure but in fact has non-empty interior. 
\end{theorem} 
\noindent In light of Steinhaus's result and Theorem \ref{mainthm-1}, it is tempting to wonder whether $\Delta(E)$ might contain intervals specifically of the form $[0, a)$ if the Hausdorff dimension $s = \dim_{\mathbb H}(E) > (d+1)/2$. This, however, is not true in general even for $s$ arbitrarily close to $d$, as has been  pointed out in \cite{Mattila-Sjolin-99}. This observation is based on an example given in \cite[Theorem 2.4]{Falconer-1986}; see Section \ref{examples-applications-section} for a discussion of this example. Thus an interval in $\Delta(E)$ provided by Theorem \ref{mainthm-1} could in general be of the form $(a, b)$ for some $0 < a < b$. 
\vskip0.1in
\noindent Further, if $E$ has large Hausdorff dimension but not necessarily large Hausdorff content (so that Theorem \ref{mainthm-2} no longer applies), the length $(b-a)$ could be arbitrarily small. It is therefore of interest to study  the structure of $\Delta(E)$ near the origin for general Lebesgue-null sets $E$ with large Hausdorff dimension.  A surprising consequence of Theorem \ref{mainthm-2} is that it provides such a structure theorem for general Borel sets, which in turn leads to numerous other consequences, as detailed in the sequel. 
\vskip0.1in
\noindent While Theorem \ref{mainthm-1} ensures the existence of an interval in $\Delta(E)$, its proof is qualitative and non-constructive. The proof strategy in \cite{Mattila-Sjolin-99} is as follows: if $\mathbb H^s(E) > 0$ for some $s > (d+1)/2$, then a certain non-trivial measure supported on $\Delta(E)$ admits a density function that is either smooth or H$\ddot{\text{o}}$lder continuous (depending on $s$). As a result, the support $\Delta(E)$ must contain an interval. Theorem \ref{mainthm-1} does not identify which intervals lie in $\Delta(E)$, nor does it describe the number, size, or spatial distribution of connected components of $\Delta(E)$. Finding a version of the Mattila-Sj$\ddot{\text{o}}$lin theorem that makes this quantitative phenomena explicit is a second objective of this article. Theorem \ref{mainthm-cor} is a structure theorem that specifies the behaviour of distances near zero, for sets of large, but less than full, Hausdorff dimension. A noteworthy feature is that the intervals of $\Delta(E)$ accumulating near the origin depend solely on the sizes of cubes where $E$ has large Hausdorff content. 
\begin{thm} \label{mainthm-cor} 
Given $d \geq 2$ and $\rho \in (0,1)$ sufficiently small, there exist constants $0 < a_{\rho} < b_{\rho}<1$ and $\varepsilon_{\rho} > 0$ depending only on $d$ and $\rho$ with the following property.  
\vskip0.1in
\noindent Let $E \subseteq [0,1]^d$ be any Borel set whose $s$-dimensional Hausdorff measure $\mathbb H^s(E)$ is positive for some $s > d - \varepsilon_{\rho}$. Let $\mathcal D_{\rho}(E; s)$ denote the collection of all dyadic cubes $Q$ contained in $[0,1]^d$ such that  
\begin{equation} \mathcal H_{\infty}^{s}(E \cap Q) > (1 - \rho) \ell(Q)^s,
 \label{high-density-on-cube}
\end{equation} 
where $\ell(Q)$ denotes the side-length of the cube $Q$. Then, 
 \begin{equation}
 \label{Delta-intervals} \Delta(E) \supseteq 
 \bigcup \Bigl\{\ell(Q)[a_{\rho},b_{\rho}]: \; Q \in \mathcal D_{\rho}(E;s) \Bigr\}.   
 \end{equation}  
\end{thm} 
\noindent {\em{Remarks:}} 
\begin{enumerate}[1.]
\item Theorem \ref{mainthm-cor} is proved in Section \ref{mainthm-cor-proof-section}, assuming Theorem \ref{mainthm-2}. 
\vskip0.1in
\item The constants $a_{\rho}, b_{\rho}$ and $\varepsilon_{\rho}$ in Theorem \ref{mainthm-cor} are identical to those in \eqref{EF-large} in Theorem \ref{mainthm-2}. 
\vskip0.1in 
\item The defining condition \eqref{high-density-on-cube} says that $\mathcal D_{\rho}(E;s)$ consists of dyadic cubes on which $E$ has nearly full $s$-dimensional density. This collection is non-empty for any Borel set $E \subseteq [0,1]^d$ with $\mathbb H^s(E) > 0$ and any $\rho \in (0,1)$. This has been stated and proved in Lemma \ref{density-lemma}. Thus the theorem applies non-vacuously to all sets of positive $\mathbb H^s$-measure.
\vskip0.1in 
\item Depending on the relative sizes of $\ell(Q)$, $a_{\rho}$ and $b_{\rho}$, the intervals $\ell(Q)[a_{\rho},b_{\rho}]$ given in \eqref{Delta-intervals} may overlap or remain disjoint. Explicit examples 
analyzing the contributions of $\Delta(E \cap Q)$ to $\Delta(E)$ have been worked out in Section \ref{examples-applications-section}.
\vskip0.1in
\item Theorem \ref{mainthm-cor} can be used as a tool to study the fundamental and classical ``Steinhaus property''. A set $E$ is said to possess this property if $\Delta(E)$ contains an interval of the form $[0, a)$ for some $a > 0$. It is well-known that this property is not limited exclusively to sets of positive measure, or to easy examples of sets that have non-trivial path-connected components, such as curves or surfaces. Many product Cantor sets, including the product of the standard middle-third Cantor set with itself, have this property. There are stronger results in this direction; for instance, \cite[Corollary 10]{Yavicoli-2022} establishes the existence of a set in $\mathbb R^d$ whose directional distance set contains a uniform interval containing zero in every direction. The following corollary, proved in Lemmas \ref{Quasiregular-NoInterval-Lemma} and \ref{Quasi-regular-Interval-Lemma} of Section \ref{Steinhaus-property-section}, provides new examples of this phenomenon where the Steinhaus property is attained by sets comprising Cantor blocks, where the constituent blocks lack the Steinhaus feature. This shows that the Steinhaus property can emerge from combining smaller sets that individually lack it. 
\end{enumerate}
\begin{cor} \label{Steinhaus-sparse-corollary}
Given $d \geq 2$ and $\rho \in (0,1)$, let $\varepsilon_{\rho}, b_{\rho} \in (0,1)$ be as in Theorem \ref{mainthm-cor}.  Then for every $s \in (d-\varepsilon_{\rho}, d]$, there exists an infinite collection $\mathscr{E}$ of self-similar Cantor sets in $[0,1]^d$ with the following properties. Every $\mathtt E^{\ast} \in \mathscr{E}$ has Hausdorff dimension at most $s$, 
\[ \Delta(\mathtt E^{\ast}) \not\supseteq [0, a) \text{ for any } a > 0, \text{i.e. } \mathtt E^{\ast} \text{ does not have the Steinhaus property,}  \]  
but there exists a finite union $\mathtt K^{\ast}$ of affine copies of $\mathtt E^{\ast}$ which does, namely, 
\[ \Delta(\mathtt K^{\ast}) \supseteq [0, b_{\rho}]. \]  
\end{cor} 
\subsection{Sparse unbounded sets containing all sufficiently large distances} \label{all-suff-large-dis-intro-section} 
Apart from identifying size dependencies in the structure of the distance sets and providing concrete information for specific examples,  quantitative statements like Theorems \ref{mainthm-2} and \ref{mainthm-cor} have other applications. One such application brings us to the third theme listed on page \pageref{list-of-themes}, namely, when sparse unbounded sets contain all sufficiently large distances. Local structure theorems for configuration sets can often be utilized to deduce abundance of these configurations in certain large unbounded sets. A simplified example of this phenomenon can be observed in the following easy consequence of Theorem \ref{Boardman-Lemma}; it illustrates how density at large scales forces abundance of distances in unbounded sets.
\begin{theorem}{\cite{Boardman-1970}}\label{Boardman-Corollary}
For any $d \geq 1$, let $A \subseteq \mathbb R^d$ be a set with the property
\begin{equation}  \limsup_{R \rightarrow \infty} \sup_{x \in \mathbb R^d} \frac{\lambda_d(A \cap Q(x;R))}{R^d} > \frac{1}{2}, \quad Q(x;R) := x + [0, R]^d. \label{Boardman-condition} \end{equation} 
Then $\Delta(A)$ contains all distances, i.e. $\Delta(A) = [0, \infty)$.  
\end{theorem} 
\noindent To clarify, Theorem \ref{Boardman-Lemma} provides a bound on the size of the distance set associated to a bounded set of positive Lebesgue measure; this leads to a criterion, outlined in Theorem \ref{Boardman-Corollary}, for an unbounded set to contain all distances. The proof of Theorem \ref{Boardman-Corollary} follows from Theorem \ref{Boardman-Lemma} by scaling; see Section \ref{Appendix-1-large-distance-sets}. 
\vskip0.1in 
\noindent In view of Theorem \ref{Boardman-Corollary}, it is natural to ask about properties of $\Delta(A)$ when the $\limsup$ in \eqref{Boardman-condition} is less than or equal to $\frac12$. An important landmark in this direction is the following result, obtained independently and roughly concurrently by three sets of authors, with three distinct proofs:  Bourgain \cite{1986Bourgain} using harmonic analysis; Falconer and Marstrand \cite{Falconer-Marstrand-1986} using geometric measure theory; Furstenberg, Katznelson and Weiss \cite{FKW-1990} using ergodic theory.             
\begin{theorem}[\cite{{1986Bourgain}, {Falconer-Marstrand-1986},{FKW-1990}}] \label{all-large-dist-thm}
Suppose $d \geq 2$, and that $A \subseteq \mathbb R^d$ has positive upper density, i.e., 
\begin{equation}  \limsup_{R \rightarrow \infty} \frac{\lambda_d(A \cap B(0;R))}{\lambda_d(B(0;R))} > 0, \label{Bourgain-condition} \end{equation} 
where $B(x;R)$ denotes a ball of radius $R$ centred at $x$. Then $A$ contains all sufficiently large distances; namely, there exists $R_0 = R_0(A) > 0$ with the property that for every $R \geq R_0$, one can find $x, y \in A$ with $|x-y| = R$. 
\end{theorem}
\noindent The proof in \cite{1986Bourgain}, with minor modifications, shows that the same conclusion as Theorem \ref{all-large-dist-thm} holds if \eqref{Bourgain-condition} is replaced by the weaker assumption
\begin{equation} \label{Bourgain-liminf-condition} 
\limsup_{R \rightarrow \infty} \sup_x \frac{\lambda_d(A \cap B(x;R))}{\lambda_d(B(x;R))} > 0,
\end{equation} 
thus providing a generalization of Theorem \ref{Boardman-Corollary}. A sketch of this modified proof is given in Section \ref{Appendix-1-Bourgain}. The article \cite{{Bukh-2008}} gives quantitative bounds for the density of $A$ in terms of the distances that it avoids. Even though our primary focus in this article is on distances (which are two-point configurations), Bourgain's work \cite{1986Bourgain} covers more general $k$-point non-degenerate configurations in $\mathbb R^d$ for $k \leq d$; see \cite{Lyall-Magyar} for a multi-point generalization of \cite{1986Bourgain}.  
\vskip0.1in 
\noindent Conditions \eqref{Boardman-condition}, \eqref{Bourgain-condition} and \eqref{Bourgain-liminf-condition} all involve the $d$-dimensional Lebesgue measure on large sets to ensure either all, or all sufficiently large, distances. One naturally wonders whether such conclusions might continue to hold for sets where conditions like \eqref{Bourgain-condition} or \eqref{Bourgain-liminf-condition} fail, but just barely. A possible example of marginal failure could be extremely slow decay in $R$ of the expressions occurring in \eqref{Bourgain-condition} or \eqref{Bourgain-liminf-condition}, as $R \rightarrow \infty$. A strong negative answer to this intuitive query has been obtained by Rice \cite{2020Rice}. For $d \geq 2$ and any $f:(0, \infty) \rightarrow [0,1]$ with $\lim_{R \rightarrow \infty}f(R) = 0$, the article \cite{2020Rice} gives a construction of a set $A \subseteq \mathbb R^d$ and a sequence $R_n \nearrow \infty$ such that 
\[ \frac{\lambda_d(A \cap B(0;R_n))}{ \lambda_d(B(0;R_n))} \geq f(R_n)\text{ for all $n$, }    \quad \text{ and yet } \quad \Delta(A) \cap \{R_n: n \geq 1\} = \emptyset. \]   
The same set $A$ continues to serve as a counterexample if the slow decay of the quantity in \eqref{Bourgain-condition} is replaced by that in \eqref{Bourgain-liminf-condition}. 
\vskip0.1in
\noindent The example in \cite{2020Rice} also precludes the possibility of a result like Theorem \ref{all-large-dist-thm} if the measure $\lambda_d$ in \eqref{Bourgain-condition} is replaced by $s$-dimensional dyadic Hausdorff content for $s < d$; indeed
\[ \limsup_{R \rightarrow \infty} \frac{\mathcal H_{\infty}^s(A \cap B(0;R))}{\mathcal H_{\infty}^{s}(B(0;R))} > 0 \quad \text{ for the set $A$ constructed in \cite{2020Rice}. } \] 
In other words, Rice's example is sharp in the measure-theoretic sense; any unconditional analogue of \eqref{Bourgain-condition} or \eqref{Bourgain-liminf-condition} using dyadic content $\mathcal H_{\infty}^s$ with $s < d$ is impossible. 
\vskip0.1in
\noindent It is therefore of interest to formulate a size condition that ensures all sufficiently large distances for sets $A$ where \eqref{Bourgain-condition} and \eqref{Bourgain-liminf-condition} fail, and in particular for sets $A$ of less than full dimension. At the same time, such a condition, in order to be applicable, should admit a rich class of non-trivial examples. As a consequence of Theorem \ref{mainthm-cor}, we obtain a condition that acts as a substitute for \eqref{Bourgain-condition} and \eqref{Bourgain-liminf-condition}, and is applicable to sparse sets. The corresponding result is stated in Theorem  \ref{mainthm-3} below.  
\vskip0.1in
\noindent  Let us set up the required terminology. Given any unbounded Borel set $A \subseteq \mathbb R^d$, the {\em{normalized truncations of $A$}} are defined as follows: for $x \in \mathbb R^d$, $R \geq 1$, 
\begin{align} A_R(x) &: = \mathbb T_{x,R} \bigl[A \cap Q(x;R) \bigr] \subseteq [0,1]^d \quad \text{ where } \label{def-ARx} \\
Q(x;R) &:= x + [0, R]^d \text{ and } \mathbb T_{x,R}(y) := \frac{y-x}{R},  \nonumber 
\end{align} 
so that $\mathbb T_{x,R}$ maps $Q(x;R)$ onto $[0,1]^d$. In other words, $A_R(x)$ captures the large-scale structure of 
$A$ inside a unit cube. For $\rho \in (0,1)$, let us also define a collection of {\em{high-density scales}}
\begin{equation} 
\mathscr{V} = \mathscr{V}(\rho, s; A) := \left\{R \ell \Bigl| \; \begin{aligned} &\exists R \geq 1, x \in \mathbb R^d \text{ such that } \mathbb H^s(A_R(x)) > 0 \\  &\text{and } \ell = \ell(Q) \text{ for some } Q \in \mathcal D_{\rho}(A_R(x)) \end{aligned} \right\} \subseteq (0, \infty). \label{def-V} 
\end{equation} 
Here $\mathcal D_{\rho}(E)$ refers to the collection of high-density cubes defined in \eqref{high-density-on-cube}, on which the $s$-dimensional density $\mathcal H_{\infty}^{s}(E \cap Q)/\ell(Q)^s$ of $E$ is at least $(1 - \rho)$. We say that $A$ is {\em{well-distributed with $s$-density at least $(1 - \rho)$ and growth rate at most $C_0 > 1$}} if  $\mathscr V(\rho,s;A)$ contains an infinite sequence $\{v_n: n \geq 1\}$ such that 
\begin{equation} \label{slower-than-lacunary} v_n \nearrow \infty \quad \text{ and } \quad \frac{v_{n+1}}{v_n} \leq C_0 \quad \text{ for all sufficiently large $n$.} \end{equation}   
To paraphrase, $A$ is well-distributed if its high-density scales are not too sparsely located on $\mathbb R$; the condition $v_{n+1}/v_n \leq C_0$ ensures this. Note that the collection $\mathscr V$ captures the size as well as the distribution of intervals in $\Delta(A)$ since, in view of Theorem \ref{mainthm-cor}, each scale $v \in \mathscr{V}$ corresponds to an interval of distances of size comparable to $v$. Specifically, the condition \eqref{slower-than-lacunary} implies the existence of an integer $J \geq 1$ (depending only on $C_0$; in fact choosing $2^J > C_0$ will do) such that for any lacunary sequence $\{\mathtt d_j : 1 \leq j \leq J \} \subseteq (v_0, \infty)$ with  $2\mathtt d_{j} < \mathtt d_{j+1}$, we have  
\[ \mathscr{V} \cap [\mathtt d_j, \mathtt d_{j+1}] \ne \emptyset \quad \text{ for some $j \leq J$.} \]    
The condition \eqref{slower-than-lacunary} combined with Theorem \ref{mainthm-cor} is sufficient to ensure that $A$ contains all sufficiently large distances. This is the content of the next theorem. 
\begin{thm} \label{mainthm-3}
Let $d \geq 2$, $\rho \in (0, 1)$, $0 < a_{\rho} < b_{\rho}< 1$, $\varepsilon_{\rho} > 0$ be as in Theorem \ref{mainthm-cor}. Then for $s > d - \varepsilon_{\rho}$, any well-distributed set $A \subseteq \mathbb R^d$ with $s$-density at least $(1 - \rho)$ and growth rate at most $b_{\rho}/a_{\rho}$ contains all sufficiently large distances, i.e., 
\[ \text{$ \exists M = M(A) > 0$ with $\Delta(A) \supseteq [M, \infty)$.} \]
\end{thm} 
\noindent The reader will recognize the lacunarity threshold $b_{\rho}/a_{\rho}$ arising directly from Theorem \ref{mainthm-cor}, related to the intervals of distances $\ell(Q)[a_{\rho}, b_{\rho}]$.
\vskip0.1in
\noindent Theorem \ref{mainthm-3} admits an interesting corollary summarized in Theorem \ref{mainthm-4} below. Theorem \ref{mainthm-3} assumes a certain growth rate of the set $\mathscr{V}$, which is defined in \eqref{def-V} using normalized truncations $A_R(x)$ with positive $\mathbb H^s$ measure. In contrast, Theorem \ref{mainthm-4} uses a sequence of truncations $A_R(x)$ with large dyadic Hausdorff content and controlled growth of the truncation scales $R$. This yields a more accessible size condition based on dyadic Hausdorff content alone, analogous in spirit to Theorem \ref{Boardman-Corollary}.
\begin{thm}\label{mainthm-4} 
Let $d$, $\rho$ and $\varepsilon$ be as in Theorem \ref{mainthm-2}. 
Let $A \subseteq \mathbb R^d$ be any Borel set that obeys  
  \begin{align} 
&  \lim_{n \rightarrow \infty} \sup_{x \in \mathbb R^d} \mathcal H^s_{\infty}\bigl(A_{R_n}(x)\bigr) > 1 - \rho \text{ for some exponent $s > d - \varepsilon$, and } \label{ARn-large} \\ 
&  \text{some sequence $R_n \nearrow \infty$ with } \frac{R_{n+1}}{R_n} \leq \frac{b_{\rho}}{a_{\rho}}.   \label{Rn-growth} 
  \end{align}  
Then $A$ admits all sufficiently large distances.   
\end{thm} 
\noindent {\em{Remarks:}} 
\begin{enumerate}[1.]
\item Theorems \ref{mainthm-3} and \ref{mainthm-4} have been proved in Section \ref{Suff-large-distances-section}. 
\vskip0.1in
\item Well-distributed sets abound in practice, so Theorem \ref{mainthm-4} encompasses a wide variety of sets. In Section \ref{Rice-example-counterpoint}, we give an example of a well-distributed set $A$ where the conditions \eqref{Bourgain-condition} and \eqref{Bourgain-liminf-condition} fail in a very strong sense, yet $A$ admits all sufficiently large distances by Theorem \ref{mainthm-4}; in fact we will see that $\Delta(A)$ in this example contains all distances!  
\vskip0.1in
\item In Section \ref{Sparse-set-all-distances-proof-section}, we use the Cantor sets specified in Corollary \ref{Steinhaus-sparse-corollary} to construct unbounded sets in $\mathbb R^d$ of strictly positive co-dimension with the following property: the distance set of every compact subset of $A$ has large gaps, whereas every co-compact subset of $A$ contains all distances.   
\begin{cor}\label{Sparse-all-distances-corollary} 
For $d \geq 2$ and $s$ sufficiently close to $d$, one can find an unbounded set $A \subseteq \mathbb R^d$ with $\dim_{\mathbb H}(A) \leq s$ such that for every $R > 0$ and $B(0;R) := \{ x \in \mathbb R^d : |x| \leq R\}$, 
\begin{align*}
& \Delta\bigl(A \setminus B(0;R)\bigr) = [0, \infty), \; {\text{ i.e., $A$ contains all distances near infinity,}}  \\
& \mathbb R \setminus \Delta(A \cap B(0;R)) \supseteq [a_{R}, b_{R}], \; {\text{ i.e., distances of $A$ on $B(0;R)$ miss intervals}},   \\ 
& (b_{R} - a_{R}) \nearrow \infty \text{ as } R \rightarrow \infty \; {\text{ i.e., the missed intervals are arbitrarily large.}} 
\end{align*} 
\end{cor}  
\vskip0.1in
\item The growth conditions \eqref{slower-than-lacunary} and \eqref{Rn-growth} cannot be removed. In Section \ref{uniformly-welldistributed-section} and inspired by \cite{2020Rice}, we provide an example of a set $A$ of positive Lebesgue measure but of density zero, avoiding infinitely many large distances, for which \eqref{ARn-large} holds only for rapidly growing sequences $R_n$ not obeying \eqref{Rn-growth}.  
\end{enumerate} 
\subsection{The role of Fourier asymptotics in distance sets} \label{Fourier-asymptotics-section}
While the previous sections focused on identifying intervals in $\Delta(E)$, we now turn to finer questions about the distribution of distances, especially the behaviour of isolated points or clusters within $\Delta(E)$. 
\vskip0.1in
\noindent The information provided by Theorems \ref{mainthm-cor} and \ref{mainthm-2} is universal, in the sense that the intervals in $\Delta(E)$ ensured by these results apply to any set $E$ with high enough Hausdorff dimension (Theorem \ref{mainthm-cor}) or Hausdorff content (Theorem \ref{mainthm-2}). The identifiers of these intervals, namely the sizes of high-density cubes of $E$, are in general arbitrary. However, if additional properties of $E$ are available, one may extract further information about the structure of $\Delta(E)$, addressing finer questions. For example, what is the distribution or separation of points in $\Delta(E)$?  Does the existence of a point in $\Delta(E)$ ensure that there are other points in $\Delta(E)$ nearby? It turns out that such problems can be explored via $L^2$-Fourier asymptotics of measures supported on $E$. The following theorem is a first result in this direction.
\vskip0.1in
\noindent Given a probability measure $\mu$, let us denote by 
\begin{equation} \label{def-F}
\mathbb F_{\mu}(T) := \int_{|\xi| \leq T} \bigl| \widehat{\mu}(\xi) \bigr|^2 \, d\xi
\end{equation} 
the partial $L^2$-Fourier integral of $\mu$ over a ball of radius $T$. If $E$ is Lebesgue-null, then $\mathbb F_{\mu}(T) \nearrow \infty$ as $T \rightarrow \infty$ for any probability measure supported on $E$. Conversely, if $E \subseteq [0,1]^d$ has positive Lebesgue measure, then $\{ \mathbb F_{\mu}(T) : T \geq 1\}$ is bounded for $\mu = \mathbf 1_E/\lambda_d(E)$. 
\begin{thm} \label{F-asymptotics-thm}
Let $d \geq 2$. For every choice of constants $c \in (0, \frac{d-1}{4})$, $T_0 \geq 1$ and $M \geq 4$, there exists a positive constant $\delta_M$ (depending only on these quantities) as follows. 
\vskip0.1in 
\noindent Suppose that $\mu$ is a probability measure supported on a set $E \subseteq [0,1]^d$, whose $L^2$-Fourier integral obeys the growth condition  
\begin{equation}\label{controlled-growth}  
\mathbb F_{\mu}(T_1T_2) \leq T_1^c \mathbb F_{\mu}(T_2) \quad \text{ for all } T_1, T_2 \geq T_0.  
\end{equation} 
If there exists $\delta \in (0, \delta_M)$ and a finite sequence $\{R_j: 1 \leq j \leq J\} \subseteq (T_0, \infty)$ satisfying  
\begin{equation} \label{Rj-condition} 
R_{j+1} \geq \delta^{-2} R_j \quad \text{ and } \quad \sum_{j=1}^{J} \mathbb F_{\mu}(R_j) > \Bigl(1 - \frac2M \Bigr)  \mathbb F_{\mu}\bigl(\delta^{-2} R_J \bigr),
\end{equation}  
then at least one of the points $\{\delta R_j^{-1} : 1 \leq j \leq J \}$ must lie in $\Delta(E)$. 
\end{thm}  
\noindent{\em{Remarks:}}
\begin{enumerate}[1.]
\item In essence, Theorem \ref{F-asymptotics-thm} says that if $\widehat{\mu}$ exhibits controlled growth and its $L^2$-mass is sufficiently spread across a sequence of scales, then $\Delta(E)$ must contain a point at one of the corresponding reciprocal scales.
\vskip0.1in
\item Theorem \ref{F-asymptotics-thm} is proved in Section \ref{F-asymptotics-section}. A more precise quantitative version of the theorem tracking the dependence of $c, T_0, M, \delta$ appears in Proposition \ref{F-asymptotics-prop} of that section. 
\vskip0.1in
\item Condition \eqref{controlled-growth} is a regularized sub-multiplicative condition that controls the growth of $\mathbb F_{\mu}(T)$ for large $T$. It limits how much new Fourier mass can be accumulated when the frequency window is enlarged by a multiplicative factor. It is well-known \cite{{Agmon-Hormander},{Strichartz-Asymptotics}, {Lau-Wang-1993}, {Lau-Wang-1995-correction}, {Raani-2014},{Raani-2017}} that such growth properties are related to the size and spread of $E = \text{supp}(\mu)$. Roughly speaking, slow growth of $\mathbb F_{\mu}(T)$ reflects faster decay of $\widehat{\mu}$, which in turn indicates that $E$ is more spatially spread out. From this perspective, a connection of $\mathbb F_{\mu}(T)$ with the structure of $\Delta(E)$ seems natural. 
\vskip0.1in
\item The two conditions in \eqref{Rj-condition} are in apparent tension. The first one requires the sequence $\{R_j\}$ to grow at least geometrically. The second one asserts that the partial Fourier mass at scales $\{R_j : 1 \leq j \leq J\}$ nearly saturates the corresponding mass at the enlarged scale $\delta^{-2}R_J$, forcing many of the scales $R_j$ to lie in regions where $\widehat{\mu}$ is large. 
\vskip0.1in
\noindent Sequences $\{ \mathbb F_{\mu}(R_j): 1 \leq j \leq J\}$ and therefore $\{R_j: 1 \leq j \leq J \}$ satisfying the second condition in \eqref{Rj-condition} cannot grow too fast. The increments $\mathbb F_{\mu}(R_{j+1}) - \mathbb F_{\mu}(R_j)$ must be small enough that the repeated summing of the $L^2$ mass of $\widehat{\mu}$ across these scales dominates its value at the largest scale $\delta^{-2}R_{J}$. Choosing sequences of $\{R_j: j \geq 1\}$ that are as slow-growing as possible subject to the geometric growth requirement, one can find denser clusters of points in $\Delta(E)$. A specific application is given in Section \ref{quasiregular-section}.
\vskip0.1in
\item Theorem \ref{F-asymptotics-thm} and Proposition \ref{F-asymptotics-prop} are fractal analogues of Bourgain's local Szemer\'edi-type theorem for sets of positive density \cite[Proposition 3]{1986Bourgain}, in the special case $k=2$. The latter result states that if $\lambda_d(E) \geq \varepsilon_0$, then there exists an integer $J$ depending only on $\varepsilon_0$ such that $\Delta(E)$ intersects {\em{any}} sequence $\{t_j : 1 \leq j \leq J \} \subseteq (0, 1]$ for which $t_{j+1} \leq t_j/2$. The underlying measure here is 
\begin{equation*}  \mu = \mathbf 1_E/\lambda_d(E) \quad  \text{where $E \subseteq [0,1]^d$ obeys $\lambda_d(E) \geq \varepsilon_0 > 0$.} \label{normalized-Lebesgue} 
\end{equation*} In this special case, there is an absolute dimensional constant $c_d > 0$
\[ \sup \bigl\{\mathbb F_{\mu}(T) : T \geq 1 \bigr\} \leq ||\widehat{\mu}||_2^2  \leq c_d^{-1} \varepsilon_0^{-1}, \] 
so \eqref{controlled-growth} is trivially satisfied. On the other hand, 
\[ \inf \bigl\{ \mathbb F_{\mu}(T) : T \geq 1 \bigr\} \geq c_d \quad \text{ for any probability measure $\mu$ supported on $[0,1]^d$.} \] A proof of this  last statement, based on the continuity of $\widehat{\mu}$, may be found on page \pageref{F-mu-lower-bound}, in the derivation of \eqref{F-mu-lower-bound}. Combining the upper and lower bounds on $\mathbb F_{\mu}(T)$ for the normalized Lebesgue measure $\mu$ of $E$ yields for any increasing sequence $R_j \nearrow \infty$:   
\[  \bigl[\mathbb F_{\mu}\bigl(\delta^{-2} R_J \bigr) \bigr]^{-1} \sum_{j=1}^{J} \mathbb F_{\mu}(R_j) \geq c_d \varepsilon_0^{-1} J \rightarrow \infty \quad \text{ as } J \rightarrow \infty. \] 
This establishes \eqref{Rj-condition} for all $\delta \in (0,1)$, $M \geq 4$, and an integer $J$ that depends only on $\varepsilon_0$. 
\vskip0.1in
\item Bourgain's local Szemer\'edi-type theorem \cite{1986Bourgain} encompasses not just distances (which are determined by two points), but $k$-point non-degenerate configurations in $\mathbb R^d$ for $k < d$.  It would be of interest to identify appropriate analogues of \eqref{controlled-growth} and \eqref{Rj-condition} that shed similar light on $k$-point patterns. A specific question in this direction has been recorded in Section \ref{section: Bourgain local Szemeredi}.  
\end{enumerate} 
\subsection{Distribution of distances in quasi-regular sets} \label{quasiregular-section}
In \cite{Strichartz-Asymptotics}, Strichartz connected the notions of size and distribution of a set with Fourier asymptotics of measures supported on that set. The concepts of {\em{locally uniformly $s$-dimensional measures}} and {\em{quasi-regular}} sets introduced there have proved extremely impactful in subsequent work. We recall these definitions in Section \ref{quasiregularity-section}. The relevance of these concepts in relation with Fourier asymptotics is captured in the following results from \cite{Strichartz-Asymptotics}, re-stated here in the special case $f \equiv 1$. 
\begin{theorem}\cite[Theorems 5.3 and 5.5]{Strichartz-Asymptotics} \label{Strichartz-theorem}
Let $E \subseteq [0,1]^d$ be a Borel set supporting a positive, finite measure $\vartheta$ that is locally uniformly $s$-dimensional of the form 
\begin{equation} \vartheta = \mathbb H^{s} \bigl|_E + \nu, \label{Strichartz-measure} \end{equation}  where $\mathbb H^{s}|_E$ denotes the $s$-dimensional Hausdorff measure $\mathbb H^s$ restricted to $E$, and  the measure $\nu$ is null with respect to $\mathbb H^{s}$, i.e. $\mathbb H^s(\cdot) < \infty$ implies $\nu(\cdot) = 0$. \footnote{This is slightly different from the Radon-Nikodym theorem; see \cite[Theorem 3.1]{Strichartz-Asymptotics}}
\begin{enumerate}[(a)]
\item Then there exists a constant $C > 0$ such that
\begin{equation} \label{Strichartz-upper} \limsup_{T \rightarrow \infty} T^{s-d} \mathbb F_{\vartheta}(T) \leq C \mathbb H^s(E). \end{equation} 
\vskip0.1in 
\item  If $E$ is also quasi-regular, then there exists a constant $C > 0$ such that 
\begin{equation}  \liminf_{T \rightarrow \infty} T^{s-d} \mathbb F_{\vartheta}(T) \geq C^{-1}\mathbb H^s(E). \label{Strichartz-lower}  \end{equation} 
\vskip0.1in
\item Combining \eqref{Strichartz-upper} and \eqref{Strichartz-lower} leads to the following assertion. Suppose that $E$ is a quasi-regular set supporting a locally uniformly $s$-dimensional probability measure $\mu = \vartheta/||\vartheta||$, with $\vartheta$ given by \eqref{Strichartz-measure}. Then there exists a constant $C_0 \geq 1$ depending on $E$ such that  
\begin{equation} 
C_0^{-1} T^{d-s}\leq \mathbb F_{\mu}(T) \leq C_0 T^{d-s}\quad \text{ for all $T \geq 1$.} \label{Strichartz-upper-lower} 
\end{equation} 
\end{enumerate}
\end{theorem}
\noindent Given two parameters $0 < \tau_1 \leq \tau_2 < 1$, a (not necessarily infinite) sequence $\{t_j : j \geq 1\} \subseteq (0, 1]$ will be called {\em{$(\tau_1, \tau_2)$-lacunary}} if  
\begin{equation} \label{lacunarity-defn} 
\tau_1 \leq \frac{t_{j+1}}{t_j} \leq \tau_2 \quad \text{ for all $j \geq 1$.}
\end{equation} 
If we were allowed to choose $\tau_1 = 0$, this definition would correspond to a traditional lacunary sequence, with lacunarity constant at most $\tau_2 < 1$. In general, such sequences can decay arbitrarily fast. The parameter $\tau_2$ in the definition \eqref{lacunarity-defn} imposes a minimal rate of decay on the sequence $\{t_j\}$, whereas the positive parameter $\tau_1$ prevents it from decaying too fast. In other words, $\tau_1$ controls the sparsity of the sequence. It turns out that the notion of $(\tau_1, \tau_2)$-lacunarity quantifies the accumulation of $\Delta(E)$ near the origin if $E$ is uniformly $s$-dimensional and quasi-regular. Loosely speaking, if $\tau_1 = \tau_2 = \tau < 1$ and $s$ is close to $d$, then $\Delta(E)$ intersects {\em{every}} geometric sequence $\{a\tau^j: j \geq 1\} \subseteq (0,1)$ at infinitely many points, with the intersection indices $j$ spaced roughly uniformly.  The precise statement is below.   
\begin{thm} \label{quasiregular-distance-thm}
Given any $d \geq 2$, $C_0 \geq 1$ and parameters $0 < \tau_1 \leq \tau_2 < 1$, there exist constants $\varepsilon_0 \in (0, 1)$ and $J_0 \in \mathbb N$  depending on these quantities with the following property. 
\vskip0.1in 
\noindent For $s \in (d - \varepsilon_0, d)$, let $E \subseteq [0,1]^d$ be any quasi-regular Borel set, equipped with a locally uniformly $s$-dimensional probability measure $\mu$  for which \eqref{Strichartz-upper-lower} holds with the specified constant $C_0$. Then $\Delta(E)$ intersects every $J_0$-long sequence that is $(\tau_1, \tau_2)$-lacunary.  Specifically, any sequence $\{t_j : 1 \leq j \leq J_0\} \subseteq (0,1]$ with \eqref{lacunarity-defn} must satisfy 
\begin{equation}  \# \Bigl[ \Delta(E) \cap \bigl\{t_{1}, \ldots, t_{J_0}\bigr\} \Bigr] \geq 1. \label{J-consecutive} \end{equation}  
In particular, if $\{t_j : j \geq 1\}$ is any infinite sequence that is $(\tau_1, \tau_2)$-lacunary, then every consecutive block $\{k+1, \ldots, k+J_0\}$ of $J_0$ integers contains an index $j$ with $t_j \in \Delta(E)$.   
\end{thm} 
\noindent {\em{Remarks:}} 
\begin{enumerate}[1.]
\item Theorem \ref{quasiregular-distance-thm} is proved in Section \ref{F-asymptotics-section}, where it is derived as a corollary of Proposition \ref{F-asymptotics-prop}, a quantitative version of Theorem \ref{F-asymptotics-thm}. 
\vskip0.1in 
\item If $\mu = \mathbf 1_E/\lambda_d(E)$ where $E \subseteq [0,1]^d$ is a set of positive Lebesgue measure with $\lambda_d(E) \geq \varepsilon_0$, \cite[Proposition 3]{1986Bourgain} shows that $J_0$ can be chosen to depend only on $\varepsilon_0$ and $\tau_2$, and uniform in $\tau_1$ as $\tau_1 \searrow 0$. Thus, in this case, $\Delta(E)$ intersects every lacunary sequence. 
\vskip0.1in 
\item  Theorem \ref{quasiregular-distance-thm} can be applied to a large class of self-similar fractals; see \cite[Theorem 5.8]{Strichartz-Asymptotics}. 
\vskip0.1in
\item Fourier dimension is an alternative notion of size often compared with Hausdorff dimension. Sets for which these two dimensions coincide are called Salem sets \cite{{Hambrook-1}, {Hambrook-2}, {Jarnik-1}, {Jarnik-2}, {Kahane-1}, {Kahane-2}, {Kaufman}}.  Such sets are of special interest in geometric measure theory; they often enjoy special properties not owned by their non-Salem counterparts of the same Hausdorff dimension \cite{{Korner-1},{Korner-2},{Liang-Pramanik}}. Statements involving the Hausdorff dimension and Lebesgue measure of $\Delta(E, F)$, where at least one of the sets $E$ or $F$ is Salem, may be found in \cite[Theorems 5.3, 5.6]{Mattila-1987}. Distinctive structural features of $\Delta(E, F)$ when $E$ is a Salem set are currently unknown; identification of such features remains a natural avenue of inquiry.   
\end{enumerate} 
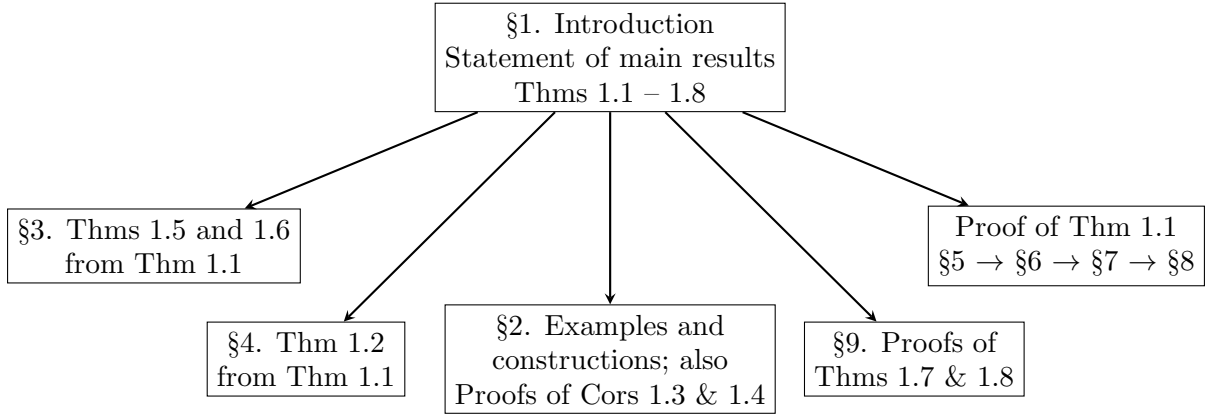
\begin{figure}
\begin{tikzpicture}[node distance=2cm]
\node (thm) [box, align=center] {\S \ref{intro-section}. Introduction \\ Statement of main
results \\ Thms \ref{mainthm-2} -- \ref{quasiregular-distance-thm}};
\node (S3) [box, below of=thm, xshift=-6cm, yshift=-0.5cm, align=center]
{\S \ref{Suff-large-distances-section}. Thms \ref{mainthm-3} and \ref{mainthm-4} \\ from Thm \ref{mainthm-2}};
\node (S4) [box, below of=thm, xshift=-4cm, yshift=-2cm, align=center] {\S \ref{Mainthm1-Proof-Section}. Thm \ref{mainthm-cor} \\ from Thm \ref{mainthm-2}};
\node (S2) [box, below of=thm, xshift=0cm, yshift=-2cm,align=center] {\S \ref{examples-applications-section}.
Examples and \\ constructions; also \\ Proofs of Cors \ref{Steinhaus-sparse-corollary} \& \ref{Sparse-all-distances-corollary}};
\node (S9) [box, below of=thm, xshift=4cm, yshift=-2cm, align=center] {\S \ref{F-asymptotics-section}. Proofs of \\ Thms \ref{F-asymptotics-thm} \& \ref{quasiregular-distance-thm} };
\node (proof) [box, below of=thm, xshift=6cm, yshift=-0.5cm,  align=center]
{Proof of Thm \ref{mainthm-2} \\ \S\ref{configuration-integral-section} $\rightarrow$ \S \ref{Energy-and-Spectral-Gap-section} $\rightarrow$ \S \ref{mainprop-1-section}
$\rightarrow$ \S \ref{mainprop-2-section}};

\draw [arrow] (thm) -- (S3);
\draw [arrow] (thm) -- (S4) ;
\draw [arrow] (thm) -- (S2) ;
\draw [arrow] (thm) -- (S9) ;
\draw [arrow] (thm) -- (proof);
\end{tikzpicture}
\caption{Article layout} \label{layout-fig}
\end{figure} 
\subsection{Proof overview} 
We summarize the main conceptual and methodological ingredients of the paper.
The conceptual novelty of the article lies in: 
\begin{itemize}
\vskip0.05in
\item Identifying Hausdorff content and Fourier asymptotics as quantitative indicators of distances; these tools act in complementary roles, content for identifying macroscopic intervals in $\Delta(E)$ (Theorems \ref{mainthm-2}, \ref{mainthm-cor}), Fourier asymptotics for microscopic distribution of individual points in $\Delta(E)$ (Theorems \ref{F-asymptotics-thm}, \ref{quasiregular-distance-thm});   
\vskip0.05in
\item Using local structure theorems for distance sets to make statements about abundance of distances in unbounded sets (Theorems \ref{mainthm-3}, \ref{mainthm-4}, Corollary \ref{Sparse-all-distances-corollary}). This connection between local dyadic structure and global abundance of distances is central to the paper’s theme.
  \end{itemize} 
The methodological novelty involves: 
\vskip0.05in 
\begin{itemize} 
\item Use of a spectral gap condition in identifying distances; specific details appear in Propositions \ref{mainprop-1}-\ref{mainprop-2'}, and their proofs in Sections \ref{mainprop-1-section} and \ref{mainprop-2-section}. Such a condition first appeared in the work of Kuca, Orponen and Sahlsten \cite{Kuca-Orponen-Sahlsten} in connection with a continuous S\'ark$\ddot{\text{o}}$zy type problem, unrelated to the distance set problem considered in this paper; see also an extension of \cite{Kuca-Orponen-Sahlsten} by Bruce and the first author \cite{Bruce-Pramanik-2023}. Although these settings involve additive patterns rather than distances, the underlying mechanism is similar. The essential and robust role played by the spectral gap condition in these widely disparate settings indicates potential for further usage.    
\vskip0.05in
\item Hands-on analysis of specific sets, including thickened Cantor-like sets and those appearing in the proof of Corollary \ref{Steinhaus-sparse-corollary}. Combined with the main theorems, they provide new examples and counter-examples of distance phenomena in the sparse setting:  All these examples appear in Section \ref{examples-applications-section}, and may be of independent interest. 
\end{itemize} 
Theorems \ref{mainthm-cor}, \ref{mainthm-3} and \ref{mainthm-4} are direct consequences of Theorem \ref{mainthm-2}; Theorem \ref{quasiregular-distance-thm} follows from Theorem \ref{F-asymptotics-thm}. The inter-dependencies of various sections are presented in Figure \ref{layout-fig}.   
\vskip0.1in
\noindent The proofs of Theorems \ref{mainthm-2} and \ref{F-asymptotics-thm} exploit a common strategy, widely used in the literature. They all involve estimation of a configuration integral, defined as $\Lambda_{\mu, \nu}(t)$ in Section \ref{configuration-integral-section}. This integral, if nonzero, signals the existence of points $x \in \text{supp}(\mu), y \in \text{supp}(\nu)$ with $|x - y| = t$. Versions of this integral have appeared in previous work \cite{{1986Bourgain}, {Mattila-Sjolin-99}}. It admits a Fourier representation, given in \eqref{mass-Lambda-mu-nu}, that is useful for estimation. The non-vanishing of the integral is determined by analysing its size on three mutually exclusive and exhaustive domains of integration - low, medium and high frequencies. The dominant contribution to the integral in \eqref{mass-Lambda-mu-nu} comes from low frequencies, where all the factors in the integrand take large values close to 1. The contribution from high frequencies is small, by virtue of convergence of the integral, and from well-known Fourier decay estimates of the spherical measure $\sigma$. 
\vskip0.1in
\noindent The main challenge is to control the contribution from intermediate frequencies, which are too far away from the origin to benefit from the size of $\widehat{\mu}(0)=1$, but too small for Fourier decay to be effective. Indeed, we do not know how to control this term for a general measure. The distinctive feature of our proof is the construction of a special measure that is supported on the set of interest and enjoys a spectral gap property. The spectral gap ensures that the intermediate-frequency contribution of the configuration integral can be suppressed, a key step in proving existence of distances. 
\subsection{Acknowledgements} This project germinated at the 16$^{\text{th}}$ Discussion Meeting in Harmonic Analysis that was held in the Indian Institute of Science, Education and Research (IISER), Bhopal in December, 2019. We gratefully acknowledge the warmth and hospitality of the meeting organizers and the host institute. Feedback from Pertti Mattila on an earlier version of the manuscript led to a more precise formulation of Corollaries \ref{Steinhaus-sparse-corollary} and \ref{Sparse-all-distances-corollary}. Suggestions from an anonymous referee improved the exposition and led to the addition of Section \ref{Open Problems Section}. Karthik Ramaseshan helped with creating the diagrams. We thank them all for their input. MP was supported in part by NSERC Discovery grant GR010263.

 \section{Examples, applications and constructions} \label{examples-applications-section} 
 \noindent  In this section we present a few examples that highlight new phenomena concerning distance sets, provide some insight into the results stated in Section \ref{intro-section} and help place them in context with other results and examples extant in the literature. We collect some necessary items along the way. A few ubiquitous terms and notations, some of which have already appeared in the statements of the theorems, are reviewed very briefly in Section \ref{preliminaries-section}; a more in-depth treatment is available in the textbooks \cite{{Mattila-Book1}, {Mattila-Book2}}. 
\subsection{Preliminaries} \label{preliminaries-section}
\subsubsection{Hausdorff dimension and content} For any Borel set $E \subseteq [0,1]^d$ and $s \in (0, d]$, the standard {\em{$s$-dimensional Hausdorff measure}} $\mathbb H^s(E)$\cite[Chapter 2]{Mattila-Book2} is defined as: 
\begin{equation}  
\mathbb H^s(E) := \lim_{r \rightarrow 0+} \inf \Bigl\{\sum_{i=1}^{\infty} \bigl[\text{diam}(E_i) \bigr]^s : E \subseteq \bigcup_{i=1}^{\infty} E_i, \; \text{diam}(E_i) \leq r \Bigr\}, 
\end{equation}
where the infimum is taken over all countable covers of $E$ consisting of sets $E_i$ with diameter at most $r$. It is well-known that the {\em{Hausdorff dimension}} of $E \subseteq \mathbb R^d$ is given by 
\begin{equation} \label{Hdim-def} \dim_{\mathbb H}(E) = \inf \bigl\{s : \mathbb H^s(E) = 0 \bigr\} = \sup \bigl\{s : \mathbb H^s(E) = \infty \bigr\}. \end{equation} 
The notions of $s$-dimensional {\em{dyadic Hausdorff measure and dyadic Hausdorff content of $E$}} are closely related to the standard Hausdorff measures: 
\begin{align}  
\mathcal H^s(E) &:=  \lim_{r \rightarrow 0+} \inf \Bigl\{\sum_{i=1}^{\infty} \ell(Q_i)^s : E \subseteq \bigcup_{i=1}^{\infty} Q_i, \; Q_i \in \mathcal Q_d, \ell(Q_i) \leq r \Bigr\}, \noindent \\ 
 \mathcal H_{\infty}^{s}(E) &:= \inf \Bigl\{ \sum_{i=1}^{\infty} \ell(Q_i)^s : E \subseteq \bigcup_{i=1}^{\infty} Q_i, \; Q_i \in \mathcal Q_d \Bigr\}. \label{dyadic-H-content-def}
 \end{align}
Here $\mathcal Q_d$ denotes the collection of all closed dyadic cubes in $\mathbb R^d$. For $Q \in \mathcal Q_d$, we write $Q = c(Q) + \ell(Q)[0,1]^d$, so that $c(Q)$ is the vertex of $Q$ which is smallest in every coordinate, and $\ell(Q)$ denotes its side-length. It follows from the definitions that for any set $E \subseteq [0,1]^d$,
\vskip0.1in
\begin{itemize}  
\item the quantities $\mathbb H^s(E), \mathcal H^s(E) \in [0, \infty]$, while $\mathcal H^s_{\infty}(E) \in [0,1]$;   
\vskip0.1in
\item these three quantities are either all zero or all positive.
\end{itemize}  
\vskip0.1in 
In particular, the definition \eqref{Hdim-def} of Hausdorff dimension remains the same if $\mathbb H^s$ is replaced by $\mathcal H^s$.  These facts will be used throughout the paper without further reference. 
\vskip0.1in
\subsubsection{Definitions associated with Fourier asymptotics} \label{quasiregularity-section} Let us also recall some relevant terminology from \cite{Strichartz-Asymptotics} that appeared in Section \ref{Fourier-asymptotics-section} and \ref{quasiregular-section}. A probability measure $\mu$ on $[0,1]^d$ is {\em{locally uniformly $s$-dimensional}} if there is a constant $C > 0$ such that
\begin{equation} \label{loc-u-dim-def} \mu \bigl(Q\bigr) \leq C \ell(Q)^s \quad \text{ for all cubes } Q \subseteq [0,1]^d.  \end{equation}  
The definition implies that $\mu$ is absolutely continuous with respect to $s$-dimensional Hausdorff measure.  We will say that an $s$-dimensional set $E \subseteq [0,1]^d$ is {\em{quasi-regular}} if there exists $\kappa > 0$ such that
\begin{equation} 
\label{Quasiregular-def} 
\liminf_{r \rightarrow 0} \frac{1}{(2r)^s} \mathbb H^{s}\bigl(E\cap B(x;r)\bigr) \geq \kappa \text{ for $\mathbb H^s$-almost every $x \in E$}. 
\end{equation}  
 \subsection{Notation}  
In the sequel, we will occasionally use $\ll, \gg, \sim$ to denote size relationships between two quantities. These are to be interpreted as follows: $A \ll B$, which is equivalent to $B \gg A$, implies the existence of a positive constant $C_d$, depending only on dimension, such that $A \leq C_d B$. We write $A \sim B$ if there exists $C_d > 1$ such that $C_d^{-1} \leq B/A \leq C_d$. 

\subsection{A building block example} \label{building-block-example-section}
Let us fix an exponent $s \in (0, d)$, and a large integer $q$. Setting $\sigma := d/s > 1$ and $\mathbb Z_{q} := \{0, 1, \ldots,  q - 1\}$, we define $\mathtt F = \mathtt F[q] \subseteq [0,1]^d$ as follows, 
\begin{equation} \label{def-building-block-F} 
\mathtt F = \mathtt F[q; \sigma] := \bigcup \left\{\frac{p}{q} + {q}^{-\sigma}[0,1]^d \; \Bigl| \; p \in \mathbb Z_{q}^d \right\}. 
\end{equation}  
For ease of exposition, we will assume that $q$ is an integer power of 2 and that $s$ is a rational number close to but less than $d$, so that $q^{\sigma}$ is an integer power of 2 as well. Let us note that $\mathcal H_{\infty}^{s}(\mathtt F) = 1$. This is likely well-known, but we include a proof in the appendix Section \ref{Building-block-Hcontent-section} for completeness. Thus the hypotheses of Theorem \ref{mainthm-2} hold with $E = F = \mathtt F$ and for all small $\rho > 0$. The goal is to explicitly work out the structure of $\Delta(\mathtt F)$, and compare with the conclusion of the theorem. As we will see later in this section, $\mathtt F$ is a fundamental building block for several examples and counterexamples concerning distance sets. 
\vskip0.1in 
\noindent To start off, let us record that 
\begin{align*} 
\mathtt F - \mathtt F &= \bigcup \left\{\frac{p-p'}{q} + q^{-\sigma}[-1,1]^d \; \Bigl| \; p,p' \in \mathbb Z_{q}^d \right\} \\
&= \bigsqcup \mathtt Q^{\ast}[p; q^{-\sigma}], \text{ where } \mathtt Q^{\ast}[p; q^{-\sigma}] := \left\{\frac{p}{q} +  q^{-\sigma}[-1,1]^d : p \in \mathbb Z_{q}^d - \mathbb Z_{q}^d \right\}.  
\end{align*}
Since the cube $\mathtt Q^{\ast}[p; q^{-\sigma}]$ is connected for every $p \in  \mathbb Z_{q}^d - \mathbb Z_{q}^d$, its image $|\mathtt Q^{\ast}[p; q^{-\sigma}]|$ under the absolute value map $x \mapsto |x|$ is also connected, hence an interval.  Points of the form $p/q + q^{-\sigma} x$ with $x = 0, \pm p/|p|$ lie in $\mathtt Q^{\ast}[p; q^{-\sigma}]$; hence their lengths, which are $|p|/q$, and $|p|/q \pm q^{-\sigma}$ respectively, are elements of $|\mathtt Q^{\ast}[p; q^{-\sigma}]|$. Connectedness of $|\mathtt Q^{\ast}[p; q^{-\sigma}]|$ implies that the interval given by the convex hull of these points is contained in $|\mathtt Q^{\ast}[p; q^{-\sigma}]|$ as well. A similar argument involving the points $p/q \pm \sqrt{2d} p/|p| \notin \mathtt Q^{\ast}[p;q^{-\sigma}]$ shows that $|p|/q \pm \sqrt{2d} q^{-\sigma} \notin |\mathtt Q^{\ast}[p; q^{-\sigma}]|$. Combining these observations, we conclude the existence of a small dimensional constant $0 < \gamma_d < 1$ such that 
\begin{align} 
&\mathtt J_r \subseteq |\mathtt Q^{\ast}[p; q^{-\sigma}]| \subseteq \mathtt I_r, \text{ where } r = |p|, \text{ and } \nonumber \\  
&\mathtt I_r := \frac{r}{q} + \gamma_d^{-1} q^{-\sigma} \bigl[-1, 1\bigr], \quad \mathtt J_r := \frac{r}{q} + \gamma_d q^{-\sigma} \bigl[-1, 1\bigr]. \label{IrJr}
\end{align}
As a result, $\Delta(\mathtt F)$ can be written in terms of these constituent intervals (not necessarily disjoint), as follows:
\begin{equation} 
\bigcup \bigl\{\mathtt J_r : r \in \mathcal R  \Bigr\} \subseteq \Delta(\mathtt F) \subseteq \bigcup \bigl\{\mathtt I_r : r \in \mathcal R \Bigr\}, \text{ with } \mathcal R = \mathcal R[q] := \bigl\{|p| : p \in \mathbb Z_{q}^d - \mathbb Z_q^d  \bigr\}.   
\label{Falconer-E-upper-lower} 
\end{equation} 
Since $|p|^2 = p_1^2 + \cdots + p_d^2$ can assume at most $dq^2$ distinct values for $p \in \mathbb Z_q^d - \mathbb Z_q^d$, we deduce that $\#(\mathcal R) \leq dq^2$, i.e., the number of distinct intervals $\{\mathtt I_r : r \in \mathcal R\}$ is at most $dq^2$. Also $\max(\mathcal R) \leq \sqrt{d}q$. The main observation in this section is that  the intervals $\mathtt I_r$ are disjoint for small values of $r$. Away from the origin, where $r$ is large, they overlap significantly, so that their subintervals $\mathtt J_r$ coalesce to generate a much larger interval, whose length is bounded below by a fixed dimensional constant. The interval $[a_{\rho}, b_{\rho}]$, whose existence is ensured by Theorem \ref{mainthm-2}, lies in this larger interval.  
 \begin{lem} \label{building-block-example-lemma} 
 There exists an absolute constant $c_0 \in (0, \frac{\gamma_d}{100})$, depending only on the dimension $d$, with the following property.
 \vskip0.1in
\noindent Suppose that $s$ is close to $d$, so that $\sigma := d/s \in (1, \frac{3}{2})$. Then for integers $q$ such that $c_0 q^{\sigma - 1} \gg 1$ and $\mathtt F[q]$ defined as in \eqref{def-building-block-F}, the following conclusions hold :
\vskip0.1in 
 \begin{enumerate}[(a)]
 \item The intervals $\mathtt I_r$ are disjoint for $r \in \mathcal R, \, r \leq 2c_0q^{\sigma-1}$, where $\mathcal R$ is defined in \eqref{Falconer-E-upper-lower}.  Any two distinct intervals in this range are separated by a distance of at least $q^{-\sigma}/(8c_0)$.\label{disjointness}
 \vskip0.1in
 \item 
 Near the origin, $\Delta(\mathtt F)$ is a disjoint union of intervals of length $\sim q^{-\sigma}$. Specifically, 
 \begin{equation} \bigsqcup \bigl\{\mathtt J_r : r \in \mathcal R, \, r \leq \frac{c_0}{2}q^{\sigma-1} \bigr\} \subseteq \Delta(\mathtt F) \cap \bigl[0, c_0q^{\sigma-2} \bigr) \subseteq \bigsqcup \bigl\{\mathtt I_r : r \in \mathcal R, \,  r \leq 2c_0 q^{\sigma-1} \bigr\}. \label{F-E-small-r} \end{equation} \label{small-disjoint-intervals}
 \item For $r \in \mathcal R$ with ${q^{2(\sigma-1)}}/{c_0^2} - 1\leq r < q$, the union of the intervals $\mathtt J_r$ consists of a single interval whose length is bounded from below by an absolute constant. As a result, 
 \begin{equation} \label{large-interval} 
 \Delta(\mathtt F) \supseteq   \Bigl[c_0^{-2}q^{2\sigma -3}, 1 \Bigr].
 \end{equation} 
 \label{large-single-interval}
 \end{enumerate} 
 \end{lem} 
\begin{proof} 
For part (\ref{disjointness}), we choose $r, r' \in \mathcal R$ such that $r' < r \leq 2c_0 q^{\sigma-1}$.  Since $r^2$ and $r'^2$ are distinct integers, 
\[ 1 \leq  | r^2 - r'^2| = ( r + r' )  ( r - r' ) \leq 4c_0 q^{\sigma-1} ( r - r'),   \]
and hence $r - r' \geq q^{1-\sigma}/(4c_0)$. This implies that 
\[ \text{dist}\bigl(\mathtt I_r, \mathtt I_{r'} \bigr)  \geq \frac{(r - r')}{q} - 2 \gamma_d^{-1} q^{-\sigma} \geq \Bigl[\frac{1}{4c_0} - 2 \gamma_d^{-1} \Bigr] q^{-\sigma} \geq \frac{q^{-\sigma}}{8c_0}> 0, \] 
establishing the disjointness of $\{\mathtt I_r: r \leq c_0 q^{\sigma-1}\}$. The relation \eqref{F-E-small-r} in part (\ref{small-disjoint-intervals}) now follows from \eqref{Falconer-E-upper-lower} after intersecting all the sets in the chain of inclusion with the interval $[0, c_0 q^{\sigma-2})$.
\vskip0.1in
\noindent Let us turn to part (\ref{large-single-interval}).  It follows from \eqref{Falconer-E-upper-lower} that
\begin{equation} 
\Delta(\mathtt F) \supseteq \bigcup \bigl\{\mathtt J_r : r \in \mathcal R, \, r > c_0^{-2}q^{2(\sigma-1)}, \; r \in \mathcal R \bigr\}.  \label{claim-1} \end{equation}  
In view of the definition of the intervals $\mathtt J_r$ from \eqref{IrJr}, two intervals $\mathtt J_{r}$ and $\mathtt J_{r'}$ overlap if and only if $|r - r'| \leq 2 \gamma_d q^{1 - \sigma}$.  
We will prove in Lemma \ref{number-theoretic-lemma} below, that 
\begin{equation}  \text{ for every $r \in \mathcal R \cap \bigl[c_0^{-2}q^{2(\sigma -1)}, q \bigr]$, $\exists r' \in \mathcal R$ with $r' > r$ and } |r - r'| \leq 2 \gamma_d q^{1 - \sigma}. \label{claim0} \end{equation} 
The claim in \eqref{claim0} says that for every $r$ in the range $[c_0^{-2}q^{2(\sigma-1)}, q)\cap \mathcal R$, the interval $\mathtt J_r$ intersects an interval $\mathtt J_{r'}$ lying to its right, with $r' \in \mathcal R$. If $r' \leq q$,  the same argument can be repeated with $r$ replaced by $r'$, leading to a chain of intervals each overlapping with the one to its right. Since the number of distinct values of $r$ is at most $dq^2$, iterating the assertion in \eqref{claim0} at most $dq^2$ times leads to the conclusion that the union of $\{\mathtt J_r : r \in \mathcal R, \, c_0^{-2} q^{2(\sigma-1)}-1 \leq r  \leq q \}$ is a single interval containing
\[ \Bigl[\frac{c_0^{-2}q^{2(\sigma-1)}}{q}, \frac{q}{q} + q^{-\sigma} \Bigr) \supseteq \bigl[c_0^{-2}q^{2\sigma-3},  1 \bigr].\] 
This is the desired inclusion, claimed in \eqref{large-single-interval}. 
\end{proof} 
\begin{lem} \label{number-theoretic-lemma}
For $d \geq 2$, the claim in \eqref{claim0} holds.  
\end{lem} 
\begin{proof} 
The proof is of a number-theoretic flavour. The argument is especially simple for $d \geq 4$ and yields slightly stronger statements than those claimed in \eqref{claim0} and \eqref{large-interval}. We address this scenario separately. According to Lagrange's theorem \cite[\S 20.5, p 302, Theorem 369]{Hardy-Wright}, every positive integer $k$ is the sum of four squares, i.e., the sum of square of four non-negative integers. Thus, 
\begin{align*} 
&\text{for $d \geq 4$ and any integer $k$ with $\sqrt{k} \in \bigl(\frac{2}{c_0}q^{\sigma-1}, q \bigr)$,} \\  
&\text{there exist $p, p' \in \mathbb Z_q^d - \mathbb Z_q^d \;$ such that $k = |p|^2$ and $k + 1= |p'|^2$.}
\end{align*}
In fact, $p, p'$ can be chosen to have no more than four non-zero coordinates, each of which must be strictly smaller than $q$ in absolute value. Consequently, $r = \sqrt{k} = |p| \in \mathcal R$, $r' = \sqrt{k+1} = |p'| \in \mathcal R$, and    
\[ r' - r = \sqrt{k+1} - \sqrt{k} = \frac{1}{\sqrt{k+1} + \sqrt{k}} \leq \bigl[ \frac{2}{c_0} q^{\sigma-1} \bigr]^{-1} \leq 2 \gamma_d q^{1 - \sigma}. \]
The last inequality follows from the assumption $c_0 \in (0, \gamma_d/100)$. This proves \eqref{claim0} for $r$ in the bigger range $|r| > 2q^{\sigma -1}/c_0$, which in turn leads to $\Delta(\mathtt F) \supseteq [2q^{\sigma-2}/c_0, 1]$, a bigger interval than the one claimed in \eqref{large-interval}.  
\vskip0.1in
\noindent The above argument involving Lagrange's theorem does not apply for $d = 2,3$. However, there is a wealth of number-theoretic literature addressing the asymptotic behaviour of integers that can be represented as a sum of two squares \cite{{Landau-1908}, {Berndt-Rankin-1995}, {Hardy-Ramanujan}, {Moree-Cazaran}}. In particular, the following result of Bambah and Chowla \cite{Bambah-Chowla-1947}, reproduced in \cite[p.137]{Halberstam-1983}, offers a viable substitute: For all $m \geq 1154$, the interval $(m, m+3m^{1/4})$ contains an integer that is expressible as the sum of two squares.  Thus, for $q$ sufficiently large, $r \in \mathcal R$, and  
\begin{align*}
&\text{$r = \sqrt{k} \in \Bigl[\frac{q^{2(\sigma-1)}}{c_0^2}-1, q \Bigr]$, one can find an integer $u\in \bigl[1, 3k^{\frac{1}{4}} \bigr]$ such that } \\ 
&\text{$k' = k+u$ can be expressed as $k' = |p'|^2 = r'^2$ for some $p' \in \mathbb Z_q^{d}$.} 
\end{align*} 
In  fact, \cite{Bambah-Chowla-1947} ensures that $p'$ can be chosen to have at most two nonzero coordinates, each smaller than $q$. Substituting this into $||p| - |p'|| = |r' - r|$ leads to the expression
\[ |r' - r| = \sqrt{k+u} - \sqrt{k} = \frac{u}{\sqrt{k+u} + \sqrt{k}}
\leq \frac{3k^{\frac{1}{4}}}{\sqrt{k}} \leq 3 k^{-\frac{1}{4}} \leq 6c_0 q^{1 - \sigma} \leq 2\gamma_d q^{1 - \sigma}. \] 
This completes the proof of \eqref{claim0} in all dimensions $d \geq 2$. 
\end{proof}
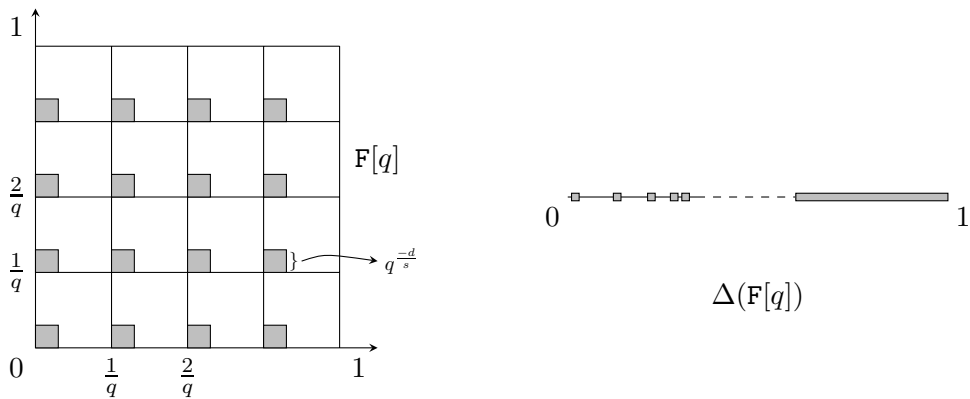
\begin{figure}
        \begin{tikzpicture}[>=stealth]

                \draw[step=1cm] (0,0) grid (4,4);

                \foreach \x in {0, 1, 2, 3}
                \foreach \y in {0, 1, 2, 3}
                \draw[fill=gray!50] (\x,\y) rectangle (\x+0.3, \y+0.3);

                \draw (-0.25,0) node[below] {$0$};
                \draw (4.25,0) node[below] {$1$};
                \draw (-0.25,4) node[above] {$1$};

                \foreach \x in {1,2}
                \draw (\x,0) node[below] {$\frac{\x}{q}$};
                \foreach \y in {1,2}
                \draw (0,\y) node[left] {$\frac{\y}{q}$};

                \draw [->](4,0) -- (4.5,0);
                \draw [->](0,4) -- (0,4.5);

                \draw (4.5,2.5) node {$\mathtt{F}[q]$};

                \draw (3.25,1.15) node [scale=0.65, right] {$\}$};
                \draw [->] (3.5,1.15) .. controls (3.8,1.25) .. (4.5,1.15) node
[scale=0.7, right] {$q^{\frac{-d}{s}}$};

                \def\eps{0.05}
                \def\x0{7}
                \def\y0{2}

                \draw (\x0,\y0) -- (\x0+1.7,\y0+0);

                \draw[fill=gray!50] (\x0+0.05,\y0-\eps) rectangle (\x0+0.15,
\y0+\eps);
                \draw[fill=gray!50] (\x0+0.60,\y0-\eps) rectangle (\x0+0.70,
\y0+\eps);
                \draw[fill=gray!50] (\x0+1.05,\y0-\eps) rectangle (\x0+1.15,
\y0+\eps);
                \draw[fill=gray!50] (\x0+1.35,\y0-\eps) rectangle (\x0+1.45,
\y0+\eps);
                \draw[fill=gray!50] (\x0+1.5,\y0-\eps) rectangle (\x0+1.6,
\y0+\eps);

                \draw[dashed] (\x0+1.7,\y0) -- (\x0+3.5,\y0);
                \draw (\x0+3.5,\y0) -- (\x0+5,\y0);

                \draw[fill=gray!50] (\x0+3,\y0-\eps) rectangle
(\x0+5,\y0+\eps);

                \draw (\x0-0.2,\y0) node [below] {$0$};
                \draw (\x0+5.2,\y0) node [below] {$1$};

                \draw (\x0+2.5,\y0-1) node[below] {$\Delta(\mathtt{F}[q])$};

        \end{tikzpicture}
        \caption{The set $\mathtt F[q]$ given by \eqref{def-building-block-F} and its distance set $\Delta(\mathtt F[q])$} \label{SingleBox}
\end{figure}
\subsection{A counterpoint to Rice's example} \label{Rice-example-counterpoint}
In Section \ref{all-suff-large-dis-intro-section}, we noted an example due to Rice \cite{2020Rice}, which shows that given any $d \geq 2$, there exists a set $A \subseteq \mathbb R^d$ of infinite Lebesgue measure that avoids infinitely many large distances,  for which the positivity criteria \eqref{Bourgain-condition} and \eqref{Bourgain-liminf-condition} barely fail. Using the example constructed in Section \ref{building-block-example-section}, we establish a complementary statement. There exists a set of infinitesimally small Lebesgue measure whose connected components can be  made to shrink arbitrarily fast at infinity, yet which manages to capture all distances outside every compact set. The requirement of small connected components is to rule out trivial examples, such as lines or curves; the set $A$ we construct is totally disconnected at infinity. In particular, the criteria \eqref{Bourgain-condition} and \eqref{Bourgain-liminf-condition} fail in a very strong sense for such a set. Not only is its upper density zero, its density on any ball of radius $R$ is smaller than a pre-assigned constant multiple of $R^{-d}$.       
\begin{prop} \label{Rice-counterpoint-lemma} 
Fix any $d \geq 2$ and $\ell_0 > 0$. Let $\eta: [0, \infty) \rightarrow (0, 1)$ be a decreasing function such that $\eta(R) \searrow 0$ as $R \rightarrow \infty$. Then there exists a set $A \subseteq \mathbb R^d$ with $\lambda_d (A) \leq \ell_0$ such that 
\begin{equation} \label{density+distance} 
\Delta (A \setminus B(0;R)) = [0, \infty) \quad \text{ and } \quad \text{diam}(\mathscr{C}_R) \leq \eta(R) \quad \text{ for all $R \geq 1$},
\end{equation}  
where  $\mathscr{C}_R$ is any connected component of $A \setminus B(0;R)$. 
\end{prop} 
\subsubsection{Construction of the set $A$}
Let us first choose any increasing sequence of positive integers $\{\mathtt r_n : n \geq 1\}$; for example, $\mathtt r_n = n$ will do. Set $ R_j := \mathtt r_1 + \cdots + \mathtt r_j$. 
Let $c_0 > 0$ be the small absolute constant appearing in Lemma \ref{building-block-example-lemma}, and let $\sigma > 1$ be a rational parameter close to 1 so that $3 - 2 \sigma > 0$.  Given $\ell_0 > 0$ and a function $\eta$ as in the statement of the proposition, we next specify an algorithm for obtaining strictly increasing sequences of positive integers $\{q_n : n \geq 1\}$ with $q_n, q_n^{\sigma} \in \{2^{j} :j \in \mathbb N\}$ for all $n$:
 \begin{align}  
 &q_1 \gg c_{0}^{-1}, \quad \mathtt r_1q_1^{d(1 - \sigma)} < \frac{\ell_0}{2}; \text{ having selected $q_1, \ldots q_{n-1}$, we choose $q_n$ with } \label{r1q1}\\  
 &\mathtt r_{n}^d q_{n}^{d(1 - \sigma)} \leq \frac{1}{2}\mathtt r_{n-1}^d q_{n-1}^{d(1 - \sigma)}, \quad \mathtt r_n q_n^{-\sigma} < \min \Bigl[\mathtt r_{n-1} q_{n-1}^{-\sigma}, \frac{1}{\sqrt{d}}\eta(R_n) \Bigr], \quad \mathtt r_n q_n^{2 \sigma -3} < \frac{1}{n}. \label{rnqn-choice}
 \end{align}  
 We are now in a position to describe the set $A$. Let us denote $\mathtt e_1 = (1, 0, \ldots, 0)$ and $\mathtt F_n := \mathtt F[q_n; \sigma] \subseteq [0,1]^d$, where $\mathtt F[q; \sigma]$ has been defined as in \eqref{def-building-block-F}. Set 
\begin{equation} \label{def-A-all-distances}
A := \bigsqcup_{n=1}^{\infty} A_n \text{ where } A_n := R_{n-1} \mathtt e_1 + \mathtt r_n \mathtt F_n, \; \text{ with the convention that } R_0 = 0.  
\end{equation} 
Here $\mathtt r \mathtt F$ denotes the dilate $\mathtt r \mathtt F := \{\mathtt r x : x \in \mathtt F \}$. Thus $A$ is the disjoint union of scaled copies of $\mathtt F_n$ inserted in increasingly large cubes that move away from the origin in the direction of $\mathtt e_1$ as $n$ gets larger; see diagram \ref{Multibox}. 
\tikzset{
        pics/basicset/.style args={#1/#2/#3}{
                code={
                \foreach \x in {0,...,\numexpr#1-1}
                        \foreach \y in {0,...,\numexpr#1-1}
                                \draw[fill=gray!50] (\x * #3/#1,\y * #3/#1)
rectangle (\x * #3/#1 +#2, \y * #3/#1+#2);

                \draw (0,0) rectangle (#3,#3);

        }
}
}

\begin{figure}
        \begin{tikzpicture}[>=stealth]

                \pic at (0,0) {basicset=2/0.2/1};
                \pic at (1,0) {basicset=5/0.15/2};
                \pic at (3,0) {basicset=12/0.08/3};

                \draw (1,0) node[scale=0.8] [below] {$R_1$};
                \draw (3,0) node[scale=0.8] [below] {$R_2$};
                \draw (6,0) node[scale=0.8] [below] {$R_3$};
                
                \draw (0,1) node[scale=0.8] [left] {$\mathtt r_1$};
                \draw (0,2) node[scale=0.8] [left] {$\mathtt r_2$};
                \draw (0,2.8) node[scale=0.8] [left] {$\mathtt r_3$}; 

                \draw (0.5,1) node[scale=0.8] [above] {$\mathtt r_1 \mathbb{F}[q_1]$};
                \draw (2,2) node[scale=0.8] [above] {$\mathtt r_2 \mathbb{F}[q_2]$};

                \draw [->](0,0) -- (6.5,0);
                \draw [->](0,0) -- (0,3.5);

                \draw (6.5,2) node {$\cdots$};
        \end{tikzpicture}
\caption{A zero-density set with all distances} \label{Multibox}
\end{figure}
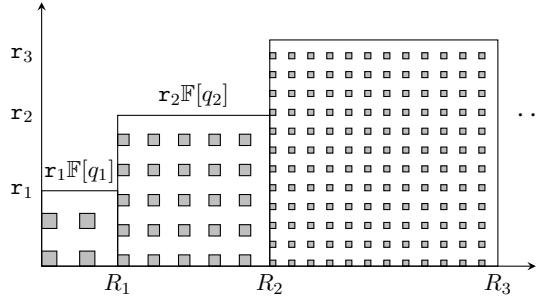

\subsubsection{Proof of Proposition \ref{Rice-counterpoint-lemma}}
As a consequence of \eqref{r1q1} and \eqref{rnqn-choice}, we deduce that
\begin{equation}  \label{choosing-qnRn-3}
\lambda_d(A) = \sum_{n=1}^{\infty} \lambda_d(A_n) = \sum_{n=1}^{\infty} \mathtt r_n^d q_n^{d(1 - \sigma)} \leq \mathtt r_1^d q_1^{d(1 - \sigma)} \sum_{n=0}^{\infty} 2^{-n} \leq \frac{\ell_0}{2} \sum_{n=0}^{\infty} 2^{-n} \leq \ell_0,   
\end{equation}  
which establishes the volume estimate for $A$. 
\vskip0.1in
\noindent Next let us examine the behaviour of $A$ outside a large ball $B(0;R)$. Given $R \geq 1$, suppose that $n$ is the unique index such that $R_{n-1} \leq R < R_n$. Then $A \setminus B(0;R)$ contains the disjoint union of the blocks $A_j$ for $j \geq n$. 
Since each connected component $\mathscr{C}_R$ of $A_j$ is a cube of side-length $\mathtt r_j q_j^{-\sigma}$, the second condition in \eqref{rnqn-choice} gives 
\[ \text{diam}(\mathscr{C}_R) \leq \sqrt{d} \sup_{j \geq n}\mathtt r_j q_j^{-\sigma} \leq \sqrt{d} \mathtt r_n q_n^{-\sigma} \leq \eta(R_n) \leq \eta(R), \]
which is one of the conclusions claimed in \eqref{density+distance}.
\vskip0.1in     
\noindent It remains to examine $\Delta(A \setminus B(0;R))$. For $R \in [R_{n-1}, R_n)$ and by Lemma \ref{building-block-example-lemma}, we obtain 
\begin{align*} \Delta\bigl(A \setminus B(0;R) \bigr) \supseteq \bigcup_{j=n}^{\infty} \Delta(A_j) &\supseteq \bigcup_{j=n}^{\infty} \mathtt r_j \Delta(\mathtt F_j) \\  &\supseteq \bigcup_{j=n}^{\infty} \mathtt r_j \bigl[c_0^{-2} q_j^{2 \sigma -3}, 1 \bigr] \supseteq \bigcup_{j=n}^{\infty} \bigl[c_0^{-2}/j, \mathtt r_j\bigr] = [0, \infty). \end{align*} 
The last step uses the third condition in \eqref{rnqn-choice}, which ensures that intervals in the union are nested and increasing for consecutive indices $n$, thereby creating the single, unbounded interval $[0, \infty)$.  
\qed
\vskip0.1in
\noindent {\em{Remark: }} For any $\rho \in (0,1)$ and $A$ as in \eqref{def-A-all-distances},
\[ \mathscr{V}(\rho, s;A) \supseteq \{\mathtt r_n : n \geq 1\}  \] 
where $\mathscr{V}$ is the collection of high-density scales of $A$ defined in \eqref{def-V}.  This is because 
\[ A_{\mathtt r_n}(R_{n-1}\mathtt e_1) = \mathtt F_n = \mathtt F_n \cap Q_0, \; \text{ with  } Q_0 = [0,1]^d, \] 
and we have observed in Section \ref{building-block-example-section} that $\mathcal H_{\infty}^s(\mathtt F_n \cap Q_0) = 1 = \ell(Q_0)^d$.    
If $\{\mathtt r_n: n \geq 1\}$ is a slow-growing sequence like $\mathtt r_n = n$ for which $\mathtt r_{n+1}/\mathtt r_n \rightarrow 1$, then $A$ is well-distributed in the sense of definition \eqref{slower-than-lacunary} with growth rate controlled by any constant $C_0 > 1$. It therefore obeys the hypotheses, and hence the conclusions of Theorems \ref{mainthm-4} and \ref{mainthm-3}. One is able to deduce from these theorems that $A$ contains all sufficiently large distances, a slightly weaker version of Proposition \ref{Rice-counterpoint-lemma}. However, if the sequence $\{\mathtt r_n: n \geq 1\}$ is rapidly increasing, Theorem \ref{mainthm-3} is no longer applicable. The conclusion of Proposition \ref{Rice-counterpoint-lemma} then relies on special properties of $\Delta(\mathtt F_n)$ derived in Section \ref{building-block-example-section}. 
\subsection{The role of growth rate in well-distributed sets} \label{uniformly-welldistributed-section} 
\begin{lem} 
For every $d \geq 2$,  and $s < d$ there exist a set $A \subseteq \mathbb R^d$ and a rapidly growing sequence $\{R_n: n \geq 1\}$ such that  
\begin{equation} \label{well-distributed-def-2}
\lim_{n \rightarrow \infty} \sup_{x \in \mathbb R^d} \mathcal H_{\infty}^s \bigl(A_{R_n}(x)\bigr) = 1  
\end{equation}   
and yet $A$ avoids infinitely many large distances. Here $A_R(x)$ refers to the normalized truncations of $A$ defined in \eqref{def-ARx}.   
\vskip0.1in
\noindent Thus the growth conditions \eqref{slower-than-lacunary} and \eqref{Rn-growth} in Theorems \ref{mainthm-3} and \ref{mainthm-4} cannot be dropped.  
\end{lem} 
\begin{proof} 
The proof is a small variation on the example in Section \ref{Rice-example-counterpoint} and the one in \cite{2020Rice}. Let us pick parameters $\mathtt d_j \nearrow \infty$ and $\kappa_j \searrow 0$ in the order $\mathtt d_1, \kappa_1, \mathtt d_2, \kappa_2, \ldots$ and obeying the properties
\begin{align} 
 \bigl| |x| - \mathtt d_j \bigr| > 4\kappa_n \sqrt{d} \text{ for all $x \in \mathbb Z^d$ and } 1 \leq j \leq n,  \label{Rice-conditions-1} \\ \mathtt d_n \geq 100d^2 \mathtt d_{n-1} \quad \text{ and } \quad \mathtt d_{n}^{1 - \sigma} \leq \mathtt \kappa_{n-1} \quad \text{ where } \sigma = \frac{d}{s}.  \label{Rice-conditions-2} \end{align}  
As has been pointed out in \cite{2020Rice}, the condition \eqref{Rice-conditions-1} is achievable. The set $\{|x| : x \in \mathbb Z^d\}$ of lengths of integer vectors is discrete with no limit points, since its intersection with any compact subinterval of $[0, \infty)$  is finite. Let $\mathtt q_n$ denote the unique integer power of 2 such that $\mathtt d_n \in (\mathtt q_n/2, \mathtt q_n]$. Set $\mathbb Z_{\mathtt q} := \{0, 1, \ldots, \mathtt q -1\}$, and define 
\[ A := \bigsqcup_{n=1}^{\infty} A_n \; \text{ where } \; A_n := 10d \mathtt d_n \mathtt e_1 + \bigsqcup \bigl\{ Q(p; \kappa_{n-1}) : p \in \mathbb Z_{\mathtt q_n}^d \bigr\} \text{ and } \mathtt e_1= (1, 0, \ldots, 0). \]  
It follows from the first relation in \eqref{Rice-conditions-2} that 
\begin{align*} \sup\{x_1 : x = (x_1, \ldots, x_n) \in A_n \} &\leq 10d \mathtt d_n + \mathtt q_n \leq 10d \mathtt d_n + 2 \mathtt d_n \\ &< 20 d \mathtt d_n < 10 d \mathtt d_{n+1} = \inf\{x_1 : x \in A_{n+1} \} , 
\end{align*}  
proving that the blocks $A_n$ are disjoint, and therefore $A \cap Q(10d\mathtt d_n; \mathtt q_n) = A_n$.  Let us choose the truncation parameters $R_n = \mathtt q_n$, $x_n = 10d\mathtt d_n\mathtt e_1$. The second relation in \eqref{Rice-conditions-2} ensures that $\kappa_{n-1} \geq \mathtt d_n^{1 - \sigma} \geq \mathtt q_n^{1-\sigma}$, which in turn leads to the inclusion 
\begin{equation*} 
[0,1]^d \supseteq \frac{1}{\mathtt q_n} \big[A_{n} - 10d\mathtt d_n \mathtt e_1 \bigr] = A_{R_n}(x_n) 
= \bigsqcup \Bigl\{Q \bigl(\frac{p}{\mathtt q_n}, \frac{\kappa_{n-1}}{\mathtt q_n} \bigr) : p \in \mathbb Z_{\mathtt q_n}^d \Bigr\} \supseteq \mathtt F[\mathtt q_n].
\end{equation*}
Here $\mathtt F[q]$ is the building block example given in \eqref{def-building-block-F}, and $A_R(x)$ is the normalized truncation of $A$ defined in \eqref{def-ARx}. Recalling from Section \ref{building-block-example-section} that $\mathcal H_{\infty}^{s}(\mathtt F[q]) = 1$, we conclude:
\[ \mathtt F[\mathtt q_n] \subseteq A_{R_n}(x_n)  \subseteq [0,1]^d \quad \text{ and hence } \quad  \mathcal H_{\infty}^s \bigl(A_{R_n}(x_n) \bigr) =1 \text{ for every $n \geq 1$},  \]
which leads to the claim \eqref{well-distributed-def-2}. Let us pause to observe from the selection of $\mathtt q_n$ that $\{R_n : n \geq 1\}$ is a rapidly increasing sequence  with $R_{n+1}/R_n \nearrow \infty$, and hence does not meet the controlled growth condition \eqref{Rn-growth}. 
\vskip0.1in
\noindent It remains to show that $\mathtt d_j \notin \Delta(A)$ for every $j \geq 1$. Towards a contradiction, suppose if possible that there exist $x, y \in A$ such that $|x - y| = \mathtt d_j$. Since $A$ is the disjoint union of blocks $A_n$, there exist indices $m, n \geq 1$ such that $x \in A_m$ and $y \in A_n$. Without loss of generality, we may assume $m \leq n$. Let us write
\begin{align}  x &=10  \mathtt d_m \mathtt e_1  + p_x + u_x, \quad p_x \in \mathbb Z_{\mathtt q_m}^d, \; u_x \in [0, \kappa_m]^d,  \nonumber \\ 
y &= 10 \mathtt d_n \mathtt e_1 + p_y + u_y, \quad \; p_y \in \mathbb Z_{\mathtt q_n}^d, \;  u_y \in [0, \kappa_n]^d, \text{ so that }  \nonumber \\ 
x - y &= 10 (\mathtt d_m - \mathtt d_n) \mathtt e_1 + (p_x - p_y) + (u_x - u_y).  \label{x-y} 
\end{align} 
We consider two cases, each of which will lead to a contradiction. Suppose first that $m = n$. Then it follows from \eqref{x-y} that
\begin{align*} 
&\bigl| \mathtt d_j - |p_x - p_y| \bigr| \leq |(x - y) - (p_x - p_y)| \leq |u_x - u_y|  \leq 2 \kappa_m \sqrt{d}, \\ &\text{which in turn implies that }   \exists z \in \mathbb Z^d \text{ such that } \bigl||z| - \mathtt d_j \bigr| \leq 2 \kappa_m \sqrt{d}.  
\end{align*} 
We deduce from the choice of parameters \eqref{Rice-conditions-1} that $j > m$. But this would then imply that 
\[ \mathtt d_j = |x - y| \leq  \bigl| p_x - p_y  \bigr| + \bigl| u_x - u_y  \bigr| \leq 2 \sqrt{d} \mathtt d_m + 2 \sqrt{d} \kappa_m < 100 d^2 \mathtt d_m, \]  
a contradiction to the second condition in \eqref{Rice-conditions-2}.
\vskip0.1in
\noindent The proof for $m < n$ is similar. On one hand, we deduce from \eqref{x-y} that 
\begin{align*} \mathtt d_j = |x-y| &\geq 10d|\mathtt d_n - \mathtt d_m| - |p_x - p_y| - |u_x - u_y| \\ &\geq 10 d \mathtt d_n - 10d \mathtt d_{m} - 2 \sqrt{d} \mathtt d_n - 2 \sqrt{d} \kappa_m > \mathtt d_n, 
\end{align*} 
which implies that $j > n$. On the other hand, the same relation \eqref{x-y} yields
\begin{align*} 
\mathtt d_j = |x-y| &\leq 10d|\mathtt d_n - \mathtt d_m| + \bigl| p_x - p_y  \bigr| + \bigl| u_x - u_y  \bigr| \\ 
&\leq 20 d \mathtt d_n + 2 \sqrt{d} \mathtt d_n + 2 \kappa_m \sqrt{d} < 100 d^2 \mathtt d_n.  \end{align*} 
In view of the previous deduction $j > n$, this contradicts \eqref{Rice-conditions-2}. The contradiction in both cases implies that $\mathtt d_j \notin \Delta(A)$, as claimed in the lemma.    
\end{proof}

\subsection{Iterative schemes} 
Sets $\mathtt F[q]$, defined as in \eqref{def-building-block-F}, have positive Lebesgue measure. Iterative constructions with these as building blocks yield sparser sets that are Lebesgue-null and have strictly positive codimension. In this section, we construct a few such sets of Hausdorff dimension $s < d$. Examined in light of the results of this paper, specifically Lemma \ref{building-block-example-lemma} and Theorem \ref{mainthm-cor}, these sets provide examples of several distinct but related phenomena:
\begin{itemize} 
\item Lemma \ref{distance-counterexample-lemma} shows the existence of a compact set $\mathtt F^{\ast} \subseteq [0,1]^d$, $d \geq 2$, whose distance set does not contain the interval $[0, a)$ for any $a > 0$. The construction of $\mathtt F^{\ast}$ is not new; it follows an example of Falconer \cite[Theorem 2.4]{Falconer-1986}, used in a different context.  The aforementioned property of its distance set was noted in subsequent work of Mattila and Sj${\ddot{\text{o}}}$lin \cite{Mattila-Sjolin-99}.
\vskip0.1in
\item However, Lemma \ref{Quasiregular-NoInterval-Lemma} offers a further refinement. For $d \geq 2$, it shows that there exists a compact set $\mathtt E^{\ast} \subseteq [0,1]^d$ that is {\em{locally uniformly $s$-dimensional and quasi-regular}}, in the sense of \eqref{loc-u-dim-def} and \eqref{Quasiregular-def} respectively, which enjoys the same non-Steinhaus property, namely $\Delta(\mathtt E^{\ast}) \not\supseteq [0, a)$ for any $a > 0$. The existence of such a set appears to be new, and shows that the Steinhaus property need not be ensured by uniform dimensional and quasi-regular sets. It also offers an interesting counterpoint to the following complementary statement.  
\vskip0.1in   
\item For $d \geq 2$, there exists a compact set $\mathtt K^{\ast} \subseteq \mathbb R^d$ which is a finite union of affine copies of $\mathtt E^{\ast}$, therefore of the same dimension as $\mathtt E^{\ast}$ as well as locally uniformly $s$-dimensional and quasi-regular, whose distance set does contain an interval of the form $[0, a)$ (Corollary \ref{Steinhaus-sparse-corollary} and Lemma \ref{Quasi-regular-Interval-Lemma}).  The set $\mathtt E^{\ast}$ has further applications, a case in point being Corollary \ref{Sparse-all-distances-corollary}; see Section \ref{Sparse-set-all-distances-proof-section}. 
\end{itemize}  
   
\subsubsection{Distance sets not containing $[0, a)$ for any $a$} \label{Appendix-2} 

Following \cite{Falconer-1986}, let us fix a rapidly increasing sequence $\{q_1, q_2, \ldots, \}$ of integers in $2^{\mathbb N}$. For simplicity and ease of exposition, let us assume that $\sigma = d/s$ is a rational number close to but larger than 1, and that $q_{n}^{\sigma}$ is an integer that divides $q_{n+1}$, for every $n \geq 1$. Setting $\mathbb Z_{q} := \{0, 1, \ldots, q - 1\}$, we use the sets $\mathtt F[q]$ as building blocks to define $\mathtt F \subseteq \mathbb R^d$ as follows: 
\begin{align} \label{Falconer-E-original} \mathtt F^{\ast} &= \mathtt F^{\ast}[q_1, q_2, \ldots] := \bigcap_{n=1}^{\infty} \mathtt F_n, \; \text{ where }  \\ \mathtt F_n &= \mathtt F[q_n; \sigma] = \Bigl\{x \in [0,1]^d \bigl| \exists p \in \mathbb Z_{q_n}^d \ni x \in \frac{p}{q_n} + q_n^{-\sigma}[0,1]^d  \Bigr\}. \label{Falconer-F-original}
\end{align}  
It has been proved in \cite[Section 8.5]{Falconer-Book-GFS} that $\dim_{\mathbb H}(\mathtt E) = s$, where $\sigma = d/s$.  
 For $s > (d+1)/2$, Theorem \ref{mainthm-1} of Mattila and Sj$\ddot{\text{o}}$lin implies that $\text{int}\bigl(\Delta(\mathtt E) \bigr) \ne \emptyset$. However, they observed the following feature. 
\begin{lem}[\cite{{Mattila-Sjolin-99},{Falconer-1986}}] \label{distance-counterexample-lemma}
For the set $\mathtt F^{\ast}$ defined in \eqref{Falconer-E}, there does not exist any $a > 0$ with the property that $\Delta(\mathtt F^{\ast}) \supseteq [0, a)$.  
\end{lem} 
\begin{proof}
Fix $a > 0$, and recall the absolute constant $c_0$ whose existence has been guaranteed in Lemma \ref{building-block-example-lemma}. Then $c_0 q_n^{\sigma-2} < a$ for all sufficiently large indices $n$. Since $\Delta(\mathtt F^{\ast}) \subseteq \Delta(\mathtt F_n)$ for every $n$, it follows from Lemma \ref{building-block-example-lemma}(\ref{small-disjoint-intervals}) that $\Delta(\mathtt F^{\ast}) \cap [0, c_0 q_n^{\sigma -2})$ is contained in a disjoint union of intervals of length $\sim q_n^{-\sigma}$ and separated by at least $\sim q_n^{-\sigma}$. In particular, $[0, a) \not\subseteq \Delta(\mathtt F^{\ast})$.
\end{proof}   
\subsubsection{Locally uniform and quasi-regular sets} \label{uniform-and-quasiregular-section} 
Fix $\sigma = d/s$ a rational number close to but larger than 1, and integers $q, q^{\sigma} \in \{2^{j}: j \in \mathbb N\}$. Setting $\mathbb Z_{q} := \{0, 1, \ldots, q - 1\}$ as before, we use the sets $\mathtt F[q]$ in \eqref{def-building-block-F} as building blocks to define a sparser set $\mathtt E^{\ast} \subseteq \mathbb R^d$ as follows: 
\begin{align} \label{Falconer-E} 
\mathtt E^{\ast} &:= \bigcap_{n=1}^{\infty} \mathtt E_n, \quad \text{ where } \mathtt E_1 := \mathtt F[q], \; \text{ and } \\ 
\mathtt E_{n} = \mathtt E_n&:= \bigcup \Bigl\{\frac{p}{q} + q^{-\sigma} \mathtt E_{n-1}: \; p \in \mathbb Z_q^d \Bigr\} \nonumber \\ &\;= \bigcup \Bigl\{\sum_{k=1}^{n} \frac{p_k}{q^{(k-1)\sigma + 1}} + q^{-n\sigma}[0,1]^d : p_k \in \mathbb Z_q^d \Bigr\} \; \text{ for } n \geq 2. \label{self-similar-Falconer}
\end{align}  
In other words, $\mathtt E_n \subseteq [0,1]^d$ is a $q^{nd}$-fold disjoint union of basic (dyadic) cubes of sidelength $q^{-n \sigma}$. Given $\mathtt E_n$, the next iterate $\mathtt E_{n+1}$ is obtained by inserting an affine copy of $\mathtt E_1$ in each basic cube of $\mathtt E_n$. Thus $\mathtt E^{\ast}$ is a bounded self-similar set of Hausdorff dimension $s$ for any choice of $q$, in the sense of \cite{Strichartz-Asymptotics}, and therefore locally uniformly $s$-dimensional and quasi-regular by \cite[Theorem 5.8]{Strichartz-Asymptotics}. The construction ensures it is totally disconnected. Let us first record the Hausdorff dimension and content of $\mathtt E^{\ast}$. 
\begin{lem} \label{Falconer-example-content-lemma}
For any rational $\sigma > 1$ and integers $q, q^{\sigma}$ that are powers of 2,  the following conclusions hold for the set $\mathtt E^{\ast}[q]$ as in \eqref{Falconer-E}: 
\begin{equation} \label{Falconer-example-content}  
\mathcal H_{\infty}^{\tau}(\mathtt E^{\ast}) =  \begin{cases} 0 &\text{ for } \tau > s, \\ 1 &\text{ for } \tau \leq s.  \end{cases}  \end{equation} 
In particular, $\dim_{\mathbb H}(\mathtt E^{\ast}) = s$. Moreover, for $\tau \leq s$, and any $n \geq 1$, 
\begin{equation}  \mathcal H_{\infty}^{\tau}(\mathtt E^{\ast} \cap Q) = \ell(Q)^{\tau} \text{ for every basic cube $Q \subseteq \mathtt E_n$, $\ell(Q) = q^{-n\sigma}$.} \label{maximal-density} \end{equation}   
\end{lem} 
\begin{proof} Constructions of the type \eqref{Falconer-E} are ubiquitous in fractal geometry. Methods for computing their Hausdorff dimension and measure may be found in standard textbooks, \cite[Chapters 2, 3, 4]{Falconer-Book-FG} or \cite[Section 1.5]{Falconer-Book-GFS}. We only sketch the proof of the lemma, referring to the existing literature for further details. For instance, \cite[Theorem 1.15]{Falconer-Book-GFS} shows that 
\[ \dim_{\mathbb H}(\mathtt E^{\ast})=s \; \text{ and } \; \mathbb H^s(\mathtt E^{\ast}) = 1, \; \text{ and hence } \mathbb H^{\tau}(\mathtt E^{\ast}) = 0 \text{ for $\tau > s$.} \] 
As mentioned in Section \ref{preliminaries-section}, the quantities $\mathbb H^{\tau}(\cdot)$ and $\mathcal H_{\infty}^{\tau}(\cdot)$ vanish simultaneously, resulting in $\mathcal H^{\tau}(\mathtt E^{\ast}) = 0$ for $\tau > s$. On the other hand $\mathtt E^{\ast}$ is self-similar, hence the results of \cite{Farkas-Fraser-2015} imply that the Hausdorff measure and Hausdorff content of $\mathtt E^{\ast}$ match at the critical dimension; namely, $\mathbb H_{\infty}^{s}(\mathtt E^{\ast}) = 1$. Consequently, $1 = \mathbb H_{\infty}^s(\mathtt E^{\ast}) \leq \mathcal H_{\infty}^{s}(\mathtt E^{\ast}) \leq \mathbb H^s(\mathtt E^{\ast}) = 1$. Combining with the relation $1 = \mathcal H_{\infty}^{s}(\mathtt E^{\ast}) \leq  \mathcal H_{\infty}^{\tau}(\mathtt E^{\ast}) \leq 1$ leads to $\mathcal H_{\infty}^{\tau}(\mathtt E^{\ast})=1$ for $\tau \leq s$. This proves \eqref{Falconer-example-content}. 
\vskip0.1in
\noindent  We turn now to \eqref{maximal-density}. The self-similar structure of $\mathtt E^{\ast}$ implies that for any basic cube $Q \subseteq \mathtt E_n$ with $\ell(Q) = q^{-n\sigma}$ and $\mathbb T_{Q}(x) := (x - c(Q))/\ell(Q)$ we have 
\[ \mathbb T_{Q}\bigl[ \mathtt E^{\ast} \cap Q \bigr] = \mathtt E^{\ast}, \; \text{ therefore } \; \mathcal H_{\infty}^{\tau}(\mathtt E^{\ast} \cap Q) = \ell(Q)^{\tau} \mathcal H_{\infty}^{\tau}(\mathtt E^{\ast}) = \ell(Q)^{\tau}.  
\] 
The last step follows from \eqref{Falconer-example-content} for $\tau \leq s$, and the translation and scaling properties of $\mathcal H_{\infty}^s(\cdot)$, concluding the proof of the lemma. 
\end{proof} 
\vskip0.1in 
\noindent The relation \eqref{maximal-density} says that $\mathtt E^{\ast}$ has full density in every basic cube of $\mathtt E_n$. Each such cube (which has sidelength $q^{-n \sigma}$) therefore lies in $\mathcal D_{\rho}(\mathtt E^{\ast})$ for every $\rho \in (0,1)$. Theorem \ref{mainthm-cor} states: 
\begin{equation}  \Delta(\mathtt E^{\ast}[q; \sigma]) \supseteq \bigcup_{n=0}^{\infty} q^{-n\sigma} [a_{\rho}, b_{\rho}] \quad \text{ provided } s \in (d - \varepsilon_{\rho}, d). \label{Falconer-example-distance-set}  \end{equation} 
\subsubsection{Locally uniform and quasi-regular sets with and without the Steinhaus property} \label{Steinhaus-property-section}
\noindent We will say that a set $E \subseteq \mathbb R^d$ has the {\em{Steinhaus property}} if there exists $a > 0$ such that $[0, a) \subseteq \Delta(E)$. Corollary \ref{Steinhaus-sparse-corollary} claims the existence of a special type of set $\mathtt K^{\ast} \subseteq \mathbb R^d$ for $d \geq 2$ of Hausdorff dimension $s < d$ that has the Steinhaus property. We prove this corollary in this section. The set $\mathtt E^{\ast}$ defined in \eqref{Falconer-E} does not have the Steinhaus property but it is possible to create one that does using finitely many disjoint affine copies of it. 
\begin{lem} \label{Quasiregular-NoInterval-Lemma} 
Fix $\sigma > 1, \sigma \in \mathbb Q$. Suppose that $\mathtt E^{\ast} = \mathtt E^{\ast}[q, \sigma]$ is as in \eqref{Falconer-E}, where $q$ is an integer power of 2 that is large enough to satisfy $c_0 q^{\sigma-1} \gg 1$, with $c_0$ being the absolute constant specified in Lemma \ref{building-block-example-lemma}.   
Then \[ \Delta(\mathtt E^{\ast})  \not\supseteq [0, a)  \quad \text{ for any $a > 0$}. \]  
\end{lem} 
\begin{proof}
The main observation is that $\mathtt E_n$ is a disjoint union of cubes of the form $p/q_n + q_n^{-\sigma_n}[0,1]^d$, $p \in \mathbb Z_{q_n}^d$, and hence  
\begin{equation}  \label{Estar-Fn}
\mathtt E^{\ast} \subseteq \mathtt E_n \subseteq \mathtt F[q_n;\sigma_n], \; \text{ with } q_{n} = q^{(n-1)\sigma + 1}, \; \sigma_n = \frac{n \sigma}{(n-1) \sigma + 1} \; \text{ so that } q_n^{\sigma_n} = q^{n \sigma}. 
\end{equation} 
Here $\mathtt F[q; \sigma]$ denotes the building block set \eqref{def-building-block-F}. Let us note that $\sigma_n > 1$ for all $n \geq 1$, so Lemma \ref{building-block-example-lemma} applies. Further, $q_n^{\sigma_n-1} = q^{\sigma -1}$,  
  \begin{align*} 
  &\sigma_n - 2 = \frac{n \sigma}{(n-1)\sigma + 1} - 2 = -\frac{(n-2) \sigma + 2}{(n-1) \sigma + 1}, \quad\text{ so that } \\ &q_n^{\sigma_n-2} = q^{-(n-2) \sigma - 2} \longrightarrow 0 \text{ as } n \rightarrow \infty.
  \end{align*}  
Given any $a > 0$, we can therefore choose $n$ such that $c_0q_n^{\sigma_n-2} = c_0q^{-(n-2) \sigma - 2}  < a$. It follows from \eqref{Estar-Fn} that $\Delta(\mathtt E^{\ast}) \subseteq \Delta( \mathtt F[q_n;\sigma_n])$ for every $n \geq 1$. The assumption $c_0 q_n^{\sigma_n-1} = c_0 q^{\sigma-1} \gg 1$ ensures that Lemma \ref{building-block-example-lemma}(\ref{disjointness}) is applicable. It implies that $\Delta(\mathtt E^{\ast})$ is contained in a disjoint union of intervals of length $q_n^{-\sigma_n} = q^{-n \sigma}$ lying in $[0, c_0q_n^{\sigma_n-2}]$. The proof now follows exactly in the same way as Lemma \ref{distance-counterexample-lemma}.   
\end{proof} 
 
\begin{lem} \label{Quasi-regular-Interval-Lemma} 
Given $d \geq 2$ and $\rho > 0$, let $\varepsilon_{\rho}, a_{\rho}, b_{\rho}$ be as in Theorem \ref{mainthm-cor}. For $s \in (d - \varepsilon_{\rho}, d) \cap \mathbb Q$, let us choose an integer $q$ with $q, q^{\sigma} \in \{2^j : j \in \mathbb N\}$ and $q^{\sigma} > b_{\rho}/a_{\rho}$, where $\sigma = d/s$. 
\vskip0.1in
\noindent Let $m$ be the largest integer such that $q^{\sigma} > (b_{\rho}/a_{\rho})^m$. Then there exist $\{x_k : 0 \leq k \leq m\} \subseteq \mathbb R^d$ and $\{u_k: 0 \leq k \leq m\} \subseteq (0, \infty)$ such that the compact and totally disconnected set 
\begin{equation} \mathtt K^{\ast} := \bigsqcup_{k=0}^{m} \bigl(x_k + u_k \mathtt E^{\ast}\bigr), \text{ which has } \quad \text{dim}_{\mathbb H}(\mathtt K^{\ast}) = s \label{def-K} \end{equation}    
also has the Steinhaus property, i.e., $\Delta(\mathtt K^{\ast}) \supseteq [0, a)$ for some $a > 0$. In particular, Corollary \ref{Steinhaus-sparse-corollary} holds. 
\end{lem} 
\begin{proof} 
Set $x_0 = 0$, $u_0 = 1$. The dilation factors $\{u_1, \ldots, u_m \}$ will be determined shortly, but given these, the translation parameters $x_k$ are chosen so that the cubes $Q(x_k; u_k) = x_k + [0, u_k]^d$ are disjoint for $0 \leq k \leq m$. Since $\mathtt E^{\ast} \subseteq [0,1]^d$, 
the choice of $u$ and $x$ shows that $\mathtt K^{\ast}$ is the disjoint union of $(m+1)$ affine copies of the compact set $\mathtt E^{\ast}$. As a result, 
\[ \dim_{\mathbb H}(\mathtt K^{\ast}) = \max_{0 \leq k \leq m} \dim_{\mathbb H}(u_k\mathtt E^{\ast} + x_k) = \dim_{\mathbb H}(\mathtt E^{\ast}) = s. \] 
This establishes the statement about the size of $\mathtt K^{\ast}$. 
\vskip0.1in
\item We turn now to the statement about the distance set of $\mathtt K^{\ast}$. Let us set $\mathtt I_n := q^{-n \sigma} [a_{\rho}, b_{\rho}]$. Applying \eqref{Falconer-example-distance-set}, we note that 
\begin{equation} \Delta(\mathtt K^{\ast}) \supseteq \bigcup_{k=0}^{m} \Delta(u_k \mathtt E^{\ast} + x_k)  \supseteq \bigcup_{n=0}^{\infty} \bigcup_{k=0}^{m} u_k \mathtt I_n \supseteq  \bigcup \bigl\{\mathtt I_n \cup u_k \mathtt I_n : n \geq 0, 1 \leq k \leq m \bigr\}. 
\label{Distance-E-star} \end{equation} 
Given the integer $m$ as in the statement of the lemma, we choose the scale parameters $\{u_k: 1 \leq k \leq m\}$ in the following way
\begin{equation} \label{choice-uk} 
u_1 = \frac{b_{\rho}}{a_{\rho}} q^{-\sigma}, \quad u_2 = u_1 \frac{b_{\rho}}{a_{\rho}} = \Bigl(\frac{b_{\rho}}{a_{\rho}} \Bigr)^2 q^{-\sigma}, \quad \ldots, \quad u_m = u_{m-1} \frac{b_{\rho}}{a_{\rho}} = \Bigl(\frac{b_{\rho}}{a_{\rho}} \Bigr)^m q^{-\sigma}.  \end{equation}  
The main claim concerning the intervals $\mathtt I_n$ is the following: 
\begin{equation} \label{In-claim} \mathtt I_{n+1} \cup u_1 \mathtt I_{n} \cup \cdots \cup u_m \mathtt I_{n} \cup \mathtt I_n \supseteq \text{convex hull$(\mathtt I_n \cup \mathtt I_{n+1})$} = \bigl[q^{-(n+1)\sigma}a_{\rho}, q^{-n\sigma}b_{\rho} \bigr]. \end{equation}
In other words, the intervals $\{u_k \mathtt I_{n} : 1 \leq k \leq m \}$ fill the gap between $\mathtt I_n$ and $\mathtt I_{n+1}$, with the right endpoint of one interval lining up with the left endpoint of the next. Since the intervals $\bigl[q^{-(n+1)\sigma}a_{\rho}, q^{-n\sigma}b_{\rho} \bigr]$ are overlapping, combining \eqref{In-claim} and \eqref{Distance-E-star} leads to
\[ \Delta(\mathtt K^{\ast}) \supseteq \bigcup_{n=0}^{\infty} \bigl[q^{-(n+1)\sigma}a_{\rho}, q^{-n\sigma}b_{\rho} \bigr] = (0, b_{\rho}], \] 
which is the Steinhaus property asserted in the statement of the lemma. It remains to justify the claim \eqref{In-claim}. Let $\mathscr{L}(\mathtt I)$ and $\mathscr{R}(\mathtt I)$ denote the left and right endpoints respectively of an interval $\mathtt I$. The choice \eqref{choice-uk} of $a_{\rho}, b_{\rho}$, $q$ and $u_1$ ensure that
\[ \mathscr{R}(\mathtt I_{n+1}) = b_{\rho}q^{- (n+1) \sigma} = u_1 a_{\rho} q^{-n \sigma} = \mathscr{L}(u_1 \mathtt I_n) < a_{\rho} q^{-n \sigma} = \mathscr{L}(\mathtt I_n). \]
In general for $1 \leq k \leq m-1$, the assumption $q^{\sigma} > (b_{\rho}/a_{\rho})^{k+1}$ implies that  
\[
 \mathscr{R}(u_k \mathtt I_{n}) = u_k b_{\rho} q^{-n \sigma} = u_{k+1} a_{\rho} q^{-n \sigma}  =  \mathscr{L}(u_{k+1} \mathtt I_n) <  a_{\rho} q^{-n \sigma} = \mathscr{L}(\mathtt I_n). \]
However, the choice of $m$ ensures that $q^{\sigma} \leq (b_{\rho}/a_{\rho})^{m+1}$, as a result of which $\mathscr{R}(u_m \mathtt I_{n}) = u_m b_{\rho} q^{-n \sigma} \geq q^{-n \sigma} a_{\rho} = \mathscr{L}(\mathtt I_n)$. 
This completes the proof of the claim \eqref{In-claim}. 
\end{proof} 
\subsection{Sparse sets with all distances}  \label{Sparse-set-all-distances-proof-section} 
Given $d \geq 2$, let $\varepsilon_{\rho}, a_{\rho}, b_{\rho} \in (0,1)$ be as in Theorem \ref{mainthm-cor}. Fix $s \in (d - \varepsilon_{\rho}, d) \cap \mathbb Q$, and set $\sigma = d/s$. Choose a large integer $q$ as in Lemma \ref{Quasi-regular-Interval-Lemma}. In this section, we use the set $\mathtt E^{\ast} = \mathtt E^{\ast}[q;\sigma]$ as given by \eqref{Falconer-E} to prove Corollary \ref{Sparse-all-distances-corollary}.
Given a sequence $\{\mathtt r_j : j \geq 1\}$ of controlled growth, namely 
\begin{align} &\mathtt r_j \nearrow \infty,  \quad 1 \leq \frac{\mathtt r_{j+1}}{\mathtt r_j} \leq \frac{b_{\rho}}{a_{\rho}}, \quad \text{ we set } \label{growth-of-rj} \\ 
&A_j := \mathtt R_j \mathtt e_1 + \mathtt r_j \mathtt E^{\ast}, \quad \mathfrak A_{\mathtt J} = \bigsqcup_{j=1}^{\mathtt J} A_j, \quad \mathfrak A = \bigsqcup_{j=1}^{\infty} A_j. \label{A-AJ}
\end{align}   
The rapidly increasing constants $\mathtt R_j$ are chosen to satisfy 
\begin{equation} \label{Rj-growth} 
\mathtt R_{j+1} > C_d(\mathtt R_j + \mathtt r_j)  \text{ for a large dimensional constant $C_d > 2^{2d}$},  
\end{equation} 
so that the sets $A_{j}$ are mutually disjoint. 
\begin{lem} 
 The set $\mathfrak A$ in \eqref{A-AJ} obeys the following properties: 
\begin{enumerate}[(a)]
\item $\dim_{\mathbb H}(\mathfrak A) = s$. \label{dim-part}
\vskip0.1in
\item The set $\mathfrak A$ attains every distance in the complement of every compact set: namely,  
\begin{equation} \label{all-distances-counterexample}  
\Delta (\mathfrak A \setminus \mathfrak A_{\mathtt J}) = [0, \infty).
\end{equation} \label{all-distances-counterexample-part}
\item For every $\mathtt J \geq 1$, $\Delta(\mathfrak A_{\mathtt J})$ contains an interval gap whose length tends to infinity as $\mathtt J \rightarrow \infty$, i.e. there exist $a_{\mathtt J} < b_{\mathtt J}$ such that    
\begin{equation}   \mathbb R \setminus \Delta(\mathfrak A_{\mathtt J}) \supseteq [a_{\mathtt J}, b_{\mathtt J}], \quad (b_{\mathtt J} - a_{\mathtt J}) \rightarrow \infty \text{ as } \mathtt J \rightarrow \infty. \label{distances-gap} 
\end{equation}  \label{gap-distances-part}
\vskip0.1in  
\item The constituent blocks $A_j$ do not have the Steinhaus property for any $\mathtt J \geq 1$. \label{lack-of-Steinhaus} 
\vskip0.1in 
\item There exists $\mathtt J_0 \geq 1$ such that $A_1 \sqcup A_2 \sqcup \ldots \sqcup A_{\mathtt J_0}$ has the Steinhaus property.    \label{finite-union-Steinhaus}  
\end{enumerate} 
\vskip0.1in
In particular, Corollary \ref{Sparse-all-distances-corollary} holds, with $\mathfrak A$ as the supporting example.  
\end{lem} 
\begin{proof} 
Part (\ref{dim-part}) follows from the construction of $\mathfrak A$: 
\[ \dim_{\mathbb H}(\mathfrak A) = \sup_j \dim_{\mathbb H}(A_j) = \dim_{\mathbb H}(\mathtt E^{\ast}) = s.  \] 
Since each set $A_j$ is an affine copy of $\mathtt E^{\ast}$, part (\ref{lack-of-Steinhaus}) follows from Lemma \ref{Quasiregular-NoInterval-Lemma}. 
\vskip0.1in 
\noindent Let us turn to the statement \eqref{all-distances-counterexample} in part (\ref{all-distances-counterexample-part}).  We deduce from  \eqref{Falconer-example-distance-set} that
\begin{align*}
\Delta(\mathfrak A \setminus \mathfrak A_{\mathtt J}) &= \Delta \Bigl(\bigsqcup_{j=\mathtt J+1}^{\infty} A_j \Bigr) \supseteq \bigcup_{j=\mathtt J+1}^{\infty} \Delta(A_j) \supseteq \bigcup_{j=\mathtt J+1}^{\infty} \mathtt r_j \Delta(\mathtt E^{\ast})\\
&\supseteq \bigcup_{j = \mathtt J+1}^{\infty} \bigcup_{n=0}^{\infty} [\mathtt r_j a_{\rho}, \mathtt r_{j} b_{\rho}] q^{-n \sigma} \supseteq \bigcup_{n=0}^{\infty} [\mathtt r_{\mathtt J+1} a_{\rho}, \infty) q^{-n \sigma} = (0, \infty).
\end{align*}  
In the second line of the display above, we have used the growth condition \eqref{growth-of-rj} to deduce that the interval $\mathfrak I_j := [\mathtt r_j a_{\rho}, \mathtt r_{j} b_{\rho}]$ overlaps with $\mathfrak I_{j+1}$, so that the union of $\mathfrak I_j$ over $j \in \{\mathtt J+1, \mathtt J+2, \ldots\}$ is an infinite half line. 
\vskip0.1in
\noindent A similar intersection property also yields part (\ref{finite-union-Steinhaus}). Since $\mathtt r_j \rightarrow \infty$, we can choose $\mathtt J_0$ so that $\mathtt r_{\mathtt J_0}/\mathtt r_1 > q^{\sigma} {a_{\rho}}/{b_{\rho}}$.  This ensures that the intervals $[\mathtt r_1 a_{\rho}, \mathtt r_{\mathtt J_0} b_{\rho}] q^{-n \sigma}$ intersect for any two consecutive indices $n$, leading to:
\[ \Delta \bigl(\bigsqcup_{j=1}^{\mathtt J_0} A_j \bigr) \supseteq  \bigcup_{j=1}^{\mathtt J_0} \mathtt r_j \Delta(\mathtt E^{\ast}) \supseteq \bigcup_{j = 1}^{\mathtt J} \bigcup_{n=0}^{\infty} [\mathtt r_j a_{\rho}, \mathtt r_{j} b_{\rho}] q^{-n \sigma} =   \bigcup_{n=0}^{\infty} [\mathtt r_1 a_{\rho}, \mathtt r_{\mathtt J_0} b_{\rho}] q^{-n \sigma} = (0, \mathtt r_{\mathtt J_0}b_{\rho}]. \]   
\vskip0.1in 
\noindent It remains to prove part (\ref{gap-distances-part}).  We do this by analyzing the contributions of $A_j$ to $\Delta(\mathfrak A)$:
\begin{equation}  \Delta(\mathfrak A_{\mathtt J}) \subseteq \bigcup \bigl\{\Delta(A_j, A_{j'}) : 1 \leq j, j' \leq \mathtt J \bigr\}. \label{components-distance-set} \end{equation} 
The trivial inclusion yields $\Delta(A_j) = \mathtt r_j \Delta(\mathtt E^{\ast}) \subseteq  \mathtt r_j \Delta([0,1]^d)  = \mathtt r_j [0, \sqrt{d}]$. These intervals are nested and increasing, so their union in contained in $\mathtt r_{\mathtt J}[0, \sqrt{d}]$. 
On the other hand for $j > j'$, 
\begin{align}  
\Delta(A_j, A_{j'}) &= \{\bigl|(\mathtt R_j - \mathtt R_{j'})\mathtt e_1 + \mathtt r_{j}u - \mathtt r_{j'}u' \bigr| : u, u' \in \mathtt E^{\ast}\}  \nonumber \\ 
&\subseteq \mathtt R_{j} - \mathtt R_{j'} + 2\mathtt r_j [-\sqrt{d}, \sqrt{d}] =: \mathfrak J[j';j]. \label{off-diagonal-interval}
\end{align}   
The main observations concerning the intervals $\mathfrak J[j';j]$ are the following: 
\vskip0.1in 
\begin{itemize} 
\item For fixed $j$, the intervals $\{ \mathfrak J[j';j] : j' < j \}$ are not necessarily disjoint, but  they are all of the same length $2 \mathtt r_j \sqrt{d}$ and move to the right as $j'$ decreases. As a result, the leftmost and rightmost endpoints of their union are given respectively by 
\begin{align}
\min \bigl\{x : x \in \mathfrak J[j';j], \; j' < j \bigr\} &= \mathtt R_j - \mathtt R_{j-1} - 2 \mathtt r_j \sqrt{d}, \label{min} \\
\max \bigl\{x : x \in \mathfrak J[j';j], \; j' < j \bigr\} &= \mathtt R_j - \mathtt R_1 + 2 \mathtt r_j \sqrt{d}. \label{max}  
\end{align} 
\vskip0.1in 
\item The above statement implies that the union of the intervals $\{\mathfrak J[j; \mathtt J] : j < \mathtt J \}$ lies to the right of $\mathtt r_{\mathtt J}[0, \sqrt{d}]$, and also to the right of all intervals of the form $\{ \mathfrak J[j';j] : j' < j < \mathtt J\}$. More precisely, \eqref{min} and \eqref{max} say that if  
\begin{align} 
a_{\mathtt J} &:= \max(\mathtt r_{\mathtt J}\sqrt{d}, \mathtt R_{\mathtt J-1} - \mathtt R_1 + 2 \mathtt r_{\mathtt J-1}\sqrt{d}), \;  b_{\mathtt J} := \mathtt R_{\mathtt J} - \mathtt R_{\mathtt J-1} - 2 \mathtt r_{\mathtt J}\sqrt{d}, \text{ then } \nonumber \\ 
\bigcup_{j' < j < \mathtt J} &\mathfrak J[j'; j] \cup [0, \mathtt r_{\mathtt J} \sqrt{d}] \subseteq [0, a_{\mathtt J}), \text{ whereas }\bigcup_{j < \mathtt J}  \mathfrak J[j; \mathtt J] \subseteq [b_{\mathtt J}, \mathtt R_{\mathtt J} - \mathtt R_{1} + 2 \mathtt r_{\mathtt J}\sqrt{d}] \label{aJbJ}
  \end{align}     
\end{itemize} 
Combining \eqref{components-distance-set}, \eqref{off-diagonal-interval} and \eqref{aJbJ}, we find that $\mathbb R \setminus \Delta(\mathfrak A_{\mathtt J})$ contains the interval 
$[a_{\mathtt J}, b_{\mathtt J}]$.  The rapid growth condition \eqref{Rj-growth} implies the existence of a small dimensional constant $c_d > 0$ such that  
\begin{align*} 
b_{\mathtt J} - a_{\mathtt J} &\geq \mathtt R_{\mathtt J} - 2 \mathtt r_{\mathtt J} \sqrt{d} - \max \bigl[ \mathtt R_{\mathtt J-1} +  \mathtt r_{\mathtt J} \sqrt{d}, 2\mathtt R_{\mathtt J-1} + \mathtt R_1 + 2 \mathtt r_{\mathtt J-1} \sqrt{d} \bigr] \\ 
&\geq c_d \mathtt R_{\mathtt J} \rightarrow \text{ as } \mathtt J \rightarrow \infty, 
\end{align*} 
completing the proof.  
\end{proof}

\section{Sufficiently large distances in well-distributed sets} \label{Suff-large-distances-section} 
\subsection{Proof of Theorems \ref{mainthm-3} and \ref{mainthm-4}, assuming Theorems \ref{mainthm-2} and \ref{mainthm-cor}} 
\begin{proof} 
Let us fix $d \geq 2$, a set $A \subseteq \mathbb R^d$, and recall from \eqref{def-ARx} the definition of the normalized truncations $A_R(x)$.  If $\mathbb H^s(A_R(x)) > 0$, then Theorem \ref{mainthm-cor} implies that 
\[ \Delta(A) \supseteq R \Delta(A_R(x)) \supseteq R \bigcup \Bigl\{\ell(Q)[a_{\rho}, b_{\rho}] : Q \in \mathcal D_{\rho}\bigl(A_R(x); s\bigr) \Bigr\}, \]
provided $s > d - \varepsilon_{\rho}$. Here $\rho \in (0,1)$, and  $a_{\rho}, b_{\rho}, \varepsilon_{\rho}$ are the parameters provided by Theorem \ref{mainthm-2}. Since $A$ is given to be well-distributed with $s$-density at least $(1- \rho)$ and growth rate at most $b_{\rho}/a_{\rho}$, it follows from definitions \eqref{def-V} and \eqref{slower-than-lacunary} that there exists a sequence $v_n \nearrow \infty$ in $\mathscr V(\rho, s; A)$ such that  
\[ \frac{v_{n+1}}{v_n} \leq \frac{b_{\rho}}{a_{\rho}} \text{ for $n \geq N$, as a result of which } \Delta(A) \supseteq \bigcup_{n=N}^{\infty} \mathscr I_n  \supseteq  [v_N a_{\rho}, \infty), \]
as has been claimed in Theorem \ref{mainthm-3}. Here $\mathscr I_n := v_n[a_{\rho}, b_{\rho}]$. The last inclusion in the displayed sequence above  follows from \eqref{slower-than-lacunary}, which ensures for $n \geq N$, we have the string of inequalities $v_n a_{\rho} < v_{n+1} a_{\rho} \leq v_n b_{\rho} < v_{n+1} b_{\rho}$. Consequently,   
\[ \mathscr{I}_n \cup \mathscr{I}_{n+1} = \text{ convex hull}(\mathscr{I}_n \cup \mathscr{I}_{n+1}) = [v_na_{\rho}, v_{n+1}b_{\rho}]. \]   
\vskip0.1in
\noindent Theorem \ref{mainthm-4} is a special case of Theorem \ref{mainthm-3}, where the hypotheses permits the application of Theorem \ref{mainthm-2} instead of Theorem \ref{mainthm-cor} on $A_{R_n}(x)$, with $\ell = 1$.  In this case, the condition \eqref{slower-than-lacunary} reduces to \eqref{Rn-growth}. 
\end{proof}

\section{Structure of distance sets near the origin} \label{Mainthm1-Proof-Section}

The goal of this section is to prove Theorem \ref{mainthm-cor}, assuming Theorem \ref{mainthm-2}. 
\begin{lem} \label{density-lemma}
Fix any $\rho \in (0,1)$. For any Borel set $E \subseteq \mathbb R^d$ with $\mathbb H^s(E) > 0$ and any $\rho \in (0,1)$, there exists $Q \in \mathcal Q_d$ such that
\begin{equation} \mathcal H^s_{\infty}(E \cap Q) > (1- \rho) \ell(Q)^s.  \label{density-claim} \end{equation}
Thus the collection $\mathcal D_{\rho}(E;s)$ defined by \eqref{high-density-on-cube} is non-empty for sets $E$ with $\mathbb H^s(E) > 0$. 
\end{lem} 
\begin{proof} Since $\mathbb H^s(E)$ is positive by assumption, so is $\mathcal H^s_{\infty}(E)$. Set $\mathcal H^s_{\infty}(E) := \varrho > 0$. For $\kappa > 0$, let $\{Q_i = Q_i(\kappa) : i \geq 1\}$ denote a countable collection of dyadic cubes covering $E$ with    
\begin{equation} \label{inf-condition} \sum_i \ell(Q_i)^s \leq \varrho(1 + \kappa). \end{equation} 
One of the covering cubes $Q_i = Q_i(\kappa)$ must obey \eqref{density-claim} for small $\kappa$; otherwise \eqref{inf-condition} would lead to the inequality  
 \[ \varrho = \mathcal H_{\infty}^s(E) \leq \sum_i \mathcal H_{\infty}^s(E \cap Q_i) \leq (1 - \rho) \sum_i \ell(Q_i)^s \leq \varrho (1 - \rho)(1 + \kappa), \]
 which is a contradiction for all sufficiently small $\kappa \leq \rho$.  
 \end{proof} 
\subsection{Proof of Theorem \ref{mainthm-cor}, assuming Theorem \ref{mainthm-2}} \label{mainthm-cor-proof-section}
\begin{proof} 
Let $\mathcal D_{\rho}(E)$ be the collection of dyadic cubes defined in \eqref{high-density-on-cube}. 
For any $Q \in \mathcal D_{\rho}(E)$ given by $Q = c(Q) + \ell(Q) [0,1]^d$, let 
\[ \mathbb T_{Q}: \mathbb R^d \rightarrow \mathbb R^d, \quad \mathbb T_Q(y) := \frac{y-c(Q)}{\ell(Q)} \] denote the linear transformation that maps $Q$ onto $[0,1]^d$. Since $\mathbb T_{Q}$ maps the collection $\mathcal Q_d$ of all closed dyadic cubes onto itself, every dyadic covering of $E \cap Q$ corresponds uniquely to a dyadic covering of $\mathbb T_Q(E \cap Q)$ via the linear map $\mathbb T_Q$. Invoking the translation-invariance and scaling properties of dyadic Hausdorff content, \eqref{high-density-on-cube} gives rise to the following inequality:
\[ \mathcal H^s_{\infty}\bigl(\mathbb T_{Q}(E \cap Q) \bigr) = \ell(Q)^{-s} \mathcal H_{\infty}^s(E \cap Q) \geq 1 - \rho. \] 
Thus the hypotheses of Theorem \ref{mainthm-2} hold with both $E$ and $F$ in that theorem replaced by $\mathbb T_{Q}(E \cap Q)$. Applying the theorem leads to the conclusion that $\Delta(\mathbb T_{Q}(E))$ contains the interval $[a_{\rho}, b_{\rho}]$, i.e. the conclusion \eqref{Delta-intervals} of Theorem \ref{mainthm-cor} holds.
\end{proof}

\section{An integral that identifies distances} \label{configuration-integral-section}
\noindent Given two compact sets $E, F \subseteq [0,1]^d$, we will determine the distances generated by $E$ and $F$ using a certain linear functional, as in \cite{1986Bourgain}. Such functionals are often called configuration integrals. In this section, we will set up the configuration integral that is relevant for identifying a specific distance $t$, and discuss its properties. 
\vskip0.1in
\noindent Given probability measures $\mu, \nu$ supported on the compact sets $E, F \subseteq [0,1]^d$ respectively, and any $t > 0$, we define a linear functional $\Lambda_{\mu, \nu}(t)$ as follows:  for any continuous function $f \in C\bigl([0,1]^d \times [0,1]^d \bigr)$, 
\begin{equation} \label{def-Lambda}
\langle \Lambda_{\mu, \nu}(t), f \rangle  := \lim_{\kappa \rightarrow 0} \int \mu_{\kappa}(x) \nu_{\kappa}(x + t \omega) f(x, x+ t\omega) dx d\sigma(\omega),  
\end{equation} 
with the notational convention $\Lambda_{\mu} := \Lambda_{\mu, \mu}$. Here $\sigma$ denotes the normalized surface measure on $\mathbb S^{d-1}$, and $\mu_{\kappa} := \mu \ast \varphi_{\kappa}$ for an approximate identity sequence $\{\varphi_{\kappa} : \kappa > 0\}$, $\varphi_{\kappa}(x) = \kappa^{-d} \varphi(x/\kappa)$, where $\varphi$ is a non-negative Schwartz function with $\widehat{\varphi}(0) = 1$. The following proposition verifies that under certain conditions, $\Lambda_{\mu, \nu}(t)$ is well-defined as a bounded linear functional and an effective instrument for identifying distances realized by joining points in $E$ with points in $F$.
\vskip0.1in
\noindent Let us recall from \cite[Chapters 2 and 3, Theorems 2.8 and 3.10]{Mattila-Book2} the definition of $s$-energy of a measure $\mu$ in $\mathbb R^d$: for $0 < s \leq d$, 
\begin{equation} \label{def-s-dim-energy}
\mathbb I_\mu(s) := \iint |x-y|^{-s} d\mu(x) d\mu(y) = \gamma(d, s) \int \bigl| \widehat{\mu}(\xi) \bigr|^{2} |\xi|^{s-d} d\xi
\end{equation} 
where $\gamma(d, s) = \pi^{s - d/2} \Gamma(\frac{d-s}{2})/\Gamma(\frac{s}{2})$ is a positive finite constant. 
\begin{prop} \label{Lambda-prop}
Suppose that $\mu$ and $\nu$ are two probability measures, supported on the compact sets $E$ and $F$ respectively, each wih finite $(d+1)/2$-energy: 
\begin{equation} \mathbb I_\mu ((d+1)/2) < \infty, \quad  \mathbb I_\nu ((d+1)/2) < \infty. \label{finite-energy-conditions} \end{equation} 
Then the following conclusions hold. 
\begin{enumerate}[(a)] 
\item For every continuous function $f$ on $[0,1]^{2d}$, the functional $\Lambda_{\mu, \nu}(t)$ is well-defined; in other words, the limit in \eqref{def-Lambda} exists and is independent of the choice of the approximate identity $\{\varphi_\varepsilon : \varepsilon > 0 \}$.  Moreover, there exists an absolute constant $C_d$ depending only on $d$ such that
\begin{equation}  \bigl| \langle \Lambda_{\mu, \nu}(t), f \rangle \bigr| \leq C_d t^{-\frac{d-1}{2}} \bigl[\mathbb I_\mu ((d+1)/2) \times  \mathbb I_\nu ((d+1)/2) \bigr]^{\frac{1}{2}} ||f||_{\infty}. \label{op-norm} \end{equation} 
Thus $\Lambda_{\mu, \nu}(t)$ can be identified as a non-negative finite measure.   \label{existence-part}
\vskip0.1in
\item If the measure $\Lambda_{\mu, \nu}(t)$ is non-trivial (i.e., has positive mass), it is supported on 
\begin{equation} 
X_{\mu}(t) := \bigl\{(x, y) \in E \times F : |x - y| = t \bigr\}. 
\end{equation} \label{support-part}
In this case, the total mass of $\Lambda_{\mu, \nu}(t)$ is given by 
\begin{equation}  \label{mass-Lambda-mu-nu} 
||\Lambda_{\mu, \nu}(t)|| = \int \widehat{\mu}(\xi) \overline{\widehat{\nu}(\xi)} \widehat{\sigma}(t \xi) \, d\xi > 0. 
\end{equation}  
Conversely, if the integral in \eqref{mass-Lambda-mu-nu} s strictly positive, then $\Lambda_{\mu, \nu}(t)$ is a non-trivial measure, and hence its support $X_{\mu}(t)$ is non-empty; in other words, $t \in \Delta(E, F)$. 
\end{enumerate} 
\end{prop} 
\subsection{Existence of $\Lambda_{\mu, \nu}$} 
\begin{lem} \label{Iatmu-lemma}
There exists an absolute positive constant $C_d$ as follows. Given $a \in \mathbb R^d$, $t \in (0,1)$, $\beta \geq 0$, and a probability measure $\mu$, the integral $\mathscr{I}(a, t; \mu)$ defined by 
\begin{equation} 
\mathscr{I}(a, t; \mu) := \int \bigl| \widehat{\mu}(\xi) \bigr|^2 \, \bigl( 1 + t |a + \xi|\bigr)^{-\beta} \, d\xi \label{def-Iatmu} 
\end{equation} 
obeys the bound 
\begin{equation} \label{Iat-bound} 
\bigl| \mathscr{I}(a, t; \mu) \bigr| \leq C_d \bigl[|a|^d + t^{-\beta}\mathbb I_{\mu}(d-{\beta}) \bigr].
\end{equation} 
\end{lem} 
\begin{proof} 
Let us write $\mathscr{I} = \mathscr{I}_1 + \mathscr{I}_2$, where $\mathscr{I}_1$ represents the integral over $\{\xi: |a + \xi| \leq |\xi|/2\}$. Since $|\xi| \leq |\xi + a| + |a| \leq |\xi|/2 + |a|$,  the domain of integration of $\mathscr{I}_1$ is contained in $\{\xi: |\xi| \leq 2|a|\}$. The integrand is trivially bounded above by 1. These observations lead to 
\begin{equation} \label{I1at-bound} 
\mathscr{I}_1(a, t; \mu) := \int_{|a + \xi| \leq |\xi|/2} \bigl| \widehat{\mu}(\xi) \bigr|^2 \, \bigl( 1 + t |a + \xi|\bigr)^{-\beta} \, d\xi \leq \int_{|\xi| \leq 2 |a|} \bigl| \widehat{\mu}(\xi) \bigr|^2 d\xi \leq  (2 |a|)^d. 
\end{equation}  
The remaining integral $\mathscr{I}_2$, whose domain of integration is $|a + \xi| \geq |\xi|/2$, is estimated as follows, 
\begin{align}
\mathscr{I}_{2}(a, t;\mu) &:= \int_{|a + \xi| \geq |\xi|/2}  \bigl| \widehat{\mu}(\xi) \bigr|^2 \, \bigl( 1 + t |a + \xi|\bigr)^{-\beta} \, d\xi \nonumber \\ 
&\leq \int \bigl| \widehat{\mu}(\xi) \bigr|^2 \bigl(t|\xi|/2\bigr)^{-\beta} \, d\xi \leq C_d t^{-\beta} \mathbb I_{\mu}(d - \beta). \label{I2at-bound} 
\end{align} 
Combining \eqref{I1at-bound} and \eqref{I2at-bound} leads to \eqref{Iat-bound}.
\end{proof} 
\begin{lem} \label{limit-lemma}
Under the finite energy assumptions \eqref{finite-energy-conditions} and for any smooth function $f$  supported on $[0,1]^d \times [0,1]^d$, the limit in \eqref{def-Lambda} exists and is independent of the choice of $\varphi$ in the approximate identity $\bigl\{\varphi_{\kappa}: \kappa > 0 \bigr\}$.
\end{lem}
\begin{proof} 
Given a smooth function $f$ of compact support, let us define, 
\begin{align*} \Lambda_{\mu, \nu}^{\ast}(t; f, \kappa) &:= \int \mu_{\kappa}(x) \nu_{\kappa}(x + t \omega) f(x, x+ t\omega) dx d\sigma(\omega) \\ 
&=  \iiint \widehat{f}\bigl(\xi, \eta \bigr) \, \widehat{\mu}_{\kappa}(-\xi - \eta - \zeta) \, \widehat{\nu}_{\kappa}(\zeta) \, \widehat{\sigma}(t (\eta + \zeta)) \, d\xi \, d\eta \, d\zeta, 
\end{align*}
where the last step follows from Fourier inversion. We will derive the existence of the limit in \eqref{def-Lambda} and its independence of $\varphi$ from the dominated convergence theorem. In view of 
\[ \widehat{\mu}_{\kappa}(\xi) = \widehat{\mu}(\xi) \widehat{\varphi}(\kappa \xi) \longrightarrow \widehat{\mu}(\xi) \text{ as } \kappa \rightarrow 0, \quad \text{ and } \quad \bigl| \widehat{\mu}_{\kappa}(\xi) \bigr| \leq ||\widehat{\varphi}||_{\infty} \bigl|\widehat{\mu}(\xi) \bigr|, \] with analogous relations for $\nu$, and the well-known bound for the Fourier transform of the spherical measure: \[|\widehat{\sigma}(\xi)| \leq C_d (1 + |\xi|)^{-\frac{(d-1)}{2}},\]
we are led to verify the absolute convergence of the following integral: 
\begin{align} 
\mathscr{I}^{\ast} :=& \iint \bigl|\widehat{f}(\xi, \eta) \bigr| \Bigl[ \int  \bigl|\widehat{\nu}(\zeta) \bigr| \, \bigl| \widehat{\mu}(-\xi - \eta - \zeta) \bigr| \bigl(1 + t|\eta + \zeta|)^{-\frac{d-1}{2}} \, d\zeta \Bigr] \, d\xi \, d\eta \nonumber \\ 
  =&  \iint \bigl|\widehat{f}(\xi, \eta) \bigr| \mathscr{J}(\xi, \eta) \, d\xi \, d\eta. \label{Istar-final}
  \end{align}
  The inner integral $\mathscr{J}(\xi, \eta)$ can be estimated, using the Cauchy-Schwarz inequality, as follows:
  \begin{align} 
  \mathscr{J}(\xi, \eta) := &   \int  \bigl|\widehat{\nu}(\zeta) \bigr| \, \bigl| \widehat{\mu}(-\xi - \eta - \zeta) \bigr| \bigl(1 + t|\eta + \zeta|)^{-\frac{d-1}{2}} \, d\zeta \nonumber \\ 
  \leq & \Bigl[ \int \bigl|\widehat{\nu}(\zeta) \bigr|^2 \bigl(1 + t|\eta + \zeta|)^{-\frac{d-1}{2}} \, d\zeta  \Bigr]^{\frac{1}{2}} \Bigl[\bigl| \widehat{\mu}(-\xi - \eta - \zeta) \bigr|^2 \bigl(1 + t|\eta + \zeta|)^{-\frac{d-1}{2}} \, d\zeta \Bigr]^{\frac{1}{2}} \nonumber \\
  \leq & \Bigl[ \int \bigl|\widehat{\nu}(\zeta) \bigr|^2 \bigl(1 + t|\eta + \zeta|)^{-\frac{d-1}{2}} \, d\zeta  \Bigr]^{\frac{1}{2}} \Bigl[\bigl| \widehat{\mu}(\zeta) \bigr|^2 \bigl(1 + t|\xi + \zeta|)^{-\frac{d-1}{2}} \, d\zeta \Bigr]^{\frac{1}{2}} \nonumber \\
  \leq & \bigl[\mathscr{I}(\eta, t; \nu) \bigr]^{\frac{1}{2}} \times \bigl[\mathscr{I}(\xi, t; \mu) \bigr]^{\frac{1}{2}} \nonumber \\ 
  \leq &C_d \bigl[|\xi|^d + t^{-\frac{d-1}{2}} \mathbb I_{\mu} ((d+1)/2)\bigr]^{\frac{1}{2}} \times \bigl[|\eta|^d + t^{-\frac{d-1}{2}} \mathbb I_{\nu} ((d+1)/2)\bigr]^{\frac{1}{2}}. \label{J-bound}
\end{align} 
The integral $\mathscr{I}$ used in the penultimate step refers to \eqref{def-Iatmu}, and we have used its bound from Lemma \ref{Iatmu-lemma} at the last step. Also, the final expression above is finite in view of \eqref{finite-energy-conditions}. Since $\widehat{f}(\xi, \eta)$ is a Schwartz function, inserting \eqref{J-bound} into \eqref{Istar-final} establishes convergence of the integral $\mathscr{I}^{\ast}$.   
\end{proof} 
\subsection{Proof of Proposition \ref{Lambda-prop}} 
\begin{proof} 
 The existence of the limit in \eqref{def-Lambda} and its independence with respect to $\varphi$ follow from Lemma \ref{limit-lemma}.  The non-negativity of $\mu_{\kappa}$ and $\nu_{\kappa}$ imply that
\begin{equation}  \label{op-norm-comp} 
 \bigl| \langle \Lambda_{\mu, \nu}(t), f \rangle \bigr| \leq ||f||_{\infty} \times \langle \Lambda_{\mu, \nu}(t), \mathtt 1 \rangle,  \end{equation} 
where $\mathtt 1$ denotes the constant function that is identically 1 on $[0,1]^{2d}$. Using Fourier inversion as in Lemma \ref{limit-lemma}, we estimate the latter quantity:   
\begin{align}
\langle \Lambda_{\mu, \nu}(t), \mathtt 1 \rangle &= \lim_{\kappa \rightarrow 0}  \int \mu_{\kappa}(x) \nu_{\kappa}(x + t \omega) dx d\sigma(\omega) = \lim_{\kappa \rightarrow 0} \int \widehat{\mu}_{\kappa}(\xi) \widehat{\nu}_{\kappa}(-\xi) \widehat{\sigma}(t \xi) \, d\xi \nonumber \\ 
&\leq \int \bigl|\widehat{\mu}(\xi)\bigr| \, \bigl| \widehat{\nu}(\xi)| (1 + t|\xi|)^{-\frac{d-1}{2}} \, d\xi \leq \bigl[ \mathscr{I}(0, t;\mu) \bigr]^{\frac{1}{2}} \times \bigl[ \mathscr{I}(0, t;\nu) \bigr]^{\frac{1}{2}} \nonumber \\
&\leq C_d t^{-\frac{d-1}{2}} \Bigl[\mathbb I_{\mu}\Bigl(\frac{d+1}{2}\Bigr) \times \mathbb I_{\nu}\Bigl(\frac{d+1}{2} \Bigr)\Bigr]^{\frac{1}{2}}.  \label{op-norm-est-final}
\end{align}
As before, we have applied the Cauchy-Schwarz inequality in the penultimate step, with $\mathscr{I}$ as in \eqref{def-Iatmu}. Inserting \eqref{op-norm-est-final} into \eqref{op-norm-comp} establishes \eqref{op-norm}, proving part (\ref{existence-part}). 
\vskip0.1in
\noindent It is not difficult to deduce that the support of the measure given by $\Lambda_{\mu, \nu}$, if non-trivial, has to be contained in $E \times F$; indeed, for any continuous $f$ with 
\[ \text{supp}(f) \subseteq [0,1]^{2d} \setminus (E \times F),\; \exists \kappa_0 > 0 \text{ such that dist}(\text{supp}(f), E \times F) > \kappa_0. \] This means that $\mu_{\kappa}(x) \nu_{\kappa}(y) = 0$ for all $(x, y) \in \text{supp}(f)$ and for all $\kappa < \kappa_0$. As a result, the integrand in \eqref{def-Lambda} vanishes  for all sufficiently small $\kappa$, and hence so does the limit. Similarly, if $f^{\ast}$ is any continuous function of compact support with 
\begin{align*} 
&\text{supp}(f^{\ast}) \subseteq \{(x,y) \in [0,1]^{d} \times [0,1]^d : |x - y| \ne t \}, \text{ then } \\ 
&f^{\ast}(x, x+t\omega)  \equiv 0 \text{ for every } x  \in [0,1]^d \text{ and } \omega \in \mathbb S^{d-1}. 
\end{align*} 
Once again, the integrand of \eqref{def-Lambda} vanishes identically. This shows that supp$(\Lambda_{\mu, \nu}) \subseteq \{(x, y) : |x - y| = t \}$.  Finally, since $\Lambda_{\mu, \nu}$ is a non-negative measure by definition, its total mass is given by 
\[ ||\Lambda_{\mu, \nu}(t)|| = \langle \Lambda_{\mu, \nu}, 1\rangle = \lim_{\kappa \rightarrow 0}  \int \mu_{\kappa}(x) \nu_{\kappa}(x + t \omega) dx d\sigma(\omega) =  \int \widehat{\mu} (\xi) \widehat{\nu}(-\xi) \widehat{\sigma}(t \xi) \, d\xi, \] 
where the last step follows from the dominated convergence theorem, as has been justified in the steps leading up to \eqref{op-norm-est-final}. This completes the proof of part \eqref{support-part}, and hence the proposition.     
\end{proof} 

\section{Role of energy and spectral gap in creating large distance sets} \label{Energy-and-Spectral-Gap-section}
The proof of Theorem \ref{mainthm-2} rests on a few key propositions. In this section, we will formulate these propositions and prove Theorem \ref{mainthm-2} assuming them. The propositions, to be proved in subsequent sections, are of two distinct types. The first type, represented by Proposition \ref{mainprop-1}, shows that if a measure obeys a certain energy condition, and if its Fourier transform is quantifiably small inside a large annulus in frequency space (a condition that is called ``spectral gap'' in \cite{Kuca-Orponen-Sahlsten}), then the distance set of its support contains an interval depending on the radii of the annulus.  The second type, given by Proposition \ref{mainprop-2}, complements the first, showing that measures obeying these two criteria are guaranteed for sets with large dyadic Hausdorff content. Propositions \ref{mainprop-1} and \ref{mainprop-2} pertain to a distance set $\Delta(E)$ for a single set $E$. Some changes needed to generalize these statements to $\Delta(E, F)$ for two distinct sets $E, F$. These are contained in Propositions \ref{mainprop-1'} and \ref{mainprop-2'} respectively. 

\subsection{Finite energy + spectral gap $\implies \Delta(E)$ contains an interval} 
Before proceeding to the statement, let us recall the definitions of the partial $L^2$ Fourier integral $\mathbb F_{\mu}$ and energy integral $\mathbb I_{\mu}$ from \eqref{def-F} and \eqref{def-s-dim-energy} respectively.  
\begin{prop} \label{mainprop-1} 
For any dimension $d \geq 2$, there exists an absolute positive constant $c_d$ with the following property. 
\vskip0.1in
\noindent Let us fix any choice of parameters 
\begin{itemize} 
\item $C_0 \geq 1$, 
\item $0 < \alpha \leq \frac18 < 1 < N$,
\item small constants $a, b, \delta \in (0,\frac14)$ with   
\begin{equation}  N\alpha > (d-1)/2, \quad 2b \leq \delta, \quad a < b \leq c_d,  \quad C_0 b^{N\alpha - \frac{d-1}{2}} < c_d. \label{abdN}  \end{equation} 
\end{itemize} 
Then any probability measure $\mu$ that is supported inside $[0,1]^d$ and meets the two conditions  
\begin{align}
&\mathbb I_{\mu}\Bigl(\frac{d+1}{2} + \alpha \Bigr) \leq C_0,  \qquad \text{ (finite energy) } \label{finite-energy} \\
&\mathbb F_{\mu}(a^{-N}) - \mathbb F_{\mu}({\delta}/{b})
\leq a \qquad \text{ (spectral gap)} \label{spectral-gap}
\end{align}
must necessarily obey 
\[ ||\Lambda_{\mu}(t)|| \geq c_d > 0 \text{ for every } t \in [a, b], \]   
with $\Lambda_{\mu}$ as in \eqref{def-Lambda}. In particular 
\[ \Delta\bigl(E\bigr) \supseteq [a, b] \quad \text{ where $E$ denotes the support of $\mu$.}  \]
\end{prop}
\vskip0.1in 
\noindent {\em{Remark: }} The absolute constant $c_d$ is Proposition \ref{mainprop-1} can be chosen to be 
\begin{equation}  c_d = \frac{\omega_d}{2^{d+4} d^{\frac{d}{2}}}, \quad \text{ where $\omega_d$ = volume of the unit ball in $\mathbb R^d$} = \frac{\pi^{\frac{d}{2}}}{\Gamma(\frac{d}{2}+1)}. \label{def-cd} \end{equation}  
\vskip0.1in 
\noindent A small variation of Proposition \ref{mainprop-1} yields the following conclusion for distance sets $\Delta(E, F)$. Let us fix a smooth  function $\psi$ with compact support such that  
\begin{equation} 
\psi \geq 0, \quad 0 \leq \widehat{\psi} \leq 1, \quad \widehat{\psi}(\xi) \equiv 1 \text{ on } B(0;1), \quad \text{supp}(\widehat{\psi}) \subseteq B(0;2). \label{psi-conditions} 
\end{equation}  
Set $\psi_{\mathtt e}(x) := \mathtt e^{-d} \psi(x/\mathtt e)$, and $\mu_{\mathtt e} := \mu \ast \psi_{\mathtt e}$ for a measure $\mu$. 
\begin{prop}\label{mainprop-1'} 
For $d \geq 2$, there exists a dimensional constant $c_d > 0$ as follows.  
\vskip0.1in 
\noindent Let $\mu$ and $\nu$ be two probability measures supported inside $[0,1]^d$ obeying the finite energy and spectral gap conditions; in other words, we assume conditions \eqref{finite-energy} and \eqref{spectral-gap} hold for both $\mu$ and $\nu$ with the same parameters $a, b, \delta$ and $C_0$. Additionally, suppose that 
\begin{equation} 
\bigl|| \Lambda_{\mu_{\mathtt e}, \nu_{\mathtt e}}(t) \bigr|| \geq 4c_d \quad \text{ for all } t \in [a, b] \text{ and } \mathtt e = \frac{t}{\delta},  \label{third-condition-mu-nu}
\end{equation} 
with $\Lambda_{\mu, \nu}$ as in \eqref{def-Lambda} and the same constants $a, b$ as in Proposition \ref{mainprop-1}. Then 
\begin{equation} \label{mainprop-1'-conclusion} \Lambda_{\mu, \nu}(t) \geq c_d > 0 \text{ for every } t \in [a, b]. 
\end{equation}    In particular, $\Delta \bigl(E, F\bigr) \supseteq [a, b]$, where $E$ and $F$ denote the supports of $\mu$ and $\nu$ respectively.    
\end{prop}
\noindent  Proofs of Propositions \ref{mainprop-1} and \ref{mainprop-1'} are given in Section \ref{mainprop-1-section}. 
\subsection{Large Hausdorff content for $s$ close to $d$ $\implies$ finite energy + spectral gap}  
\begin{prop}\label{mainprop-2}
For every selection of the parameters
\vskip0.1in
\begin{itemize} 
\item $d \geq 2$, $\kappa \in (0,1)$, $0 < \alpha \leq \frac18 < 1 < N$ with $N \alpha > (d-1)/2$, 
\end{itemize}
\vskip0.1in  
there is a choice of constants $0 < c_0 <  1 < C_0$ and $\varrho_0 \in (0,1)$ with the following property. For every $\rho \in (0, \varrho_0]$, one can find positive constants $a, b, \delta, \varepsilon$ depending on $\rho$ and the above parameters such  that \eqref{abdN} holds. Further, for $s > d - \varepsilon$, 
\[ \text{ any Borel subset $E \subseteq [0,1]^d$ with $\mathcal H_{\infty}^{s}(E) \geq 1 - \rho$ } \] 
supports a probability measure $\mu$ satisfying 
\vskip0.1in 
\begin{itemize} 
\item the finite energy condition \eqref{finite-energy} with parameter $C_0$ and 
\vskip0.1in 
\item the spectral gap condition \eqref{spectral-gap} with parameters $a, b, \delta$ and $N$. 
\end{itemize}
\vskip0.1in  
Specifically, the conditions \eqref{abdN}, \eqref{finite-energy} and \eqref{spectral-gap} can be realized with the choice 
\begin{equation}  a = \rho^{c_0}, \quad \delta = a^{\kappa/2}, \quad b = \delta a^{\kappa/2} = a^{\kappa}, \quad \varepsilon = \frac{c_0 \rho}{\log(1/\rho)} \text{ for all $\rho \in (0, \varrho_0]$.} \label{relation-abrho} \end{equation} 
The constant $C_0$ can be chosen to depend only on $d$, whereas $c_0$ is given in terms of $d$ and $N$. The parameter $\varrho_0$ depends on $d, N, \alpha$ and $\kappa$.  
\end{prop} 

\begin{prop} \label{mainprop-2'} 
Given two Borel sets $E, F \subseteq [0,1]^d$ with $\min\{\mathcal H_{\infty}^{s}(E), \mathcal H_{\infty}^{s}(F)\} \geq 1 - \rho$, the measures $\mu$ (on $E$) and $\nu$ (on $F$) constructed in Proposition \ref{mainprop-2} can be chosen to satisfy the additional condition \eqref{third-condition-mu-nu}.  
\end{prop}
\noindent Proofs of Propositions \ref{mainprop-2} and \ref{mainprop-2'} are given in Section \ref{mainprop-2-section}.
\subsection{Proof of Theorem \ref{mainthm-2}, assuming Propositions \ref{mainprop-1}-\ref{mainprop-2'}} \label{mainthm-2-proof-section} 
\begin{proof} 
Part \eqref{mainthm2-parta} of Theorem \ref{mainthm-2} follows from part \eqref{mainthm2-partb}, in the special case $\rho = \rho_d$, with $\bar{a}_d = a_{\rho}$, $\bar{b}_d = b_{\rho}$ and $\bar{\varepsilon}_d = \varepsilon_{\rho}$. Theorem \ref{mainthm-2}\eqref{mainthm2-partb} is a direct consequence of the two sets of results just stated, Propositions \ref{mainprop-1}, \ref{mainprop-1'} on one hand, Propositions \ref{mainprop-2}, \ref{mainprop-2'} on the other. The conclusions of the latter match the assumptions of the former. 
\vskip0.1in
\noindent Given $d \geq 2$ and $\kappa \in (0,1)$, and any choice of positive parameters $\alpha, N$ with  
\[N \alpha > (d-1)/2, \quad \text{for example } \quad \alpha = \frac18 \text{ and }  N = 8(d-1), \] 
let $c_0, C_0, \varrho_0$ be the constants ensured by Proposition \ref{mainprop-2}. Since $N$ depends only on $d$, Proposition \ref{mainprop-2} asserts that both $c_0$ an $C_0$ are dimensional constants, whereas $\varrho_0$ depends additionally on $\kappa$. For $\rho \in (0, \varrho_0]$, Propositions \ref{mainprop-2} and \ref{mainprop-2'} also provide a constant $\varepsilon > 0$ depending on $\rho$ and the above parameters with the following property: 
any two sets $E$ and $F$ with   
\[ E, F \subseteq [0,1]^d, \quad \mathcal H^s_{\infty}(E), \mathcal H^s_{\infty}(F) \geq 1 - \rho, \, s > d - \varepsilon \] 
support probability measures $\mu, \nu$ respectively, obeying the criteria \eqref{abdN}, \eqref{finite-energy}, \eqref{spectral-gap} and \eqref{third-condition-mu-nu}. The constants $a, b, \delta$ appearing in these three statements are also provided by Propositions \ref{mainprop-2}, \ref{mainprop-2'}, and in particular by the relations \eqref{relation-abrho}. The criteria \eqref{abdN}, \eqref{finite-energy}, \eqref{spectral-gap} and \eqref{third-condition-mu-nu} are precisely the assumptions of Proposition \ref{mainprop-1'}; one therefore concludes from it that $\Delta(E, F) \supseteq [a, b]$, as claimed in the theorem. 
\end{proof}

\section{Proofs of Propositions \ref{mainprop-1} and \ref{mainprop-1'}} \label{mainprop-1-section}
\subsection{Proof of Proposition \ref{mainprop-1}}  
Let us fix the parameters $C_0, \alpha, N, a, b$ as in the statement of the proposition. Let us also recall the Fourier representation of $\Lambda_{\mu}(t)$ given by \eqref{mass-Lambda-mu-nu}, with $\mu = \nu$. Choosing $\delta \in \bigl(2b, {1}/{4} \bigr)$ and $t \in [a, b]$, we decompose this expression for $||\Lambda_{\mu}(t)||$ as follows, 
\begin{equation} \label{Lambda-decomp} 
||\Lambda_{\mu}(t)|| = \int \bigl| \widehat{\mu}(\xi) \bigr|^2 \widehat{\sigma}(t \xi) \, d \xi = \mathfrak I_1(t) + \mathfrak I_2(t) + \mathfrak I_3(t), \end{equation} 
where $\mathfrak I_j(t)$ has the same integrand as $\Lambda_{\mu}(t)$ but is evaluated on the subdomain $\mathfrak D_j$: 
\[ \mathfrak D_1 := \bigl\{ \xi \in \mathbb R^d: t|\xi| \leq \delta \}, \; \mathfrak D_2 := \bigl\{ \xi \in \mathbb R^d: \delta < t|\xi| \leq t^{-N+1} \}, \; \mathfrak D_3 := \bigl\{ \xi \in \mathbb R^d: t|\xi| > t^{-N+1} \}. \] 
We estimate these three integrals as follows.
\begin{align} 
\bigl| \mathfrak I_1(t) \bigr| &= \Bigl| \int_{\mathfrak D_1} \bigl| \widehat{\mu}(\xi) \bigr|^2 \widehat{\sigma}(t \xi) \, d \xi \Bigr| \nonumber \\ 
&\geq \int_{\mathfrak D_1} \bigl| \widehat{\mu}(\xi) \bigr|^2 \, d \xi - \int_{\mathfrak D_1} \bigl| \widehat{\mu}(\xi) \bigr|^2 \bigl|1 - \widehat{\sigma}(t \xi)\bigr| \, d \xi \nonumber \\
&\geq \mathbb (1 - \delta) \mathbb F_{\mu}(\delta t^{-1}) \geq \mathbb (1 - \delta) \mathbb F_{\mu}(\delta b^{-1}). \label{I1-est}
\end{align}
The last two steps in the above sequence use the following estimate for $\xi \in \mathfrak D_1$: 
\begin{align} 
\bigl| 1 - \widehat{\sigma}(t\xi)\bigr| = \Bigl| \int \bigl(1 - e^{-itx \cdot \xi} \bigr) d\sigma(x) \Bigr| &\leq \int_{\mathbb S^{d-1}} t |x||\xi| d\sigma(x)  \nonumber \\ 
&\leq \delta \int_{\mathbb S^{d-1}} |x| d\sigma(x)  \leq \delta \quad \text{ if }  t|\xi| \leq \delta,  \label{sigma-hat-bound}
\end{align}  
and the inclusion $[0, \delta/b] \subseteq [0, \delta/t]$. On the other hand, the uniform bound $\bigl| \widehat{\sigma}(t \xi )\bigr| \leq 1$, combined with the spectral gap assumption \eqref{spectral-gap}  implies 
\begin{align} \bigl| \mathfrak I_2(t) \bigr| &= \int_{\mathfrak D_2} \bigl| \widehat{\mu}(\xi) \bigr|^2 \bigl| \widehat{\sigma}(t \xi) \bigr| \, d \xi \nonumber \\ 
&\leq \int_{\mathfrak D_2} \bigl| \widehat{\mu}(\xi) \bigr|^2  \, d\xi  \nonumber \\ 
&\leq \int_{\delta/b \leq |\xi| \leq a^{-N}} \bigl| \widehat{\mu}(\xi) \bigr|^2  \, d\xi 
\leq \mathbb F_{\mu} \bigl(a^{-N} \bigr) - \mathbb F_{\mu}\bigl(\delta b^{-1}\bigr) \leq a. \label{I2-est}
\end{align}
We have used the containment $[\delta/t, t^{-N}] \subseteq [\delta/b, a^{-N}]$ in the third step of the display above. Finally we consider 
\begin{align*}
\mathfrak I_3(t) &= \int_{\mathfrak D_3} \bigl| \widehat{\mu}(\xi) \bigr|^2 \widehat{\sigma}(t \xi) \, d \xi.  
\end{align*}
We estimate it using the assumption $N\alpha > (d-1)/2$ and the finite energy condition \eqref{finite-energy}: 
\begin{align}
 \mathfrak I_3(t) &\leq \int_{\mathfrak D_3} \bigl| \widehat{\mu}(\xi) \bigr|^2 (t |\xi|)^{-\frac{d-1}{2}} \, d \xi \nonumber \\ 
&\leq t^{-\frac{d-1}{2}} \int_{\mathfrak D_3} \bigl| \widehat{\mu}(\xi) \bigr|^2 |\xi|^{-\frac{d-1}{2} + \alpha} |\xi|^{-\alpha} \, d \xi
\nonumber \\ 
&\leq t^{N\alpha - \frac{d-1}{2}} \mathbb I_{\mu}\bigl(\frac{d+1}{2} + \alpha \bigr) \leq C_0 b^{N\alpha - \frac{d-1}{2}}. \label{I3-est} 
\end{align}
Combining \eqref{I1-est}, \eqref{I2-est} and \eqref{I3-est}, we obtain the following lower bound on $\Lambda_{\mu}(t)$: 
\begin{align}
\Lambda_{\mu}(t) &\geq (1 - \delta) \mathbb F_{\mu}(\delta b^{-1}) - a -  C_0 b^{N\alpha - \frac{d-1}{2}} \nonumber \\ 
&\geq 4c_d \Bigl[(1 - \delta)  - \frac{1}{2^{2}} - \frac{1}{2^2} \Bigr] \geq c_d, \label{lower-bound-Lambda}
\end{align}  
provided \eqref{abdN} holds.  Here $c_d$ is the constant given in \eqref{def-cd}, and we have  used the estimate
\begin{align} \bigl|1 - \widehat{\mu}(\xi) \bigr| &= \Bigl| \int \bigl( 1 - e^{-ix \cdot \xi} \bigr) d\mu(x) \Bigr| \nonumber \\ 
&\leq \Bigl| \int_{[0,1]^d} |x| |\xi| d \mu(x)  \bigr| \nonumber \\ 
&\leq \sqrt{d}|\xi| \leq \frac{1}{2} \text{ for } |\xi| \leq  \frac{1}{2\sqrt{d}} \leq \delta b^{-1}, \text{ so that }  \nonumber  \\
 \mathbb F_{\mu}(\delta b^{-1}) &\geq \mathbb F_{\mu}\Bigl(\frac{1}{2 \sqrt{d}} \Bigr) \geq \Bigl(\frac{1}{2} \Bigr)^{2} \Bigl(\frac{\omega_d}{2^d d^{\frac{d}{2}}}\Bigr) = 4c_d.  \label{F-mu-lower-bound} 
 \end{align}
 The lower bound in \eqref{lower-bound-Lambda} is the conclusion claimed in Proposition \ref{mainprop-1}.  
\qed
\subsection{Proof of Proposition \ref{mainprop-1'}}
The proof follows the broad strokes of Proposition \ref{mainprop-1}. Choosing $\delta \in (2b, \frac14)$, $t \in [a, b]$ and setting $\mathtt e = t/\delta$, we write 
\begin{align} 
\Lambda_{\mu, \nu}(t) &= \int \widehat{\mu}(\xi) \overline{\widehat{\nu}(\xi)} \widehat{\sigma}(t \xi) \, d\xi = \mathfrak J_1(t) + \mathfrak J_1'(t) \text{ where } \\ 
\mathfrak J_1(t) :&= \int \widehat{\mu}(\xi) \overline{\widehat{\nu}(\xi)} \widehat{\sigma}(t \xi) \bigl[\widehat{\psi}(\mathtt e \xi)\bigr]^2 \, d\xi \nonumber \\ &= \int \widehat{\mu}_{\mathtt e}(\xi) \overline{\widehat{\nu}_{\mathtt e}(\xi)} \widehat{\sigma}(t \xi) \, d\xi  = \Lambda_{\mu_{\mathtt e}, \nu_{\mathtt e}}(t) \text{ and } \nonumber \\ 
 \mathfrak J_1'(t) &:= \int \widehat{\mu}(\xi) \overline{\widehat{\nu}(\xi)} \widehat{\sigma}(t \xi) \Bigl(1 - \bigl[\widehat{\psi}(\mathtt e \xi)\bigr]^2 \Bigr) \, d\xi. \nonumber 
\end{align}  
The assumption \eqref{third-condition-mu-nu} provides the lower bound $|\mathfrak J_1(t)| = \Lambda_{\mu_{\mathtt e}, \nu_{\mathtt e}} (t) \geq 4c_d$. In the remainder of the proof, we will estimate $\mathfrak J_1'(t)$ from above to show 
\begin{equation}  
\bigl| \mathfrak J_1'(t) \bigr| \leq 2 c_d  
\end{equation}   
establishing it as an error term that leads to \eqref{mainprop-1'-conclusion}. The support property \eqref{psi-conditions}  of $\widehat{\psi}$ implies that the domain of integration of $\mathfrak J_1(t)$ is contained in $\{\xi : |\xi| > 1/\mathtt e \} = \mathbb R^d \setminus B(0; \delta/t)$. This permits the decomposition
\begin{align*}
|\mathfrak J_1'(t)| &\leq \mathfrak J_2(t) + \mathfrak J_3(t), \text{ where } \\
\mathfrak J_2(t) := \int_{\mathfrak D_2} \bigl|\widehat{\mu}(\xi) \overline{\widehat{\nu}(\xi)} &\widehat{\sigma}(t \xi) \bigr| \, d\xi, \quad
\mathfrak J_3(t) := \int_{\mathfrak D_3} \bigl|\widehat{\mu}(\xi) \overline{\widehat{\nu}(\xi)} \widehat{\sigma}(t \xi) \bigr| \, d\xi.  
\end{align*} 
Here the domains $\mathfrak D_2$ and $\mathfrak D_3$ are same as in the proof of Proposition \ref{mainprop-1}. By H$\ddot{\text{o}}$lder's inequality, 
\begin{align*} 
\mathfrak J_2(t) &\leq  \bigl[\mathbb F_{\mu}(a^{-N}) - \mathbb F_{\mu}(\delta b^{-1}) \bigr]^{\frac{1}{2}}  \bigl[\mathbb F_{\nu}(a^{-N}) - \mathbb F_{\nu}(\delta b^{-1}) \bigr]^{\frac{1}{2}} \leq a \leq c_d, \\
\mathfrak J_3(t) &\leq \Bigl[ t^{N \alpha - \frac{d-1}{2}} \mathbb I_{\mu} \Bigl(\frac{d+1}{2} + \alpha \Bigr) \Bigr]^{\frac{1}{2}} \Bigl[ t^{N \alpha - \frac{d-1}{2}} \mathbb I_{\nu} \Bigl(\frac{d+1}{2} + \alpha \Bigr) \Bigr]^{\frac{1}{2}} \leq C_0 b^{N \alpha - \frac{d-1}{2}} \leq c_d. 
\end{align*} 
The final estimates for $\mathfrak I_2(t)$ and $\mathfrak I_3(t)$ in the display above follow from \eqref{I2-est} and \eqref{I3-est} respectively in Proposition \ref{mainprop-1}. This completes the proof of Proposition \ref{mainprop-1'}.  
\qed

\section{Proofs of Proposition \ref{mainprop-2} and \ref{mainprop-2'}} \label{mainprop-2-section}
\subsection{Choice of the constants $C_0, \varrho_0, a, b, \varepsilon$} \label{parameter-selection-section} 
Let the parameters $d \geq 2$, $\kappa \in (0,1)$, $\alpha$, $N$ be given, as allowed by the assumptions of the two propositions. 
\begin{enumerate}[1.] 
\item We first fix an auxiliary smooth non-negative function 
\begin{equation} \label{conditions-phi} 
\varphi: \mathbb R^d \rightarrow [0, 2] \quad \text{ such that } \varphi \in C_c^{\infty}([0,1]^d), \quad \int_{[0,1]^d} \varphi(x) \, dx = 1. 
\end{equation} 
The absolute constant $C_0$, identified in Section \ref{finite-energy-verification}, depends only on $d$ and $||\varphi||_{\infty}$. 
\vskip0.1in
\item Given $N \geq 1$, we choose $c_0 > 0$ small enough to satisfy 
\begin{equation} \label{what-is-c0}
c_0 \bigl[N(d+1) + 1 \bigr] < \frac{1}{d}.
\end{equation}  
\item For $\varphi$ as above, there exists a constant $R_0 \gg 1$ depending only on $\varphi$ and $\kappa$ such that 
\begin{equation} \label{tail-phi} 
\int_{|\xi| > R^{\kappa/2}} \bigl| \widehat{\varphi}(\xi) \bigr|\, d\xi \leq R^{-2} \quad \text{ for all } R \geq R_0. 
\end{equation} 
Let us define $\varrho_0 = 2^{-dT_0-3}$, where $T_0$ is a large constant to be specified by the criteria \eqref{rho0} and \eqref{T-large} below. The constant $T_0$ depends on $R_0$ (and hence $\kappa$) and $c_0$, so that 
\begin{equation} \label{rho0}
\varrho_0^{c_0} \leq R_0^{-1}, \quad \varrho_0^{c_0 \kappa} \leq c_d \quad \text{ and } \quad C_0 \varrho_0^{c_0 \kappa (N \alpha - \frac{d-1}{2})} \leq c_d, 
\end{equation}  
where $c_d$ is the dimensional constant provided by Proposition \ref{mainprop-1} appearing in \eqref{abdN}. 
\vskip0.1in
\item For every $\rho \in (0, \varrho_0]$, the constants $a, b, \delta$ claimed in Proposition \ref{mainprop-2} are chosen based on $\rho$ (and hence $T$) as follows: $a = \rho^{c_0}$, $\delta = a^{\kappa/2}$, $b = a^{\kappa}$. This choice is aligned with the claim \eqref{relation-abrho}.  In view of \eqref{rho0}, we observe that for all $\rho \in (0, \varrho_0]$,
\vskip0.1in
\begin{itemize} 
\item $a = \rho^{c_0}$ obeys \eqref{tail-phi} with $a = R^{-1}$; 
\vskip0.1in 
\item $b = a^{\kappa}$ obeys the conditions specified in \eqref{abdN}. 
\end{itemize}
\vskip0.1in   
\item  For every $\rho \in (0, \varrho_0]$, there is a largest integer $T \geq T_0$ such that  $\rho \leq 2^{-dT-3}$, i.e. 
\begin{equation} \label{what-is-T}
2^{-d(T+1)-3} < \rho \leq 2^{-dT-3}. 
\end{equation}  
We choose $\varrho_0$ small enough (and hence $T_0$ large enough) to ensure that $a = \rho^{c_0}$ obeys 
\begin{align} 
\vartheta_d 2^{-T+1} a^{-N(d+1)} + a^2 &\leq \vartheta_d 2^{2 + \frac{3}{d}} \rho^{\frac{1}{d} - c_0 N(d+1)} + \rho^{2c_0} \nonumber \\ &\leq \rho^{c_0} \leq a, \text{ for all } T \geq T_0, \text{ and hence all } \rho \leq \varrho_0.  \label{T-large} 
\end{align}  
Here $\vartheta_d$ is an absolute dimensional constant representing $\sqrt{d}$ times the surface area of the unit sphere in $\mathbb R^d$. The choice \eqref{what-is-c0} of $c_0$ guarantees the existence of $T_0$.  
\vskip0.1in
\item Finally, for every $\rho \in (0, \varrho_0]$ with $T$ as in \eqref{what-is-T}, let $\varepsilon = \varepsilon_{\rho}$ be a positive constant such that 
\begin{equation}  \label{what-is-epsilon}
1 - 2^{-dT-3} > 2^{(d-s)T} - 2^{-sT-1} \quad \text{ for all } s \in [d-\varepsilon_{\rho}, d]. 
\end{equation}   
One can check that $\varepsilon = \varepsilon_{\rho}$ can be chosen in accordance with \eqref{relation-abrho}, i.e. 
\[ \varepsilon_{\rho} = \frac{c_0 2^{-dT}}{T} \leq \frac{c_0 \rho}{\log(1/\rho)} < \frac{1}{8} \text{ for some small constant $c_0 > 0$ obeying \eqref{what-is-c0}.} \] 
\end{enumerate} 
\subsection{Construction of the measure $\mu$} \label{measure construction section}
\begin{lem} 
For $\rho \in (0, \varrho_0]$, let $\varepsilon = \varepsilon_{\rho}$ be as in \eqref{what-is-epsilon}. Then for $s > d - \varepsilon_{\rho}$, and any Borel set $E \subseteq [0,1]^d$,
\begin{equation} 
\mathcal H^s_{\infty}(E) \geq 1- \rho \quad \text{ implies }  \quad \mathcal H_{\infty}^s(E \cap Q) \geq \frac{1}{2} \ell(Q)^s \text{ for every } Q \in \mathscr{D}_T.   \label{HQ-large}
\end{equation}   
\noindent Here $T$ is as in \eqref{what-is-T}, and $\mathscr {D}_T$ denotes the dyadic descendants of $[0,1]^d$ of generation $T$. In other words, $\mathscr{D}_T$ consists of the dyadic sub-cubes of $[0,1]^d$ with side-length $2^{-T}$. 
\end{lem} 
\begin{proof} 
Let us set 
\[ \mathcal G := \Bigl\{Q \in \mathscr{D}_T : \mathcal H_{\infty}^s(E \cap Q) \geq \frac{1}{2} \ell(Q)^s \Bigr\}. \]
Aiming for a contradiction, let us assume that $\mathcal G$ is a strict subset of $\mathscr{D}_T$. The trivial bound $\mathcal H_{\infty}^s(E \cap Q) \leq \ell(Q)^s$ then leads to the following estimate:   
\begin{align*} 
1 - 2^{-dT-3} \leq 1 - \rho \leq \mathcal H^s_{\infty}(E) & \leq \sum_{Q \in \mathscr{D}_T}  \mathcal H^s_{\infty}(E \cap Q)
\\ 
&\leq \sum_{Q \in \mathcal G}  \ell(Q)^s + \Bigl( 1 - \frac{1}{2} \Bigr) \sum_{Q \in \mathscr{D}_T \setminus \mathcal G} \ell(Q)^s \\ 
&\leq \sum_{Q \in \mathscr{D}_T} \ell(Q)^s - \frac{1}{2} 2^{-sT} \\ 
&\leq 2^{(d-s)T} -  \frac{1}{2} 2^{-sT}.  
\end{align*} 
This contradicts the defining property \eqref{what-is-epsilon} of $\varepsilon$, completing the proof.  
\end{proof} 
\noindent In view of the conclusion in \eqref{HQ-large}, which ensures that each $\mathcal H^s_{\infty}(E \cap Q) > 0$, one can apply Frostman's lemma \cite[Theorem 8.8] {Mattila-Book1}  for every $Q \in \mathscr{D}_T$. This yields a positive finite Radon measure $\overline{\mu}_Q$ and an absolute constant $A_d > 0$ (depending only on $d$) such that 
 \begin{align} 
&\text{supp}(\overline{\mu}_Q) \subseteq E \cap Q, \quad ||\overline{\mu}_Q|| \geq \frac{1}{2} \ell(Q)^s \text{ and } \label{nu-supp-wt} \\ 
&\sup\Bigl\{\frac{\overline{\mu}_Q \bigl(B(x;r\bigr)}{r^s} : x \in \mathbb R^d, \, r > 0 \Bigr\} \leq A_d < \infty. \label{nu-ball-condition}  
\end{align}  
The measure prescribed in Proposition \ref{mainprop-2} will be of the form: 
\begin{equation} \label{def-mu-nu}
\mu := \sum_{Q \in \mathscr{D}_T} w(Q) \frac{\overline{\mu}_Q}{||\overline{\mu}_Q||}, 
\end{equation} 
where the ``weights'' $w(Q)$ are defined by 
\begin{equation} \label{weights-def} w(Q) := \int_{Q} \varphi(x) \, dx. \end{equation}
Here $\varphi$ is the auxiliary function chosen at the beginning of the proof, satisfying the properties \eqref{conditions-phi}. 
It follows from \eqref{nu-supp-wt} that $\mu$ is supported on $E$, and 
\[ ||\mu|| = \sum_{Q \in \mathscr{D}_T} w(Q) = \int_{[0,1]^d} \varphi(x) \, dx = 1; \] 
in other words, $\mu$ is a probability measure on $E$. 
\vskip0.1in
\noindent In subsections \ref{finite-energy-verification} and \eqref{spectral-gap-verification} below, we will verify that the measure $\mu$ defined in \eqref{def-mu-nu} obeys the finite energy condition \eqref{finite-energy} and the spectral gap condition \eqref{spectral-gap} respectively. This will complete the proof of Proposition \ref{mainprop-2}. We will also show, in subsection \ref{third-condition-verification}, that if $\mu$ and $\nu$ are two measures of the form \eqref{def-mu-nu} associated with two sets $E, F \subseteq [0,1]^d$ with $s$-dimensional Hausdorff content at least $1 - \rho$, then \eqref{third-condition-mu-nu} holds. This would conclude the proof of Proposition \ref{mainprop-2'}.    
\subsection{Verification of \eqref{finite-energy}} \label{finite-energy-verification} We do this in two steps. Lemma \ref{ball-condition-lemma} establishes a ball condition on $\mu$; Lemma \ref{ball-energy-lemma} then exploits a standard relation between such conditions and energy integrals to arrive at \eqref{finite-energy}.   
\begin{lem} \label{ball-condition-lemma} 
There exists a constant $A_d \geq 1$ depending only on $d$ such that the probability $\mu$ defined by \eqref{def-mu-nu} obeys the ball condition
\begin{equation} \label{ball-condition} 
\mu \bigl(B(x;r)\bigr) \leq A_d r^s \text{ for all } x \in \mathbb R^d, r > 0.  
\end{equation}  
\end{lem} 
\begin{proof} 
Let us observe that it suffices to establish \eqref{ball-condition} for dyadic cubes $\mathtt Q(x;r) = x + [0,r]^d$ of sidelength $r$ instead of balls $B(x;r)$. With this reduction, the proof is arranged in three cases, depending on the range of $r$. 
\vskip0.1in 
\noindent {\em{Case 1:}} First suppose $0 \leq r \leq 2^{-T}$. In this case, the ball $\mathtt Q(x;r)$ is contained in a single cube $Q \in \mathscr{D}_T$. Invoking \eqref{nu-supp-wt} and \eqref{nu-ball-condition} for $\overline{\mu}_Q$, we obtain  
\begin{align*} 
\mu(\mathtt Q(x;r)) &\leq \max_{Q \in \mathscr{D}_T} w(Q) \frac{\overline{\mu}_Q(\mathtt Q(x;r))}{||\overline{\mu}_Q||} \\ 
&\leq \max_{Q \in \mathscr{D}_T} \Bigl[ \int_{Q} \varphi(x) \, dx \Bigr]  2 \ell(Q)^{-s} A_d r^s \\ 
&\leq 2 A_d r^s \max_{Q \in \mathscr{D}_T} ||\varphi||_{\infty} \ell(Q)^{d-s} \\ 
&\leq 4 A_d r^{s}, 
\end{align*} 
where the penultimate step uses the fact $\ell(Q) \leq 1$, so that $(\ell(Q))^{d-s} \leq 1$. 
\vskip0.1in
\noindent {\em{Case 2:}} $r \geq 1$. Here we use the trivial bound on $\mu$:
\[ \mu(\mathtt Q(x;r)) \leq 1 \leq  r^s. \]
\vskip0.1in
\noindent {\em{Case 3:}} $2^{-T} \leq r \leq 1$. Here we are led to consider all $Q \in \mathscr{D}_T$ that are contained in $\mathtt Q(x;r)$.
\begin{align*}
\mu \bigl(\mathtt Q(x;r) \bigr) &= \sum_{\begin{subarray}{c}Q \in \mathscr{D}_T \\ Q \cap \mathtt Q(x; r) \ne \emptyset \end{subarray}} w(Q) \frac{\overline{\mu}_Q(\mathtt Q(x;r))}{||\overline{\mu}_Q||} \leq  \sum_{\begin{subarray}{c}Q \in \mathscr{D}_T \\ Q \cap \mathtt Q(x; r) \ne \emptyset \end{subarray}} w(Q) \\ 
&\leq ||\varphi||_{\infty} 2^{-dT} \# \Bigl\{Q \in \mathscr{D}_T: Q \cap \mathtt Q(x;r) \ne \emptyset \Bigr\} \\ 
&\leq  2^{-dT+1}  \Bigl(\frac{r}{2^{-T}} \Bigr)^d
\leq 2 {r^d} \leq 2r^s, 
\end{align*}  
where the last step uses the assumption $r \leq 1$. 
\end{proof} 
\begin{lem} \label{ball-energy-lemma}
Let $d \geq 2$. For every constant $A \geq 1$, there exists $C_0 = C_0(d, A) \geq 1$ with the following property. Let $\mu$ be any probability measure with supp$(\mu) \subseteq [0,1]^d$ such that  
\begin{equation} \label{ball-condition-2}
\sup \Bigl\{ \mu \bigl(B(x;r) \bigr) r^{-s} : x \in \mathbb R^d, \; r > 0 \Bigr\} = A < \infty \quad \text{ for some $s > d - \frac{1}{8}$.}  
\end{equation}    
Then 
\begin{equation} 
\mathbb I_{\mu}\Bigl(\frac{d+1}{2} + \alpha\Bigr) \leq C_0(d, A) \qquad \text{ for all $\alpha \in \bigl[0, \frac18 \bigr]$}.   
\end{equation} 
\end{lem} 
\begin{proof}
We start with the definition of the energy integral. Decomposing the domain of integration based on dyadic distances to the diagonal yields 
\begin{align*}
\mathbb I_{\mu}\Bigl(\frac{d+1}{2} + \alpha\Bigr)  &= \iint_{x, y \in [0,1]^d} \frac{d\mu(x) \, d\mu(y)}{|x-y|^{\frac{d+1}{2} + \alpha}} \\
&\leq \sum_{j=1}^{\infty} \iint_{2^j|x-y| \in [1, 2]} \frac{d \mu(x) d\mu(y)}{|x-y|^{\frac{d+1}{2} + \alpha}} \\ 
&\leq  \sum_{j=1}^{\infty} \int_{[0,1]^d} \Biggl[ \frac{\mu \bigl( B(x;2^{-j+1})\bigr)}{2^{-j \left(\frac{d+1}{2} + \alpha \right)}} \Biggr] d\mu(x) \\ 
&\leq A2^{s} \sum_{j=1}^{\infty} 2^{-j \left(s - \frac{d+1}{2} - \alpha \right)} \quad (\text{by \eqref{ball-condition-2}})\\ 
&\leq A 2^s \sum_{j=1}^{\infty} 2^{-j \left(\frac{d-1}{2} - \frac14\right)} =: C_0(d, A).  
\end{align*}
The penultimate step follows from the inequality $s - \frac{d+1}{2} - \alpha > d - \frac{1}{8}  - \frac{d+1}{2} - \alpha \geq \frac{d-1}{2} - \frac{1}{4}$,  which is bounded below by an absolute positive constant independent of $s$.  
\end{proof}
\subsection{Verification of \eqref{spectral-gap}} \label{spectral-gap-verification} 
 Let $\varphi$ be the smooth function chosen in Section \ref{parameter-selection-section}, obeying the conditions in \eqref{conditions-phi}. Our choice of $a$ given by the conditions \eqref{tail-phi} and \eqref{T-large} in that section ensure that 
 \begin{equation}  \int_{|\xi| > a^{-\kappa/2}} \bigl| \widehat{\varphi}(\xi) \bigr| \, d\xi \leq a^2 \quad \text{ and } \quad \vartheta_d 2^{-T+1} a^{-N(d+1)} + a^2 \leq a \label{phi-and-a} \end{equation} 
for all $a = \rho^{c_0}$, with $\rho \in (0, \varrho_0]$. Choosing $a > 0$ sufficiently small also ensures that \eqref{abdN} holds, with $\delta = a^{\kappa}$ and $b = a^{2\kappa}$. 

\begin{lem} \label{spectral-gap-lemma}
\begin{enumerate}[(a)]
\item Let $\mu$ be defined as in \eqref{def-mu-nu}. Then
\[ \bigl| \widehat{\mu}(\xi) - \widehat{\varphi}(\xi) \bigr| \leq \sqrt{d} |\xi|2^{-T+1} \quad \text{ for all } \xi \in \mathbb R^d.   \]  
\item The measure $\mu$ obeys the spectral gap condition \eqref{spectral-gap} for $a, b, \delta \in (0,1)$ with $a < b \leq \delta a^{\kappa/2}$. 
\end{enumerate} 
\end{lem} 
\begin{proof} 
It follows from \eqref{def-mu-nu} that 
\[ \mu(Q) = w(Q) = \int_Q \varphi(x) \, dx \quad \text{ for all } Q \in \mathscr{D}_T, \]
hence 
 \[ \int_{Q} e^{-i c_{Q}\cdot \xi} d\mu(x) = \int_{Q} e^{-i c_{Q} \cdot \xi} \varphi(x) \, dx, \] 
 where $Q = c_Q + [0, 2^{-T}]^d$.  This implies 
 \begin{align*}
 \bigl| \widehat{\mu}(\xi) - \widehat{\varphi}(\xi) \bigr| &= \Bigl| \int_{[0,1]^d} e^{-ix\dot \xi} \bigl[ d\mu(x) - \varphi(x) \, dx \bigr] \Bigr| \\  
 &= \Bigl| \sum_{Q \in \mathscr{D}_T} \int_{Q} \bigl( e^{-ix\cdot \xi} - e^{-i c_{Q} \cdot \xi} \bigr) \times \bigl[ d\mu(x) - \varphi(x) \, dx \bigr] \Bigr| \\ 
 &\leq \sum_{Q \in \mathscr{D}_T} |x - c_{Q}||\xi| \times  ||\mu - \varphi||(Q) \\
 &\leq \sqrt{d} 2^{-T} |\xi| \sum_{Q \in \mathscr{D}_T} \bigl[ \mu(Q) + \varphi(Q) \bigr] \\ 
 & \leq \sqrt{d} 2^{-T+1} |\xi|, 
 \end{align*} 
 completing the proof of part (a). 
 \vskip0.1in
 \noindent For part (b), we apply the pointwise estimate in part (a) to deduce that 
 \begin{align*}
 \mathbb F_{\mu}(a^{-N}) - \mathbb F_{\mu}(\delta b^{-1}) &= \int_{\delta/b \leq |\xi| \leq a^{-N}} \bigl| \widehat{\mu}(\xi) \bigr|^2 \, d\xi \\ 
 &\leq \int_{\delta/b \leq |\xi| \leq a^{-N}} \bigl| \widehat{\mu}(\xi) \bigr| \, d\xi \\ 
 &\leq \int_{\delta/b \leq |\xi| \leq a^{-N}} \bigl| \widehat{\mu}(\xi) - \widehat{\varphi}(\xi) \bigr| \, d\xi + \int_{|\xi| \geq \delta/b}  \bigl| \widehat{\varphi}(\xi) \bigr| \, d\xi \\ 
 &\leq  \sqrt{d} 2^{-T+1} \int_{|\xi| \leq a^{-N}} |\xi| \, d\xi +  \int_{|\xi| \geq a^{-\kappa/2}}  \bigl| \widehat{\varphi}(\xi) \bigr| \, d\xi \\ 
 &\leq \vartheta_d 2^{-T+1} a^{-N(d+1)} + a^2 \leq a, 
 \end{align*}  
 where the last two steps follow from \eqref{phi-and-a}. This completes the verification of the spectral gap condition \eqref{spectral-gap} 
\end{proof} 

\subsection{Verification of \eqref{third-condition-mu-nu}} \label{third-condition-verification} 
\begin{lem} \label{third-condition-verification-lemma}
Suppose that $E, F$ are two sets with the property $\mathcal H^s_{\infty}(E), \mathcal H^s_{\infty}(F) > 1 - \rho$. Consider the two measures  
\begin{equation} \mu := \sum_{Q \in \mathscr{D}_T} w(Q) \frac{\overline{\mu}_Q}{||\overline{\mu}_Q||}, \qquad \nu := \sum_{Q \in \mathscr{D}_T} w(Q) \frac{\overline{\nu}_Q}{||\overline{\nu}_Q||} \label{munu-sum} \end{equation} 
supported on $E$ and $F$ respectively, constructed in \eqref{def-mu-nu}. Then \eqref{third-condition-mu-nu} holds for $\mu, \nu$ as above, with 
$\mathtt e = t/\delta$, $t \in [a, b]$ and $\psi$ obeying \eqref{psi-conditions}.
\end{lem} 
\begin{proof}
Let us recall from \eqref{def-Lambda} the definition of $\Lambda_{\mu, \nu}$ and the representation \eqref{mass-Lambda-mu-nu} of $||\Lambda_{\mu, \nu}(t)||$. In this proof we will denote $\Lambda_{\mu, \nu}$ by $\Lambda[\mu, \nu](t)$ to emphasize the role of the measures $\mu, \nu$. Let us set $\mathtt m_{Q} := {\overline{\mu}_Q}/{||\overline{\mu}_Q||}$, $\mathtt n_{Q} := {\overline{\nu}_Q}/{||\overline{\nu}_Q||}$ in \eqref{munu-sum}. Since $\Lambda[\cdot, \cdot]$ is linear in each argument, the expression \eqref{munu-sum} yields 
\begin{equation} \Lambda[\mu_{\mathtt e}, \nu_{\mathtt e}](t) =  \sum_{Q, Q' \in \mathscr{D}_T} w(Q) w(Q') \Lambda[\mathtt m_{Q, \mathtt e}, \mathtt n_{Q', \mathtt e}](t) \label{QQ'sum} \end{equation} 
where $\mathtt m_{Q, \mathtt e} := \mathtt m_{Q} \ast \psi_{\mathtt e}$, $\mathtt n_{Q', \mathtt e}:= \mathtt n_{Q'} \ast \psi_{\mathtt e}$. 
The function $\psi$ is non-negative by assumption \eqref{psi-conditions}, therefore the same is true for each summand in the above sum. We now fix an integer $M$ such that
\begin{equation} 
2^M \in \Bigl[\frac{8 \sqrt{d}}{\mathtt e}, \frac{16 \sqrt{d}}{\mathtt e} \Bigr]. \label{choice-of-M}
\end{equation} 
Let us note that $M \ll T$; indeed for $t \in [a, b]$, and $a = \rho^{c_0}$ selected as in Section \ref{parameter-selection-section}, 
\[ 2^M \leq \frac{16 \sqrt{d}}{\mathtt e}\leq \frac{C_d \delta}{a} \leq \frac{C_d}{a} \leq C_d \rho^{-c_0} \leq C_d 2^{c_0dT} \ll 2^{T},  \] 
where the last two inequalities in the chain above follow from \eqref{what-is-T} and \eqref{what-is-c0} respectively. One can therefore define  
\begin{equation} 
\mathscr{E}_M := \Bigl\{(Q, Q') \in \mathscr{D}_T \times \mathscr{D}_T : \exists \mathtt Q \in \mathscr{D}_M \text{ such that } Q, Q' \subseteq \mathtt Q \Bigr\}.
\end{equation}
Recall that $\mathscr{D}_M$ denotes all the dyadic descendants $\mathtt Q$ of $[0,1]^d$ of generation $M$; $\ell(\mathtt Q) = 2^{-M}$. Thus $\mathscr{E}_M$ consists of all pairs $(Q, Q') \in \mathscr{D}_T \times \mathscr{D}_T$ with the property that $Q,Q'$ have a common ancestor at generation $M$. 
\vskip0.1in
\noindent The relevance of $\mathscr{E}_M$ will be clarified momentarily in Lemma \ref{QQ'e-lemma} below. Assuming the conclusions of this lemma for now, and replacing the sum in \eqref{QQ'sum} over $\mathscr{D}_T \times \mathscr{D}_T$ by the smaller sub-sum over $\mathscr{E}_M$, we arrive at the following lower bound:
\begin{align*}
\Lambda[\mu_{\mathtt e}, \nu_{\mathtt e}](t) &=  \sum_{(Q, Q') \in \mathscr{D}_T \times \mathscr{D}_T} w(Q) w(Q') \Lambda[\mathtt m_{Q, \mathtt e}, \mathtt n_{Q', \mathtt e}](t) \\ 
&\geq \sum_{(Q, Q') \in \mathscr{E}_M} w(Q) w(Q') \Lambda[\mathtt m_{Q, \mathtt e}, \mathtt n_{Q', \mathtt e}](t) \\ 
&\geq \tilde{\gamma}_d \mathtt e^{-d} \sum_{\mathtt Q \in \mathscr{D}_M} \sum_{Q, Q' \subseteq \mathtt Q} w(Q) w(Q') \quad \text{(from \eqref{QQ'-partc-eqn} in Lemma \ref{QQ'e-lemma})}\\
&  \geq \tilde{\gamma}_d \mathtt e^{-d} \sum_{\mathtt Q \in \mathscr{D}_M} \bigl[ w(\mathtt Q) \bigr]^2 
\geq \tilde{\gamma}_d \mathtt e^{-d} 2^{-Md} \geq \bigl(\frac{\gamma_d}{16 \sqrt{d}} \bigr)^{-d} \tilde{\gamma}_d =: 4c_d. 
\end{align*} 
Let us justify the steps above. The constant $\tilde{\gamma}_d$ appearing in the third step above is a positive dimensional constant, obtained from Lemma 
\ref{QQ'e-lemma}: $\tilde{\gamma}_d = \gamma_d^{3} 2^{-3(d+1)}$.  The fourth step is a consequence of the identity $\sum_{Q \subseteq \mathtt Q} w(Q) = w(\mathtt Q)$, which follows from the definition \eqref{weights-def} of $w(\cdot)$. The next step uses the Cauchy-Schwarz inequality: 
\[ \sum_{\mathtt Q \in \mathscr{D}_M}w(\mathtt Q) \leq \Bigl( \sum_{\mathtt Q \in \mathscr{D}_M} \bigl[ w(\mathtt Q) \bigr]^2 \Bigr)^{\frac{1}{2}} \#(\mathscr{D}_M)^{\frac12} \leq 2^{\frac{Md}{2}}  \Bigl( \sum_{\mathtt Q \in \mathscr{D}_M} \bigl[ w(\mathtt Q) \bigr]^2 \Bigr)^{\frac{1}{2}}, \] combined with the identity $\sum_{\mathtt Q \in \mathscr{D}_M} w(\mathtt Q) = \int \varphi = 1$, which in turn is a result of  the normalization \eqref{conditions-phi} of $\varphi$. The final step in the string of inequalities above uses \eqref{choice-of-M}, which gives $2^{Md} \sim \mathtt e^{-d}$. 
\end{proof} 

\begin{lem} \label{QQ'e-lemma}
Let $\mathtt e = t/\delta$, $\mathtt m_{Q, \mathtt e}$, $\mathtt n_{Q', \mathtt e}$ be as in the proof of Lemma \ref{third-condition-verification-lemma}.  
\begin{enumerate}[(a)] 
\item For any $Q, Q' \in \mathscr{D}_T$, 
\begin{equation}
\Bigl| \int \widehat{\mathtt m}_{Q, \mathtt e}(\xi) \overline{\widehat{\mathtt n}_{Q', \mathtt e}(\xi)} \bigl[ 1 - \widehat{\sigma}(t \xi) \bigr] \, d\xi \Bigr| \leq \gamma_d 2^{d+1}\delta \mathtt e^{-d}.  \label{QQ'-parta-eqn} \end{equation}  \label{QQ'-parta}  
\vskip0.1in 
\item For each $(Q, Q') \in \mathscr{E}_M$, 
\begin{equation} \label{QQ'e}
\langle \mathtt m_{Q, \mathtt e}, \mathtt n_{Q', \mathtt e}\rangle = \int \bigl[ \mathtt m_{Q} \ast \psi_{\mathtt e}(x) \bigr] \times \bigl[ \mathtt n_{Q'} \ast \psi_{\mathtt e}(x) \bigr] \, dx \geq 2^{-3d-2} \gamma_d^{3} \mathtt e^{-d}. 
\end{equation} \label{QQ'-partb}
\item There exists a dimensional constant $\delta_d > 0$ such that for $\delta \in (0, \delta_d)$ and $(Q, Q') \in \mathscr{E}_M$, 
\begin{equation} \Lambda[\mathtt m_{Q, \mathtt e}, \mathtt n_{Q', \mathtt e}](t) \geq \gamma_d^{3} 2^{-3d-3} \mathtt e^{-d}, \text{ where } \mathtt e = t/\delta. \label{QQ'-partc-eqn} \end{equation} \label{QQ'-partc}
Here $\gamma_d$ denotes the volume of the unit ball in $\mathbb R^d$. 
\end{enumerate}
\end{lem}
\begin{proof} 
Let us start with part (\ref{QQ'-parta}). Since $\mathtt m_Q$ and $\mathtt n_{Q'}$ are both probability measures, their Fourier transforms are uniformly bounded above by 1. Also, the size and support condition \eqref{psi-conditions} on $\widehat{\psi}$ implies that $0 \leq \widehat{\psi}_{\mathtt e} \leq 1$ on supp$(\widehat{\psi}_{\mathtt e}) \subseteq \{\xi: |\xi| \leq 2\delta/t \} = B(0; 2\mathtt e^{-1})$. Combining these, we can estimate the integral in \eqref{QQ'-parta-eqn} as follows, 
\begin{align*}
\Bigl| \int \widehat{\mathtt m}_{Q, \mathtt e}(\xi) \overline{\widehat{\mathtt n}_{Q', \mathtt e}(\xi)} \bigl[ 1 - \widehat{\sigma}(t \xi) \bigr] \, d\xi \Bigr| &= \Bigl| \int \widehat{\mathtt m}_{Q}(\xi) \overline{\widehat{\mathtt n}_{Q'}}(\xi) \bigl[ \widehat{\psi}_{\mathtt e}(\xi) \bigr]^2 \bigl[ 1 - \widehat{\sigma}(t \xi) \bigr]  \, d\xi \Bigr| \\  &\leq \int \bigl[ \widehat{\psi}(\mathtt e \xi) \bigr]^2 \bigl|1 - \widehat{\sigma}(t \xi) \bigr| \, d\xi \\ &\leq 2\delta \int_{B(0; 2\mathtt e^{-1})} d\xi \leq2^{d+1} \gamma_d \delta \mathtt e^{-d}.  
\end{align*} 
Recall $\gamma_d$ denotes the volume of the unit ball in $\mathbb R^d$. At the penultimate step of the display above, we have used the estimate $|1 - \widehat{\sigma}(t \xi)| \leq 2 \delta$ for $t|\xi| \leq 2 \delta$, which has been proved in \eqref{sigma-hat-bound}.  This is the estimate claimed in \eqref{QQ'-parta-eqn}.
\vskip0.1in  
\noindent We continue to part (\ref{QQ'-partb}). It follows from the assumptions \eqref{psi-conditions} on $\psi$ that 
\begin{align} &\bigl| \psi(x) - \psi(0) \bigr| = \Bigl| \int_{B(0;2)} \widehat{\psi}(\xi) (e^{i x \cdot \xi} - 1) \, d\xi \Bigr| \nonumber \\
& \hskip0.9in \leq |x| \int_{B(0;2)} \widehat{\psi}(\xi) |\xi| d\xi  \leq 2 |x| \psi(0) \leq \frac{\psi(0)}{2} \text{ if } |x| \leq \frac{1}{4}, \nonumber \\
& \text{ hence } \psi(x) \geq \frac{\psi(0)}{2} \geq \frac{\gamma_d}{2}\text{ for } |x| \leq \frac{1}{4}. \label{lower-bound-psi} 
\end{align}
If $(Q, Q') \in \mathscr{E}_M$, there exists $\mathtt Q \in \mathscr{D}_M$ such that $Q, Q' \subseteq \mathtt Q$. Therefore,  for every  $y \in Q$ and  $z \in Q'$, the relation $|x - y| \leq \mathtt e/8$ implies  
\begin{align*}  
|x-z| &\leq  |x-y| + |y-z| \leq \frac{\gamma_d \mathtt e}{8} + \sqrt{d} 2^{-M} \; (\text{since } \text{diam}(\mathtt Q) = \sqrt{d} 2^{-M}) \\
&\leq \frac{\mathtt e}{8} + \frac{\mathtt e}{8} =  \frac{\mathtt e}{4} \;  \text{ from \eqref{choice-of-M}; as a result,} \\
\psi \Bigl(&\frac{x-y}{\mathtt e} \Bigr) \psi \Bigl(\frac{x-z}{\mathtt e} \Bigr) \geq \frac{\gamma_d^2}{4}\text{ for } |x - y| \leq \frac{\mathtt e}{8} \; \text{ in view of \eqref{lower-bound-psi}}. 
\end{align*}
Incorporating this pointwise estimate into the integrand,  we arrive at the following lower bound on the integral:
\begin{align*} 
\int \bigl[\mathtt m_{Q} \ast \psi_{\mathtt e}(x) \bigr] \times \bigl[ \mathtt n_{Q'} \ast \psi_{\mathtt e}(x) \bigr] \, dx &= \mathtt e^{-2d} \iint \Bigl[ \int \psi \Bigl(\frac{x-y}{\mathtt e} \Bigr) \psi \Bigl(\frac{x-z}{\kappa} \Bigr) dx \Bigr] d\mathtt m_Q(y) d\mathtt n_{Q'}(z)  \\   
&\geq \mathtt e^{-2d} \iint  \Bigl[ \int_{B(y; \frac{\mathtt e}{8})} \frac{\gamma_d^2}{4} dx \Bigr] d\mathtt m_Q(y) d\mathtt n_{Q'}(z) \\
&\geq \mathtt e^{-2d} \frac{\gamma_d^2}{4}  \text{ vol}\bigl(B(y; \frac{\mathtt e}{8})\bigr)\geq \gamma_d^3 2^{-3d-2}\mathtt e^{-d}, 
\end{align*}  
which is the estimate claimed in \eqref{QQ'e}. 
\vskip0.1in
\noindent Part (\ref{QQ'-partc}) follows directly from parts (\ref{QQ'-parta}) and (\ref{QQ'-partb}). Using Plancherel's theorem, we estimate
\begin{align*}
\Lambda[\mathtt m_{Q, \mathtt e}, \mathtt n_{Q', \mathtt e}](t) &=  \int \widehat{\mathtt m}_{Q, \mathtt e}(\xi) \overline{\widehat{\mathtt n}_{Q', \mathtt e}(\xi)} \widehat{\sigma}(t \xi) \, d\xi \\
&=  \int \widehat{\mathtt m}_{Q, \mathtt e}(\xi) \overline{\widehat{\mathtt n}_{Q', \mathtt e}(\xi)} d\xi + \int \widehat{\mathtt m}_{Q, \mathtt e}(\xi) \overline{\widehat{\mathtt n}_{Q', \mathtt e}(\xi)} \bigl[1 - \widehat{\sigma}(t\xi) \bigr] d\xi \\ 
&= \langle \mathtt m_{Q,\mathtt e}, \mathtt n_{Q', \mathtt e} \rangle  + \int \widehat{\mathtt m}_{Q, \mathtt e}(\xi) \overline{\widehat{\mathtt n}_{Q', \mathtt e}(\xi)} \bigl[1 - \widehat{\sigma}(t\xi) \bigr] d\xi \\ 
&\geq\bigl(  \gamma_d^{3}  2^{-3d-2} - 2^{d+1} \gamma_d \delta \bigr) \mathtt e^{-d} \geq \gamma_d^{3} 2^{-3d-3} \mathtt e^{-d},
\end{align*}
provided $\delta 2^{d+1} < \gamma_d^{2} 2^{-3d-3}$. This is the bound claimed in \eqref{QQ'-partc-eqn}. 
\end{proof}
 
\section{Fourier asymptotics and distance sets} \label{F-asymptotics-section}
 In this section, we prove Theorems \ref{F-asymptotics-thm} and \ref{quasiregular-distance-thm}. Both follow from the key proposition below. We will first give conditional proofs of the two theorems, then prove Proposition \ref{F-asymptotics-prop} at the end of this section. 
\begin{prop} \label{F-asymptotics-prop} 
 For $d \geq 2$, fix constants $c \in (0, \frac{d-1}{4})$ and $T_0 \geq 1$. Suppose that $\mu$ is a probability measure supported on $E \subseteq [0,1]^d$ for which $\mathbb F_{\mu}$ obeys \eqref{controlled-growth} with these values of $c$ and $T_0$. Given any $M \geq 4$, let $\delta > 0$ be a constant such that 
\begin{equation} \label{how-small-is-delta} 
\delta^{-2} > T_0, \qquad \delta + 4^d \delta^{\frac{d-1}{2} - 2c} < \frac{1}{M}. 
\end{equation}  
If $\{R_j : 1 \leq j \leq J \} \subseteq (T_0, \infty)$ is a finite sequence that obeys \eqref{Rj-condition} for this value of $\delta$, then there exists an index $j \in \{1, 2, \ldots, J\}$  for which 
\begin{equation}
\Lambda_{\mu}(t_j) \geq \frac{1}{M} \mathbb F_{\mu}(\delta t_j^{-1}) > 0, \; \text{ where } \; t_j := \delta R_{j}^{-1}.  
\label{F-asymptotics-prop-conclusion} 
\end{equation}  
\end{prop} 
\subsection{Proof of Theorem \ref{F-asymptotics-thm}, assuming Proposition \ref{F-asymptotics-prop}}   
The conclusion of Theorem \ref{F-asymptotics-thm} is a straightforward consequence of Proposition \ref{F-asymptotics-prop}. Given $d, c, T_0$ and $M$, the quantity $\delta_M$ is defined as the largest constant that obeys both the relations in \eqref{how-small-is-delta}. For $\delta < \delta_M$, the positivity of $\Lambda_{\mu}(t_j)$ for some $j \in \{1, \ldots, J\}$ is ensured by Proposition \ref{F-asymptotics-prop}. This in turn implies that $t_j \in \Delta(E)$, from Proposition \ref{Lambda-prop}(b).    
\qed
\subsection{Proof of Theorem \ref{quasiregular-distance-thm}, assuming Proposition \ref{F-asymptotics-prop}} 
Let $E$ be a quasi-regular Borel set, equipped with a locally uniformly $s$-dimensional probability measure $\mu$ obeying \eqref{Strichartz-upper-lower}. We will verify the hypotheses of Proposition \ref{F-asymptotics-prop} with this choice of $\mu$, setting $c = \frac{d-1}{8}$ and $M = 4$.  The condition \eqref{controlled-growth} follows from \eqref{Strichartz-upper-lower}: 
\[ \mathbb F_{\mu}(T_1 T_2) \leq C_0 \bigl(T_1T_2 \bigr)^{d-s} \leq C_0^{2} T_1^{d-s} \mathbb F_{\mu}(T_2) \leq T_1^{c} \mathbb F_{\mu}(T_2)   \]
provided $c = \frac{d-1}{8}> d-s$ and $T_1, T_2 > T_0 := C_0^{2/(c-d+s)}$. With this choice of $T_0$ and $c$, let us fix a positive constant $\delta$ depending only on $d$ and $C_0$ that obeys \eqref{how-small-is-delta}. Next, given parameters $0 < \tau_1 \leq \tau_2 < 1$, the integer $J_0$ is chosen based only on $d, C_0$ and $\tau_2$, as follows:
\begin{equation} \label{choice of J0}
J_0 > 2m C_0^2 \quad \text{ where $m$ is the smallest integer with } \tau_2^{-m} > \delta^{-2}. 
\end{equation} 
\vskip0.1in 
\noindent Let us fix a sequence $\{t_j : 1 \leq j \leq J_0\}$ that is $(\tau_1, \tau_2)$-lacunary. The condition \eqref{lacunarity-defn} implies
\begin{equation} \label{lacunary-upper-lower}
\tau_1^{j-1} \leq \frac{t_j}{t_1} \leq \tau_2^{j-1} \quad \text{ for all } j \geq 1. 
\end{equation} 
We will verify \eqref{Rj-condition} in a moment, with 
\begin{equation} J = \left\lfloor \frac{J_0}{m} \right\rfloor + 1, \quad R_ j = \delta t_{jm}^{-1} \text{ for } 1 \leq j \leq J, \quad d-s < \varepsilon_0,  \label{J} \end{equation}  for a sufficiently small constant $\varepsilon_0 > 0$. This would ensure that the hypotheses and therefore the conclusion of Proposition \ref{F-asymptotics-prop} hold; namely, at least one of the elements $\delta R_{j}^{-1} = t_{jm}$ lies in $\Delta(E)$. This is the conclusion of Theorem \ref{quasiregular-distance-thm}. 
\vskip0.1in
\noindent It remains to verify \eqref{Rj-condition}. The first condition in \eqref{Rj-condition} follows from \eqref{choice of J0}:
\[ \frac{R_{j+1}}{R_j} = \frac{t_{jm}}{t_{jm+m}} = \prod_{k=0}^{m-1} \frac{t_{jm+k}}{t_{jm+k+1}} \geq \tau_2^{-m} > \delta^{-2}. \]  
To verify the second condition in \eqref{Rj-condition}, let us estimate its left hand side using \eqref{Strichartz-upper-lower}:  
\begin{align} 
\mathfrak L := \sum_{j=1}^{J} \mathbb F_{\mu}(R_j) = \sum_{j=1}^{J} \mathbb F_{\mu}(\delta t_{jm}^{-1}) &\geq C_0^{-1} \sum_{j=1}^{J} \bigl( \delta t_{jm}^{-1} \bigr)^{d-s} \nonumber \\ 
&\geq 
C_0^{-1} \left(\delta{t_1}^{-1} \right)^{d-s} \sum_{j=1}^{J} \tau_2^{-(jm-1)(d-s)} \nonumber \\
&  \geq C_0^{-1} \left(\frac{\delta}{t_1}\right)^{d-s} \tau_2^{-(m-1)(d-s)} \sum_{j=1}^{J} \tau_2^{-m(d-s)(j-1)} \nonumber \\ 
&\geq C_0^{-1} \left(\frac{\delta}{t_1}\right)^{d-s}\tau_2^{-(m-1)(d-s)} \Bigl[ \frac{\tau_2^{-mJ(d-s)}-1}{\tau_2^{-m(d-s)}-1} \Bigr] =: \mathfrak{L}_0, \label{lhs} 
\end{align}  
where the second step in the display above follows from \eqref{lacunary-upper-lower}. On the other hand, 
\begin{equation} 
\mathfrak R := \mathbb F_{\mu}(\delta^{-2} R_J) \leq C_0 \bigl(\delta^{-2} R_J \bigr)^{d-s} \leq C_0 \bigl(\delta t_{mJ} \bigr)^{s-d} \leq C_0 \left(\frac{\tau_1^{1-mJ}}{t_1 \delta} \right)^{d-s} =: \mathfrak{R}_0, \label{rhs}  \end{equation}  
using \eqref{Strichartz-upper-lower} and \eqref{lacunary-upper-lower} again. Combining \eqref{lhs} and \eqref{rhs}, we note
\[  \frac{\mathfrak L}{\mathfrak R} \geq \frac{\mathfrak L_0}{\mathfrak R_0} = C_0^{-2} \delta^{2(d-s)} \tau_2^{-(m-1)(d-s)} \tau_1^{(mJ-1)(d-s)} \Bigl[ \frac{\tau_2^{-mJ(d-s)}-1}{\tau_2^{-m(d-s)}-1} \Bigr]. \] 
Thus $\mathfrak L_0/\mathfrak R_0$ is independent of $t_1$, but depends on $\tau_1, \tau_2, C_0$ (since $\delta, J$ depend on them) and $(d-s)$. Holding all the parameters except $s$ fixed, we find that
\[ \frac{\mathfrak L}{\mathfrak R} \geq \frac{\mathfrak L_0}{\mathfrak R_0} \longrightarrow \frac{J}{C_0^2} \quad \text{ as } s \nearrow d.\] 
The choice of $J_0$ and $J$ from \eqref{choice of J0} and \eqref{J} dictate that $J/C_0^2 > J_0/(mC_0^2) > 2$. Therefore, there exists a constant $\varepsilon_0 \in (0, \frac{d-1}{8})$, depending on $\tau_1, \tau_2, d$ and $C_0$ such that $\mathfrak L/\mathfrak R > \frac{1}{2}$ for $d-s < \varepsilon_0$. Unravelling the inequalities related to $\mathfrak L$ and $\mathfrak R$ from \eqref{lhs} and \eqref{rhs}, we reach the second inequality claimed in \eqref{Rj-condition}, with $M=4$. This choice of $\varepsilon_0$ appears in the statement of Theorem \ref{quasiregular-distance-thm}.  
\qed
\subsection{Proof of Proposition \ref{F-asymptotics-prop}} 
\noindent Let $\delta \gg t > 0$ be any two numbers such that 
\begin{equation} \label{delta-T}
\delta^{-2} > T_0, \quad \text{ and } \quad \frac{\delta}{t} > T_0. 
\end{equation} 
As in the proof of Proposition \ref{mainprop-1} in Section \ref{mainprop-1-section}, we decompose 
\begin{equation}  \label{Lambda-decomp-J}
||\Lambda_{\mu}(t) || = \int \bigl| \widehat{\mu}(\xi) \bigr|^2 \widehat{\sigma}(t \xi) \, d\xi =  \mathfrak J_1(t) + \mathfrak J_2(t) + \mathfrak J_3(t). 
\end{equation}    
Here $\mathfrak J_j(t)$ is an integral with the same integrand as $||\Lambda_{\mu}(t)||$ but evaluated on $\mathfrak G_j(t)$: 
\[\mathfrak G_1(t) := \bigl\{\xi \in \mathbb R^d : t|\xi| \leq \delta \bigr\}, \quad  \mathfrak G_2(t) := \bigl\{\xi : \delta \leq t|\xi| \leq \delta^{-1} \bigr\}, \quad  \mathfrak G_3(t) := \bigl\{\xi \in \mathbb R^d : t|\xi| \geq \delta^{-1} \bigr\}. \] 
The steps that led to the bound \eqref{I1-est} still apply, giving       
\begin{equation}  \label{J1-est} 
\mathfrak J_1(t) \geq (1 - \delta) \mathbb F_{\mu}(\delta t^{-1}) > 0,  
\end{equation}  
where the positivity follows from the fact that $\widehat{\mu}(0) = 1$. We estimate $\mathfrak J_3(t)$ using a partition of the domain $\mathfrak G_3(t)$ into dyadic annuli. Setting 
\[\mathfrak G_{3,k}(t) := \bigl\{\xi: 2^{k}\delta^{-1} \leq t|\xi| \leq 2^{k+1} \delta^{-1}\bigr\}, \] we obtain 
\begin{align}
\mathfrak J_3(t) &\leq \int_{\mathfrak G_3(t)} \bigl| \widehat{\mu}(\xi) \bigr|^2 \bigl(t |\xi| \bigr)^{-\frac{d-1}{2}} \, d\xi  \leq \sum_{k=0}^{\infty} \int_{\mathfrak G_{3,k}(t)}  \bigl| \widehat{\mu}(\xi) \bigr|^2 \bigl(t |\xi| \bigr)^{-\frac{d-1}{2}} \, d\xi \nonumber \\
&\leq \sum_{k=0}^{\infty} \bigl(2^k \delta^{-1}\bigr)^{-\frac{d-1}{2}} \int_{\mathfrak G_{3,k}(t)}  \bigl| \widehat{\mu}(\xi) \bigr|^2 \, d\xi \leq \sum_{k=0}^{\infty} \bigl(2^k \delta^{-1}\bigr)^{-\frac{d-1}{2}} \mathbb F_{\mu} \bigl(2^{k+1} \delta^{-1} t^{-1}\bigr) \nonumber \\ 
&\leq \sum_{k=0}^{\infty} \bigl(2^k \delta^{-1}\bigr)^{-\frac{d-1}{2}}  (2^{k+1} \delta^{-2})^c \mathbb F_{\mu} \bigl(\delta t^{-1}\bigr)
\leq 2^c \delta^{\frac{d-1}{2} - 2c} \mathbb F_{\mu} \bigl(\delta t^{-1}\bigr)  \sum_{k=0}^{\infty} 2^{k\bigl(-\frac{d-1}{2} + c\bigr)} \nonumber \\
&\leq  2^c \delta^{\frac{d-1}{2} - 2c} \mathbb F_{\mu} \bigl(\delta t^{-1}\bigr)  \sum_{k=0}^{\infty} 2^{-\frac{(d-1)k}{4}} \leq  2^c \delta^{\frac{d-1}{2} - 2c} \mathbb F_{\mu} \bigl(\delta t^{-1}\bigr)  \sum_{k=0}^{\infty} 2^{-\frac{k}{4}} \nonumber \\  &\leq 4^d \delta^{\frac{d-1}{2} - 2c} \mathbb F_{\mu} \bigl(\delta t^{-1}\bigr). \label{J3-est}
\end{align} 
In the third line in the displayed sequence above, we have used \eqref{controlled-growth} with $T_1 = 2^{k+1} \delta^{-2}$ and $T_2 = \delta t^{-1}$, the validity of this usage being ensured by the choice \eqref{delta-T} of $\delta$ and $t$. The assumption $c \in (0, \frac{d-1}{4})$ has been used in the subsequent step to ensure that the sum in $k$ is convergent, and bounded above by an absolute constant. This shows that the term $2^c \sum_{k\geq 0} 2^{-k/4}$ appearing in the penultimate step above is at most $4^{d}$. Combining \eqref{Lambda-decomp-J}, \eqref{J1-est} and \eqref{J3-est} yields 
\begin{equation} \label{Lambda-est-level-2}
\Lambda_{\mu}(t) \geq \bigl[1 - \delta - 4^d \delta^{\frac{d-1}{2} - 2c} \bigr]  \mathbb F_{\mu} \bigl(\delta t^{-1}\bigr) - |\mathfrak J_2(t)| \geq \Bigl(1 - \frac1M\Bigr)  \mathbb F_{\mu} \bigl(\delta t^{-1}\bigr) - |\mathfrak J_2(t)|,
\end{equation}  
where the last inequality follows from \eqref{how-small-is-delta}. The estimate in \eqref{Lambda-est-level-2} holds for all pairs $(\delta, t)$ satisfying \eqref{delta-T}; in particular, it holds for pairs $(\delta, t)$ with $\delta$ as in \eqref{how-small-is-delta} and $t = t_j$ with $t_j = \delta/R_j$ where $R_j$ obeys \eqref{Rj-condition}.
\vskip0.1in
\noindent Towards a contradiction, suppose if possible that 
\[ \Lambda_{\mu}(t_j) < M^{-1} \mathbb F_{\mu}(\delta t_j^{-1}) \quad \text{ for all $1 \leq j \leq J$.} \]
Substituting this into \eqref{Lambda-est-level-2} gives 
\[ \Bigl(1 - \frac{2}{M}\Bigr)  \mathbb F_{\mu}(\delta t_j^{-1})  \leq \bigl| \mathfrak J_2(t_j) \bigr| \quad \text{ for all $1 \leq j \leq J$.} \] Summing over all indices $j \in \{1, \ldots, J\}$, we arrive at
\begin{equation} \label{contra-stat} 
\Bigl(1 - \frac{2}{M} \Bigr) \sum_{j=1}^{J} \mathbb F_{\mu} \bigl(\delta t_j^{-1}\bigr) \leq \sum_{j=1}^{J} \bigl|\mathfrak J_2(t_j) \bigr| \leq  \sum_{j=1}^{J} \int_{\mathfrak G_{2}(t_j)} \bigl| \widehat{\mu}(\xi) \bigr|^{2} d\xi \leq \mathbb F_{\mu}(\delta^{-1} t_J^{-1}).  
\end{equation} 
The last step uses the disjointness of the annuli $\{\mathfrak G_2(t_j) : 1 \leq j \leq J \}$, as a result of which
\[ \bigsqcup_{j=1}^{J}  \mathfrak G_{2}(t_j) \subseteq \bigl\{\xi : |\xi| \leq \bigl(\delta t_J \bigr)^{-1}  \bigr\}. \]  This disjointness is a consequence of the  requirement $\delta^{-1} t_{j}^{-1} = \delta^{-2} R_j \leq R_{j+1} = \delta t_{j+1}^{-1}$, as per \eqref{Rj-condition}.  The relation \eqref{contra-stat} contradicts the second criterion in \eqref{Rj-condition}, proving the proposition. The positivity of the integral $\mathbb F_{\mu}(\delta t_j^{-1}) = \mathbb F_{\mu}(R_j)$ claimed in \eqref{F-asymptotics-prop-conclusion} follows from the fact that $R_j \geq T_0 \geq 1$, and $\xi \mapsto \widehat{\mu}(\xi)$ is continuous, with $\widehat{\mu}(0) = 1$. 
\qed
 \section{Appendix: A few proofs and counterexamples} \label{Appendix-1}
Certain facts were used without proof in earlier sections, most of them likely well-known to experts. This section provides more detailed explanation for the statements for which we could not find an easily citable reference. 
 \subsection{Sets of large Lebesgue measure have large distance sets} \label{Appendix-1-large-distance-sets}
\begin{proof}[Proofs of Theorems \ref{Boardman-Lemma} and \ref{Boardman-Corollary}]
Let us start with the proof of Theorem \ref{Boardman-Lemma}, which is based on the argument in \cite[Lemma II]{Boardman-1970}. Suppose $E, F \subseteq [0,1]^d$, $\min \{\lambda_d(E), \lambda_d(F) \} \geq 1 -\rho > \frac12$. Since $2(1-\rho) > 1$, we can choose a constant $c_{\rho} >0$ small enough so that $(1 + c_{\rho})^d < 2(1 - \rho)$. Then for any $x = (x_1, \ldots, x_d) \in [0, c_{\rho}]^d$, 
\begin{align*} &(x + F) \cup E \subseteq \prod_{j=1}^{d} \bigl[0, 1 + x_j \bigr] \subseteq \prod_{j=1}^{d} [0, 1+c_{\rho}] \; \text{ and hence } \\ 
&\lambda_d \bigl[(x + F) \cup E \bigr] \leq (1 + c_{\rho})^d < 2(1 - \rho) \leq \lambda_d(x + F) + \lambda_d(E).     \end{align*}
But this implies $\lambda_d \bigl[ (x + F) \cap E \bigr] > 0$, and hence $(x + F) \cap E$ is non-empty, i.e., $x \in E-F$.  Since this is true for all $x \in [0, c_{\rho}]^d$, we conclude that $E - F \supseteq [0, c_{\rho}]^d$. This in turn implies $\Delta(E, F) \supseteq \Delta ([0, c_{\rho}]^d) = [0, c_{\rho}\sqrt{d}]$. 
\vskip0.1in
\noindent Theorem \ref{Boardman-Corollary} follows from Theorem \ref{Boardman-Lemma}. Given a set $A \subseteq \mathbb R^d$ obeying \eqref{Boardman-condition}, one can find a sequence $R_n \nearrow \infty$ and $x_n \in \mathbb R^d$ so that 
\[ \frac{\lambda_d((A \cap Q(x_n;R_n))}{\lambda_d \bigl(Q(x_n;R_n) \bigr)} \geq  1 - \rho > \frac12 \quad \text{ for all $n$.} \]  
Let us denote by $\mathbb T_n$ the linear transformation that maps $Q(x_n; R_n)$ onto $[0,1]^d$: \[ \mathbb T_n: \mathbb R^d \rightarrow \mathbb R^d, \quad  \mathbb T_n(y) := \frac{y - x_n}{R_n} \quad \text{ and } \quad E_n := \mathbb T_n(A \cap Q(x_n;R_n)). \]   
 Then $\lambda_d(E_n) \geq 1 -\rho$. Invoking Theorem \ref{Boardman-Lemma}, we obtain $ \Delta(E_n) \supseteq [0, c_{\rho}\sqrt{d}]$, and hence
\[ \Delta(A) \supseteq \bigcup_{n=1}^{\infty} \Delta \bigl[\mathbb T_n^{-1}(E_n) \bigr] \supseteq \bigcup_{n=1}^{\infty} R_n \Delta(E_n)  \supseteq \bigcup_{n=1}^{\infty} \bigl[0, c_{\rho} \sqrt{d} R_n\bigr] = [0, \infty), \] 
as claimed. 
\end{proof} 

\subsection{Sets with all sufficiently large distances} \label{Appendix-1-Bourgain} 
\begin{lem} 
For $d \geq 2$, let $A \subseteq \mathbb R^d$ be a set obeying \eqref{Bourgain-liminf-condition}. Then $A$ contains all sufficiently large distances. 
\end{lem} 
\begin{proof} 
The proof is a straightforward consequence of the following proposition that appears in \cite{1986Bourgain}, which is also key to the proof of Theorem \ref{all-large-dist-thm}. We use this proposition in the following way. The assumption \eqref{Bourgain-liminf-condition} ensures the existence of $\kappa > 0$ and a sequence $R_n \nearrow \infty$ and $x_n \in \mathbb R^d$ with
\begin{align*} &\frac{\lambda_d(A \cap B(x_n;R_n))}{\lambda_d(B(x_n;R_n)} > \kappa; \;\text{ therefore } \lambda_d(E_n) > \kappa \text{ where } \\ 
&E_n := \mathbb T_{n} \bigl[A \cap B(x_n;R_n) \bigr] \subseteq B(0;1).   
\end{align*}
Here, as in Section \ref{Appendix-1-large-distance-sets}, $\mathbb T_{n} (y) = (y-x)/R_n$ is the linear transformation mapping $B(x_n;R_n)$ onto $B(0;1)$. Let $J=J(\kappa, d)$ be the integer provided by Proposition \ref{Bourgain-prop} corresponding to this value of $\kappa$. 
\vskip0.1in
\noindent Towards a contradiction, suppose if possible that $A$ avoids infinitely many large distances; i.e., there exists an infinite sequence $\mathtt d_j \nearrow \infty$ with $\Delta(A) \cap \bigl\{\mathtt d_j : j \geq 1 \bigr\} = \emptyset$. Without loss of generality, passing to a subsequence if necessary, we may assume that $\mathtt d_{j+1} > 2 \mathtt d_j$. Let us fix any sufficiently large choice of $n$ such that $R_n > \mathtt d_{J}$, and set $t_j = \mathtt d_{J + 1 -j}/R_n$. Then the set $E_n$ and the sequence $\{t_j : 1 \leq j \leq J \}$ obey the hypotheses, and therefore the conclusion, of Proposition \ref{Bourgain-prop}. In other words, $t_j \in \Delta(E_n)$ for some $j \in \{1, \ldots J\}$. But this implies $R_n t_j = \mathtt d_{J+1-j} \in \Delta(A)$, contradicting the assumption that none of the numbers $d_{j}$ lie in the distance set $\Delta(A)$.  
\end{proof} 
\begin{prop}\cite[Proposition 3 and Section 2]{1986Bourgain} \label{Bourgain-prop}
Let $d \geq 2$. Given any $\kappa > 0$, there exists a positive integer $J = J(\kappa, d)$ with the following property. 
\vskip0.1in
\noindent For every set $E \subseteq [0,1]^d$ with $\lambda_d(E) > \kappa$, and any sequence $\{t_j : j \geq 1\} \subseteq (0,1)$ with $t_{j+1} < t_j/2$, one can find $1 \leq j \leq J$ such that 
\[ \Lambda_{\mu}(t_j) > \frac{1}{2}. \]  
Here $\mu = \mathtt 1_E/\lambda_d(E)$ is the normalized Lebesgue measure on $E$ and $\Lambda_{\mu}(t)$ is the distance-identifying integral defined in \eqref{def-Lambda}. In particular, $t_j \in \Delta(E)$. 
\end{prop}

\subsection{Distance sets containing arbitrarily small intervals} \label{Appendix-3} 
A remark on page \pageref{small-rho-needed} states that Theorems \ref{Boardman-Lemma} and \ref{mainthm-2} can only be true for small values of $\rho$. We prove this assertion here.    
\vskip0.1in
\noindent Let $\varrho \in (0,1)$ and $N \in \mathbb N$. Setting $\mathbb Z_N := \{0, 1, \ldots, N-1\}$, we define a set
 \begin{equation} \label{def-example-E} 
 E = E(\varrho, N) := \Bigl\{\frac{p}{N} + \frac{\varrho}{N}[0, 1]^d : p \in \mathbb Z_N^{d} \Bigr\} \text{ so that } \lambda_d(E) = \varrho^d > 0. \end{equation}  
For fixed $\varrho$, let us consider the family of sets 
\[ \mathscr{E}_{\varrho} = \{E(\varrho, N) :  N \in \mathbb N, N \geq 2 \}, \] 
with $E = E(\varrho, N)$ defined as in \eqref{def-example-E}. Every $E \in \mathscr{E}_{\varrho}$ has the property that $\lambda_d(E) = \varrho^d > 0$, independent of $N$. On the other hand, 
 \begin{equation} E - E = \Bigl\{\frac{p-p'}{N} + \frac{\varrho}{N}[-1,1]^d : p, p' \in \mathbb Z_N^d \Bigr\}. \label{what-is-E-E} 
 \end{equation} 
If $\varrho < \frac{1}{2}$, \eqref{what-is-E-E} shows that $E -E$  is a disjoint union of cubes of sidelength $2\varrho/N$ centred at $\mathbb Z_N^d/N$. 
Therefore, the largest cube centred at the origin and contained in $E -E$ has sidelength $\varrho/N$, which can be made arbitrarily small by letting $N \rightarrow \infty$. This can be expressed as follows: for $\varrho < \frac{1}{2}$, 
\begin{align} 
&\mathtt c_{\varrho} := \inf \bigl\{ \mathtt c(E) : E \subseteq [0,1]^d, \lambda_d(E)  = \varrho^d  \bigr\}  \nonumber \\ &\quad \; = \inf \bigl\{\mathtt c(E) : E \in \mathscr{E}_{\varrho} \bigr\} = \inf\bigl\{\frac{\varrho}{N} : N \geq 1\bigr\} = 0, \text{ where } \label{c0} \\ 
&\mathtt c(E) = \sup \bigl\{c > 0 : E - E \supseteq [0, c]^d \bigr\}. \nonumber 
\end{align} 
 The relation \eqref{what-is-E-E} also shows that
 \begin{equation}   
 \Delta(E) \subseteq \bigcup \Bigl\{ \mathbb I_r : r \in \mathbb Z_N^d - \mathbb Z_N^d \Big\}, \quad \text{ where } \mathbb I_r := \frac{|r|}{N} + \Bigl[ -\frac{\varrho \sqrt{d}}{N}, \frac{\varrho \sqrt{d}}{N} \Bigr]. \nonumber
 \end{equation} 
Choosing $\varrho \sqrt{d} < \frac{1}{2}$, we observe that $\mathbb I_r$ contains the origin if and only if $r = 0$. Thus, the largest interval in $\Delta(E)$ containing the origin has length $\leq \varrho \sqrt{d}/N$, which also $\rightarrow 0$ as $N \rightarrow \infty$. This leads to  
\begin{align} 
&\mathtt a_{\varrho} := \inf \bigl\{ \mathtt a(E) : E \subseteq [0,1]^d, \lambda_d(E) = \varrho^d  \bigr\} \leq \inf \bigl\{ \mathtt a(E) : E = E(\varrho, N) \in \mathscr{E}_{\varrho}\bigr\} \nonumber \\ &\quad \; \leq \inf \bigl\{ \lambda_1(\mathbb I_0) = 2\frac{\varrho \sqrt{d}}{N} : N \geq 1\bigr\} = 0, \text{ where } \label{a0}\\ 
&\mathtt a(E) = \sup \bigl\{a > 0 : \Delta(E) \supseteq [0, a] \bigr\}. \nonumber 
\end{align} 
Combining \eqref{c0}, \eqref{a0}, we see that Theorem \ref{Boardman-Lemma} cannot hold if $\lambda_d(E), \lambda_d(F)$ are both small, specifically $< (2 \sqrt{d})^{-d}$. 

\subsection{Dyadic Hausdorff content of a uniformly distributed family of cubes} \label{Building-block-Hcontent-section}
In Section \ref{building-block-example-section}, we studied a set $\mathtt F \subseteq [0,1]^d$ given by \eqref{def-building-block-F}, consisting of small cubes arranged on a lattice grid. The dyadic Hausdorff content of $\mathtt F$ was claimed to be 1. We prove this statement in Lemma \ref{Building-block-Hcontent-lemma} of this section. 
\begin{lem} \label{cube-Hcontent-lemma}
For any $d \geq 1$, $\mathcal H_{\infty}^{\tau}([0,1]^d) = 1$ for all $\tau \in (0, d]$. 
\end{lem} 
\begin{proof} 
Let us recall the definition of dyadic Hausdorff content from \eqref{dyadic-H-content-def}. Choosing the trivial cover $Q = [0,1]^d$ with $\ell(Q) = 1$ shows that 
\[ \mathcal H_{\infty}^{\tau}([0,1]^d) \leq 1.  \]
For the converse inequality, let us fix any cover $\{Q_i: i \geq 1\}$ of $[0,1]^d$, where each $Q_i$ is a dyadic cube. Then,
\[ \sum_{i} \ell(Q_i)^d = \sum_i \lambda_d(Q_i) \geq \lambda_d([0,1]^d) = 1. \]
For any $0 < p \leq 1$, the well-known inequality $\sum_{i=1}^{\infty} |x_i| \leq \bigl[\sum_{i=1}^{\infty} |x_i|^p \bigr]^{\frac{1}{p}}$ implies that 
\[ 1 \leq \sum_{i} \ell(Q_i)^d \leq \Bigl[ \sum_{i} \ell(Q_i)^{dp} \Bigr]^{\frac1p} = \Bigl[ \sum_{i} \ell(Q_i)^{\tau} \Bigr]^{\frac{d}{\tau}} \text{ for } p = \frac{\tau}{d},  \]   
which completes the proof. 
\end{proof} 
\begin{cor}\label{dyadic-cube-Hcontent-corollary}
For $d \geq 1$ and any dyadic cube $Q \subseteq \mathbb R^d$, we have 
\[ \mathcal H_{\infty}^{\tau}(Q) = \ell(Q)^{\tau} \qquad \text{ for any } \tau \in (0, d]. \] 
\end{cor} 
\begin{proof} 
The proof is by scaling. Let $Q = c + r[0,1]^d$ be a dyadic cube with $\ell(Q) = r$, where $r = 2^N$ and $c = m2^N$, for some $m, N \in \mathbb Z$. On one hand, the definition \eqref{dyadic-H-content-def} implies that $\mathcal H_{\infty}^{\tau}(Q) \leq \ell(Q)^{\tau} = r^{\tau}$. On the other, given any dyadic cover $\{Q_i\}$ of $Q$, the cubes $\tilde{Q}_i = (Q_i-c)/r$ are also dyadic, and form a cover for $[0,1]^d$. It follows from Lemma \ref{cube-Hcontent-lemma} that 
\[ \sum_i \ell(\tilde{Q}_i)^{\tau} = r^{-\tau} \sum_i \ell(Q_i)\geq 1; \text{ in other words, } \sum_i \ell(Q_i)^s \geq r^{\tau} = \ell(Q)^{\tau}. \]  
This is the desired conclusion. 
\end{proof} 
\begin{lem} \label{Building-block-Hcontent-lemma} 
Let $\mathtt F$ be the set given by \eqref{def-building-block-F}, with 
\[ \sigma = \frac{d}{s} \in (1, \infty)\cap \mathbb Q, \quad  q = 2^{N_1} \text{ and }  \quad q^{\sigma} = 2^{N_2} \] 
for some choice of integers $N_1, N_2 \in \mathbb N$, $N_1 < N_2$. Then for every dyadic cube $Q \subseteq [0,1]^d$ with $\ell(Q) = 2^{-N_1+k}$, $0 \leq k \leq N_1$, an optimal cover of $\mathtt F \cap Q$ that realizes the infimum \eqref{dyadic-H-content-def} defining $\mathcal H_{\infty}^{s}(\mathtt F \cap Q)$ is given by the collection of basic cubes of $\mathtt F$ (of sidelength $q^{-\sigma}$) contained in $Q$; consequently,  
\begin{equation} \mathcal H_{\infty}^{s}(\mathtt F \cap Q) = 2^{kd} q^{-d} = \ell(Q)^{d}.   \label{Hcontent-induction} \end{equation} 
In particular, 
\begin{equation} \label{F-H-content-1} 
\mathcal H_{\infty}^{\tau}(\mathtt F) = 1 \text{ for all } \tau \leq s.
\end{equation} 
\end{lem} 
\begin{proof} 
Let us recall from \eqref{def-building-block-F} that $\mathtt F$ is a $q^d$-fold disjoint union of basic cubes $\mathscr{Q}_p$:   
\[ F = \bigsqcup_{p \in \mathbb Z_q^d} \mathscr{Q}_p \quad \text{ where } \quad \mathscr{Q}_p := \frac{p}{q} + q^{-\sigma}[0,1]^d. \] 
We will prove the first statement of the lemma, leading up to \eqref{Hcontent-induction}, by induction on $k$. To establish the base case $k = 0$, let us note that for any dyadic cube $Q$ with $\ell(Q) = 2^{-N_1} = 1/q$, the set $\mathtt F \cap Q$ consists of a single basic cube $\mathscr{Q}_p$. It therefore follows from Lemma \ref{cube-Hcontent-lemma} that 
\[ \mathcal H_{\infty}^{s}(\mathtt F \cap Q) = \mathcal H_{\infty}^{s}(\mathscr{Q}_p) = \ell(\mathscr{Q}_p)^{-\sigma} = q^{-\sigma s} = q^{-d}. \] 
An optimum cover of $\mathtt F \cap Q$ is the single cube $\mathscr{Q}_p$. 
\vskip0.1in
\noindent Let us continue to the induction step. Suppose that \eqref{Hcontent-induction} holds for $k \leq k_0-1$, and that for every dyadic $Q$ with $\ell(Q) = 2^{-N_1+k} <  2^{-N_1+k_0}$, an optimum cover of  $F \cap Q$ consists of the collection of dyadic cubes $\mathscr{Q}_p$ of side-length $q^{-\sigma}$ contained in $\mathtt F \cap Q$. Let $Q_0 \in \mathcal Q_d$ be a cube $\ell(Q_0) = 2^{-N_1 + k_0}$. Then $\mathtt F \cap Q_0$ can be written as   
\[ \mathtt F \cap Q_0 = \bigsqcup \Bigl\{\mathtt F \cap Q : Q \subsetneq Q_0, \; Q \in \mathcal Q_d, \; \ell(Q) = 2^{-N_1 + k_0 -1}\Bigr\}.\] 
Our goal is to determine a dyadic cover of $\mathtt F \cap Q_0$ that realizes the infimum in $\mathcal H_{\infty}^s(\mathtt F \cap Q_0)$. If $Q$ is a dyadic descendant of $Q_0$, the induction hypothesis dictates that an optimal cover of $F \cap Q$ consists of the basic cubes of $\mathtt F$ contained in $Q$. Thus, for the purpose of realizing the infimum \eqref{dyadic-H-content-def} defining $\mathcal H_{\infty}^s(\mathtt F \cap Q_0)$, it suffices to consider the following two covering candidates:  
\[ \mathcal Q_1 = \{Q_0\} \quad \text{ and } \quad \mathcal Q_2 = \Bigl\{\mathscr{Q}_p : \mathscr{Q}_p \subseteq \mathtt F \cap Q_0 \Bigr\}.   \]  
We compare the contribution towards \eqref{dyadic-H-content-def} from each cover: 
\begin{align*} 
\sum_{\mathtt q \in \mathcal Q_2} \ell(\mathtt q)^s &= \sum \bigl\{ \mathcal H_{\infty}^s(\mathtt F \cap Q) : Q \text{ children of }Q_0 \bigr\} = 2^d 2^{(k_0-1)d} q^{-d} = 2^{(-N_1 + k_0) d}   \\ 
  \text{ whereas } \sum_{\mathtt q \in \mathcal Q_1} \ell(\mathtt q)^s &= \ell(Q_0)^s = 2^{(-N_1 + k_0)s} \geq  2^{(-N_1 + k_0) d} = \sum_{\mathtt q \in \mathcal Q_2}\ell(\mathtt q)^s. 
\end{align*}  
The last inequality follows from the assumptions that $k \leq N_1$ and $s \leq d$. This means that $\mathcal H_{\infty}^{s}(\mathtt F \cap Q_0) = \ell(Q_0)^{d}$, with the optimal cover given by $\mathcal Q_2$, which consists of the basic cubes $\mathscr{Q}_p$ in $\mathtt F \cap Q_0$. This completes the induction step.  
\vskip0.1in
\noindent For $\tau=s$, the conclusion \eqref{F-H-content-1} follows from \eqref{Hcontent-induction} with $k = N_1$; in this case, $Q = [0,1]^d$ and $\mathtt F \cap Q = \mathtt F$. For $\tau \leq s$, we observe that 
\[ \ell(Q_i)^s \leq \ell(Q_i)^{\tau} \text{ since $\ell(Q_i) \leq 1$;} \quad \text{ hence } 1 \leq \mathcal H_{\infty}^{s}(\mathtt F) \leq \mathcal H_{\infty}^{\tau}(\mathtt F) \leq 1, \]
completing the proof of the claim \eqref{F-H-content-1}.    
\end{proof} 

\section{Methodological Limitations and Open Problems} \label{Open Problems Section}
This section discusses the limitations of the proof techniques developed in this paper and highlights several open problems that arise naturally from them. 
\subsection{Sharp dimension and density thresholds}
\subsubsection{Overlap of distance sets for subsets of the unit cube with large Lebesgue measure} 
 A current gap in the distance set literature is a higher dimensional version of Boardman's Theorem \ref{Boardman-Lemma}, that is, a fully uniform Steinhaus theorem in $\mathbb R^d$. Specifically, 
 \begin{itemize} 
 \vskip0.1in
 \item For $d \geq 2$, what is 
 \begin{align*} 
& \mathtt a_{\alpha, \beta} := \inf \left\{ \mathtt a(E, F) : E, F \subseteq [0,1]^d, \; \lambda_d(E) = \alpha, \; \lambda_d(F) = \beta \right\}, \text{ where } \\
& \mathtt a(E, F)  = \sup \left\{ a>0 : \Delta(E, F) \supseteq [0, a) \right\} 
 \end{align*}  
measures the minimal interval contained in $\Delta(E, F)$? In particular, what are the best constants $\rho$ and $c_{\rho}$ for which Theorem \ref{Boardman-Lemma} holds? 
 \vskip0.1in
 \end{itemize} 
\noindent The proof of Theorem \ref{Boardman-Lemma} presented in Section \ref{Appendix-1-large-distance-sets} shows that $\rho < \frac{1}{2}$ and provides an upper bound for $c_{\rho}$ via the inequality $(1+c_{\rho})^d < 2(1-\rho)$; however, neither estimate is known to be optimal. 
\subsubsection{Fractal variants of a uniform Steinhaus-like theorem}
Extending the above uniform questions to the sparse setting leads to analogous problems when the Hausdorff dimension of $E \subseteq \mathbb R^d$ is less than the ambient dimension $d$. For instance, 
\begin{itemize} 
\vskip0.1in
\item What are the sharp constants $\bar{a}_d, \bar{b}_{d}, \rho_d, \bar{\varepsilon}_d$ for which \eqref{mainthm-2-conclusion} in Theorem \ref{mainthm-2} \eqref{mainthm2-parta} holds?
\vskip0.1in 
\item What is the optimal value of $\varepsilon_{\rho}$ for which a structural result for distance sets such as Theorem \ref{mainthm-cor} holds? Note that $\varepsilon_{\rho}$ would be the maximum allowable drop in dimension for $\Delta(E)$ to contain a union of intervals of specified geometry. 
\vskip0.1in 
\item To expand further, could the exponent $\varepsilon_{\rho} =  (d-1)/2$ be optimal (uniformly in $\rho$) for Theorem \ref{mainthm-cor}, consistent with the dimensional threshold of Mattila-Sj$\ddot{\text{o}}$lin Theorem \ref{mainthm-1}? If not, what are other contributing factors, beyond the high density criterion in \eqref{high-density-on-cube}, for generating intervals in distance sets? 
\vskip0.1in 
\end{itemize} 
In the current methodology, the main technical obstruction to obtaining better quantitative bounds on $\varepsilon_{\rho}$ occurs in the construction of the measure $\mu$ obeying the finite energy and spectral gap conditions. Specifically, in  Section \ref{measure construction section} this measure is constructed using very fine cubes in $\mathscr{D}_T$, as can be seen from \eqref{def-mu-nu}. The largeness of $T$ inherently limits the value of $\varepsilon_{\rho}$ (see \eqref{what-is-epsilon}). Alternative constructions avoiding this dependence could potentially sharpen the exponent. 
\subsection{Multi-point configurations}  
\subsubsection{Sparse unbounded sets with all sufficiently large copies of a $k$-point configuration} \label{section: all large copies}
A distance is determined by just two points; in this sense, a distance set embodies one of the simplest types of configurations in a set. The analysis in this paper prompts similar questions about configurations involving more than two points. 
\vskip0.1in
\noindent For instance, given a $k$-point configuration $V= \{v_1, \ldots, v_{k}\} \subseteq \mathbb R^d$ and $A \subseteq \mathbb R^d$, let  
\begin{equation} \label{defn: Deltak} 
\Delta_k(A, V) := \bigl\{t >0 : \exists x \in A {\text{ and a rotation $\mathcal O_d$ such that }} x + t \mathcal O_d(V) \subseteq A  \bigr\}. 
\end{equation}  
As a consistency check, the set $\Delta_k(A, V)$ for $k=1$ and $V = \{e_1\}$ recovers the distance set of $A$. Assuming that $d \geq 2$, $k < d$, $V$ is linearly independent, and $A \subseteq \mathbb R^d$ has positive upper density, Bourgain \cite{1986Bourgain} showed that $\Delta_k(A, V)$ contains a half-line, i.e., $A$ contains all sufficiently large affine copies of $V$. This leads to a similar problem in the sparse setting generalizing Theorems \ref{mainthm-3} and \ref{mainthm-4}: 
\vskip0.1in
\begin{itemize} 
\item Under what conditions on $A \subseteq \mathbb R^d$ of less than full Hausdorff dimension does $\Delta_k(A, V)$ contain a half-line $[M, \infty)$ for some $M >0$?
\end{itemize}
\vskip0.1in
Constructing an appropriate configuration integral, i.e., an analogue for Proposition \ref{Lambda-prop} for a $k$-point pattern $V$ is a technical barrier to extending the current method. 
\subsubsection{Versions of Bourgain's local Szemer\'edi-type theorem in the fractal setting} \label{section: Bourgain local Szemeredi} 
Theorem \ref{F-asymptotics-thm} isolates Fourier $L^2$ asymptotics $\mathbb F_{\mu}(T)$ as an identifier of the size of distance sets, leading to local structural results for distance sets when the original set is geometrically regular (Theorem \ref{quasiregular-distance-thm}). A natural challenge is to determine whether similar Fourier-analytic conditions govern $\Delta_k(E, V)$ as in \eqref{defn: Deltak} for $k$-point configurations $V$. 
\vskip0.1in
\begin{itemize} 
\item Is there an analogue of \eqref{controlled-growth} or \eqref{Rj-condition} that ensures the presence of copies of $V$ on many lacunary scales? 
\vskip0.1in
\item Suppose as in Theorem \ref{quasiregular-distance-thm} that $E \subseteq [0,1]^d$ is a quasi-regular set equipped with a locally uniformly $s$-dimensional probability measure obeying \eqref{Strichartz-upper-lower}. Does the set of scales $\Delta_k(A, V)$ intersect every sufficiently long  $(\tau_1, \tau_2)$-lacunary sequence? 
\vskip0.1in 
\end{itemize} 
Progress on the question in Section \ref{section: all large copies} regarding an appropriate configuration integral may be essential to shed light on this local problem. This remains an intriguing and wide-open direction.

\Addresses 

\end{document}